\documentclass[10pt]{amsart}
\usepackage{fontenc}
\usepackage[english]{babel}

\usepackage{amssymb, amsmath,mathtools}
\usepackage{mathrsfs}%stix}
\usepackage{amscd}
\usepackage{verbatim}
\usepackage{stmaryrd}
\usepackage[dvipsnames]{xcolor}
\usepackage{tikz-cd} 
\usepackage{pst-node}

\usepackage{enumerate}

\usepackage[colorlinks,linkcolor={blue},citecolor={blue},urlcolor={purple},]{hyperref}

\usepackage[colorinlistoftodos,prependcaption,textsize=tiny]{todonotes}

\usepackage{tabularx}
\newcolumntype{L}{>{\arraybackslash}X}
\usepackage{multirow}

\theoremstyle{plain}
\newtheorem{theorem}{Theorem}[section]
\theoremstyle{remark}
\newtheorem{remark}[theorem]{Remark}
\newtheorem{example}[theorem]{Example}
\theoremstyle{plain}
\newtheorem{corollary}[theorem]{Corollary}
\newtheorem{lemma}[theorem]{Lemma}
\newtheorem{proposition}[theorem]{Proposition}

\newtheorem{assumption}[theorem]{Assumption}

\numberwithin{equation}{section}

\newcommand{\C}{\mathbb{C}}
\newcommand{\R}{\mathbb{R}}

\newcommand{\N}{\mathbb{N}}
\newcommand{\Dom}{\mathcal{O}}

\renewcommand{\div}{\text{div}}

\newcommand{\E}{\mathbb{E}}
\renewcommand{\P}{\mathbb{P}}
\renewcommand{\O}{\Omega}
\newcommand{\Do}{\mathsf{D}}
\newcommand{\Dd}{\dot{\mathsf{D}}}
\newcommand{\Ran}{\mathsf{R}}
\newcommand{\om}{\omega}
\newcommand{\calL}{\mathscr{L}}
\newcommand{\Z}{\mathbb{Z}}
\newcommand{\Rsec}{\mathcal{R}}
\newcommand{\angH}{\om_{H^{\infty}}}
\newcommand{\angR}{\om_{\mathcal{R}}}
\newcommand{\x}{\mathbf{x}}
\newcommand{\A}{\mathcal{A}}
\newcommand{\B}{\mathcal{B}}
\def\id{\text{\usefont{U}{bbold}{m}{n}1}}
\newcommand{\M}{\mathcal{M}}
\newcommand{\e}{\mathsf{e}}
\newcommand{\res}{\mathsf{r}}
\newcommand{\wh}{\widehat}
\newcommand{\non}{\mathcal{N}}
\newcommand{\embed}{\hookrightarrow}

\renewcommand{\ker}{\mathsf{N}}

\newcommand{\weak}{\mathsf{w}}
\newcommand{\KS}{\mathsf{KS}}
\newcommand{\EL}{\mathsf{EL}}
\newcommand{\EB}{\mathsf{EB}}
\newcommand{\AS}{\mathsf{AS}}
\newcommand{\SD}{\mathsf{SD}}

\renewcommand{\l}{\langle}
\renewcommand{\r}{\rangle}
\newcommand{\diag}{\mathcal{D}}
\newcommand{\strong}{\mathsf{s}}
\newcommand{\dd}{{\rm d}}
\renewcommand{\emptyset}{\varnothing}

\newcommand{\thesp}{\gamma}

\DeclareMathOperator{\tr}{tr}

\providecommand{\norm}[1]{\lVert#1\rVert} %Norm
\newcommand{\Ssec}{\mathcal{S}}

\usepackage{cite}

\author{Antonio Agresti}
\address{Department of Mathematics,
	RPTU Kaiserslautern-Landau, Paul-Ehrlich-Stra{\ss}e 31,
	67663 Kaiserslautern, Germany}
\email{antonio.agresti92@gmail.com}
\curraddr{Institute of Science and Technology Austria (ISTA), Am Campus 1, 3400 Klosterneuburg, Austria}

\author{Amru Hussein}
\address{Department of Mathematics,
	RPTU Kaiserslautern-Landau, Paul-Ehrlich-Stra{\ss}e 31,
	67663 Kaiserslautern, Germany}
\email{hussein@mathematik.uni-kl.de}


\thanks{The first author has  been supported partially by the Nachwuchsring – Network for the promotion of young scientists – at TU Kaiserslautern. Both authors have been  supported by MathApp – Mathematics Applied to Real-World Problems - part of the Research Initiative of the Federal State of Rhineland-Palatinate, Germany}

\begin{document}
\title[$H^{\infty}$-calculus  for block operator matrices]{Maximal $L^p$-regularity and $H^{\infty}$-calculus   for block operator matrices and applications}

\begin{abstract}
Many coupled evolution equations can be described via
$2\times2$-block operator matrices of the form
$\A=\begin{bmatrix} A & B \\ C & D \end{bmatrix}$
in a product space $X=X_1\times X_2$ with possibly unbounded entries. 
Here, the case of diagonally dominant block operator matrices is considered, that is, the case where the full operator $\A$ can be seen as a relatively bounded perturbation of its diagonal part with $\Do(\A)=\Do(A)\times \Do(D)$ though with possibly large relative bound.
For such operators the properties of sectoriality, $\Rsec$-sectoriality and the boundedness of the $H^\infty$-calculus are studied, and for these properties perturbation results for possibly large but structured perturbations are derived. 
Thereby, the time dependent parabolic problem associated with $\A$ can be analyzed in maximal $L^p_t$-regularity spaces, and this is applied to a wide range of problems such as different theories for liquid crystals, an artificial Stokes system, strongly damped wave and plate equations, and a Keller-Segel model.
\end{abstract}

\maketitle
\addtocontents{toc}{\protect\setcounter{tocdepth}{1}}
\tableofcontents

\section{Introduction}
In this article the abstract Cauchy problem 
\begin{equation}\label{eq:CP_intro}
\left\{
\begin{aligned}
\partial_t x(t) +\A\, x(t) &=f(t), \qquad t>0, \\ 
 x(0)&=x_0,
\end{aligned}
\right.
\end{equation}
is studied for a block operator matrix
\begin{align*}
\A=\begin{bmatrix} A & B \\ C & D \end{bmatrix}
\end{align*}
in the product space $X=X_1\times X_2$, where
$X_1,X_2$ are Banach spaces, and 
\begin{align*}
A&\colon \Do(A)\subseteq X_1\to X_1, \quad \text{ and }\quad 
D\colon \Do(D)\subseteq X_2\to X_2,	\\
B&\colon \Do(B)\subseteq X_2\to X_1, \quad \text{ and }\quad 
C\colon \Do(C)\subseteq X_1\to X_2
\end{align*}
are possibly unbounded linear operators with domains $\Do(A)$, $\Do(D)$, $\Do(B)$, and $\Do(C)$, respectively. 
Such block operator matrices arise in a wide range of coupled evolution equations including mixed-order systems. In Section~\ref{sec:applications}  liquid crystal
 models, an artificial Stokes system, strongly damped wave and plate equations, and a Keller-Segel model are discussed as applications of the general theory presented here.
The recurrent theme and main question of this article is to ask under which conditions operator theoretical properties of the diagonal operators can be transferred to the full operator. Thus, the starting point is to assume that the uncoupled problem associated with the diagonal operator is well-understood   and to interpret the coupled equations as off-diagonal perturbation, that is, 
\begin{align}\label{eq:ABCD}
\A= \diag + \B \quad \hbox{with} \quad \diag=\begin{bmatrix} A & 0 \\ 0 & D \end{bmatrix} \quad \hbox{and} \quad \B= \begin{bmatrix} 0 & B \\ C & 0 \end{bmatrix}.
\end{align}
This viewpoint is captured by the assumption of the \emph{diagonal dominance} of $\A$ which requires that $\Do(\A)=\Do(\diag)$ and that $\B$ is relatively bounded with respect to $\diag$ though with possibly large relative bound, see Section~\ref{sec:blockoperator} below for the precise setting. 
For such diagonally dominant block operator matrices $\A$ the properties of sectoriality, $\Rsec$-sectoriality, and the boundedness of the $H^\infty$-calculus are studied assuming that the respective properties hold for the diagonal operator $\diag$. 

These operator theoretical properties are closely related to the solution theory of the time-dependent problem \eqref{eq:CP_intro}. Namely, $\Rsec$-sectoriality implies in UMD-spaces $X$  maximal $L^p_t$-regularity for $p\in (1,\infty)$, that is, there exists a constant $C>0$ such that if in \eqref{eq:CP_intro} the right hand side satisfies $f\in L^p(0,\infty;X)$ and $x_0=0$, then there is a unique  solution  $x$ to \eqref{eq:CP_intro} satisfying the maximal regularity estimate
\begin{align*}
\norm{\partial_t x}_{L^p(0,\infty;X)} + \norm{\A\, x}_{L^p(0,\infty;X)} \leq C \norm{f}_{L^p(0,\infty;X)},
\end{align*}
see e.g. \cite{We, DHP, pruss2016moving, KuWe} for overviews on this subject. 
The corresponding notion of stochastic maximal regularity is not equivalent to the deterministic notion, but it is implied for instance by the boundedness of the  $H^\infty$-calculus, see e.g. \cite{AV19, MaximalLpregularity} and the references given therein for details on this theory. Moreover, sectoriality, $\Rsec$-sectoriality, and the boundedness of the $H^\infty$-calculus give information on the fractional powers of $\A$. Also, 
the domains and ranges of the fractional powers of $\A$
induce scales of extrapolation spaces and thereon consistent families of operators, compare e.g. \cite{Haase:2}.
Amongst others, this is helpful in the analysis of many quasi- or semi-linear problems, compare e.g. \cite{ AV19_QSEE_1, AV19_QSEE_2, NVW11eq, PruSim04, PrussWeight1, PrussWeight2, pruss2016moving} just to give a small sample of the literature in this direction.

The classical perturbation theorems for  sectoriality, $\Rsec$-sectoriality, and the boundedness of the $H^\infty$-calculus deal with smallness conditions often expressed in form of small relative bounds, relative compactness or lower order perturbations, compare e.g. \cite{DDHPV, DHP, KuWePert, AHS1994}.  The aim here is to complement these classical results for general additive perturbations by perturbation theorems for off-diagonal perturbations in the diagonally dominant case. In this situation it turns out that the smallness or lower order perturbation conditions can be imposed on objects describing the coupling rather than on the full additive perturbation. Moreover, assuming that off-diagonal perturbations preserve a certain structure by perturbing ``in the right direction'', one  can omit even any kind of smallness conditions, see Section~\ref{sec:sec_Rsec} below.  Also, many assumptions imposed on the general setting can be weakened for the case of off-diagonal perturbations. For instance, for
the classical perturbation results for sectoriality and $\mathcal{R}$-sectoriality
on relative bounded perturbations with small relative bound, there is no straightforward counterpart for the bounded
 $H^\infty$-calculus.
 The available perturbation results for the $H^\infty$-calculus require further assumptions in addition to the smallness conditions.  
  Here we have been able to show a result of this type although only for the case of some off-diagonal perturbations behaving well in some extrapolation scales induced by $\diag$, compare Section~\ref{sec:Hinfty} below.

The approach developed here is based in spirit on a combination of the theory by Kalton, Kunstmann and Weis relating $\Rsec$-sectoriality and the boundedness of the $H^\infty$-calculus, see \cite{KKW, KW13_errata}, with concepts for diagonally dominant block operator matrices pioneered by Nagel in \cite{Nag89} for $C_0$-semigroups.   This synthesis opens a new perspective on coupled systems in $L^p$-spaces and is illustrated in the subsequent  example and also by a number of applications in Section~\ref{sec:applications}.

\subsection{An example}\label{subsec:example}
As illustration one can consider the following model case  inspired by the linearization of the
Beris-Edwards model for liquid crystals.
For $p\in (1,\infty)$ let 
$X= H^{-1,p}(\R^d)\times L^{p}(\R^d)^d$ and%\todo{Check changes. It is important to have $H^{2,p}(\R^d;\R^d)$ to have the divergence operator in $B$.}
\begin{equation*}
\A
\stackrel{{\rm def}}{=}
\begin{bmatrix}
\id-\Delta & \div(\id- \Delta) \\
\nabla & \id-\Delta
\end{bmatrix} \ \
\text{ with } \ \
\Do(\A)= H^{1,p}(\R^d)\times H^{2,p}(\R^d)^d.
\end{equation*}
The actual linearization of  the
Beris-Edwards model for liquid crystals -- where the first component describes the fluid-like behavior and the second component the orientation of the crystal-like rods -- has a similar mixed order  structure. 
However, this simplified Beris-Edwards-type model  allows for more direct computations, and therefore it seems better suited as illustrative example, see  Subsection~\ref{ss:Beris_Edwards} and \ref{subsec:BerisEdwardsType} for a  discussion of both models.   Here, one directly observes that $\A$ is diagonally dominant, but classical perturbation results are not applicable since the off-diagonal part is neither small nor of lower order. However, this case can be treated within the theory presented here by studying the coupling of the diagonal and the off-diagonal part as will be explained in the subsequent Subsection~\ref{subsec:rsec_intro}.

Special types of diagonally block operator matrices which appear in many applications are of the form
\begin{align*}
\A= \begin{bmatrix}
A & B  \\ 0 & D
\end{bmatrix} \quad \hbox{or} \quad \A=\begin{bmatrix}
0 & B  \\ C & D
\end{bmatrix}.
\end{align*}       
In these situations results on $\Rsec$-sectoriality and boundedness of the $H^\infty$-calculus are obtained in Subsection~\ref{s:special_cases_dominant}. Concrete examples such as the simplified Erickson-Leslie model for liquid crystals, the classical Keller-Segel system, the artificial Stokes system and second order problems with strong damping are discussed in Subsections~\ref{subsec:simplified_LCD} -- \ref{subsec:2ndOrder}.

\subsection{Sectoriality and $\Rsec$-sectoriality for  
	block operator matrices}\label{subsec:rsec_intro}
A key tool in our analysis of sectoriality and $\Rsec$-sectoriality is the factorization for diagonally dominant operators 
\begin{equation*}%
\lambda-\A= \M(\lambda)(\lambda-\diag) \quad \hbox{with} \quad 
\M(\lambda)=\begin{bmatrix}
\id & -B(\lambda-D)^{-1}  \\
-C(\lambda-A)^{-1} & \id 
\end{bmatrix},
\end{equation*}
where the inverse  of $\M(\lambda)$ can be described by the inverse of one of the operators
\begin{align*}
M_1(\lambda) &
\stackrel{{\rm def}}{=}
\id - B (\lambda-D)^{-1} C(\lambda- A)^{-1}  \quad \hbox{for } \quad \lambda \in \rho(A)\cap \rho(D),\\
M_2(\lambda) &
\stackrel{{\rm def}}{=}
\id - C (\lambda-A)^{-1} B(\lambda- D)^{-1} \quad \hbox{for } \quad \lambda \in \rho(A)\cap \rho(D),
\end{align*}
which encode the coupling of the off-diagonal and the diagonal part. 
This factorization has been studied already in the work \cite{Nag89} by Nagel in the context of $C_0$-semigroups, and it is discussed in Section~\ref{sec:blockoperator}.
Thus sectoriality and $\Rsec$-sectoriality can be traced back to the corresponding estimates on one of the two operator families $M_1(\cdot)$ and $M_2(\cdot)$ -- without smallness conditions -- assuming that the diagonal operator has these properties, compare Theorem~\ref{t:rsec_necessary_sufficient_condition_II} and the other statements in Section~\ref{sec:sec_Rsec}. 
For the case of the example discussed above in Subsection~\ref{subsec:example} one has for instance  %{\color{red} for $\lambda\not \in \R_+$,}
\begin{align*}
M_{1}(\lambda)
&=
\id - \div(\id- \Delta) (\lambda-\id+\Delta)^{-1} \nabla (\lambda-\id +\Delta)^{-1}\\
&=\id-\Delta (\id-\Delta) (\lambda-\id+\Delta)^{-2}  \quad \hbox{for } \quad \lambda \in \C\setminus[0,\infty)
\end{align*}
which can be analyzed explicitly.

Moreover, by this factorization one can derive perturbation results  which do not impose  the classical smallness conditions on $\B$ as compared to $\diag$, see for instance \cite{KuWePert} for perturbation results with such  smallness conditions. Instead -- taking into account the block structure -- the smallness is assumed on the coupling expressed by one of the  operators $M_j(\lambda)$, $j\in\{1,2\}$, see Corollaries~\ref{t:rsec_necessary_sufficient_condition} and \ref{cor:pert_block_operators}, and Proposition~\ref{prop:pert_block_operators}.
Here, as in the classical case, behind the smallness assumption lurks the Neumann series.

\subsection{Bounded $H^\infty$-calculus for block operator matrices}
The boundedness of the $H^\infty$-calculus  implies $\Rsec$-sectoriality on Banach space $X$ with the relatively weak property $(\Delta)$ (or triangular contraction property),  which holds in particular for UMD spaces, cf. \cite{KWcalc} or \cite[Theorem 10.3.4(2)]{Analysis2},   and for general $X$ almost $\Rsec$-sectoriality is implied, see  \cite[Proposition 3.2]{KKW}, but the converse is in general not true. This has been shown for instance by the example constructed by McIntosh and Yagi in \cite{MY_1990} where block operator matrices make a \emph{cameo} since the example relies on block triangular operator matrices, the spectrum of which is determined by the diagonal part, however its norm is heavily influenced by the off-diagonal part, cf. also \cite{Fackler_Counterexamples} for another counterexample. The question how close $\Rsec$-sectorial operators are to having a bounded $H^\infty$-calculus has been addressed in detail for many situations  by Kunstmann, Kalton and Weis in  \cite{KKW, KW13_errata}. Here, we translate their results to the case of block operator matrices, and we obtain two types of results. First, adding conditions  on the orders and on the relations  of the blocks $A,B,C,D$ encoded in terms of fractional powers, one can show that $\Rsec$-sectoriality implies boundedness of the $H^\infty$-calculus. This is discussed in Section~\ref{sec:Hinfty}, where also results for small couplings are presented. 
Second, considering interpolation scales and consistent families of operators thereon, results for certain points in the scale carry over to the full scale or at least to a part of it, compare Section~\ref{sec:extrapolation}. In particular the case of Hilbert spaces is considerably easier to access, and under certain conditions this carries over to scales of Banach spaces. In fact the interpolation results for the $H^\infty$-calculus on different scales from \cite{KKW} play a key role in the proofs in both Section~\ref{sec:Hinfty} and \ref{sec:extrapolation}.
Note that in general $\Rsec$-sectoriality or boundedness of the $H^\infty$-calculus does \emph{not} extrapolate as shown by Fackler in \cite{F14_extrapolation_MR}, however, for the situations considered here it does.

In the literature there is a large body of works on perturbation theory for the boundedness of the $H^\infty$-calculus. Many of these treat  additive perturbations $\A=\A_0+\B$, compare for instance \cite{Haase:2, KuWe, DDHPV}.  
A special case is the situation where the perturbative term admits a factorization $\B=TS$ for some operators $T$ and $S$ acting on extrapolation scales induced by $\A_0$, see \cite{HaKu2006, Amansag_etal2022}. For the particular situation of block operator matrices \eqref{eq:ABCD}, it seems that there is no application of this
%for such factorizations 
except when the blocks $B$ and $C$ already satisfy such factorizations.

%Literature History: 
\subsection{Literature on block operator matrices and coupled parabolic systems}
There is a great interest and an extensive literature related to block operator matrices and their applications in the context of spectral and semigroup theory, see e.g. the monographs by Tretter   \cite{Tretter} and Jeribi \cite{Jeribi}. This subject is also related to the study of mixed order systems, see e.g. the book by Denk and Kaip \cite{DenkKaip} and the references therein, the elliptic case has been discussed e.g. also in the classical article by Douglis, Agmon and Nirenberg \cite{ADN_1964}.
It seems however, to the best of our knowledge, that so far, there has been no study on $\Rsec$-sectoriality and the boundedness of the $H^\infty$-calculus for a general class of diagonally dominant block operator matrices. 

The theory presented here has an overlap with the study of mixed order systems as for instance the example given above in Subsection~\ref{subsec:example} can also be treated as a parabolic mixed order system.
The maximal regularity results for mixed order systems as discussed in \cite{DenkKaip} rely on Fourier multiplier techniques. %and are mostly restricted to the full space case. 
The approach presented here is however more operator theoretical in spirit. This is of particular advantage  when considering scales of spaces and operators thereon including weak settings, and such situations appear in many applications, see e.g. \cite{CriticalQuasilinear}  for the scalar case.  
Also, the approach presented here evades in some situations laborious localization procedures. 
In comparison, the Newton polygon method presented in \cite{DenkKaip}  allows one to treat not only non-homogeneous orders in space but even in space and time, while the approach given here is restricted to problems of first  order in time.

Also, there is an extensive literature on maximal $L^p_t$-regularity for particular block operators of mixed order
arising from parabolic plate and wave equations with damping.  
In these areas maximal $L^p_t$-regularity and the  boundedness of the $H^\infty$-calculus are usually proven via multiplier and localization methods rather  than by general operator theoretical considerations, compare e.g. \cite{Racke_2009, Shib_2019, Sum_2018, Shib_2017, Nau_2014, Hum_2019} and the references therein which is only a small selection of the literature in this direction. These results cover a wide range of dampings and couplings while the diagonally dominant case discussed here occurs only for certain strong dampings.
Note that $\Rsec$-sectoriality is discussed also for other very particular
block operators as for instance in \cite{Barta_2008}.

In the context of semigroup theory,
one of the starting points for the systematic study of diagonally dominant block operator matrices and evolution equations  is the work by
Nagel \cite{Nag89}. This has been extended by Engel and Nagel in \cite{Nag97, Engel_1995},  respectively, and subsequent works.  Early works in this direction are \cite{Zab_1978, Zab_1980, Nag_1985, Eng_1986} with focus on $C_0$-semigroups generated by block operator matrices. 
Along these lines, the case of Hilbert spaces and sesquilinear forms with a block structure have been investigated in  \cite{Mug_2008} and
\cite{Liu_2020}, the question of 
m-sectorial block operator matrices is discussed in \cite{Arl_2002},
second order Cauchy problems for block triangular generators are studied in \cite{Mug_2006}, and special classes of block operator matrices are treated in \cite{Adler_2018}, see also the references in these works.

From the viewpoint of spectral theory different questions have been addressed, see
\cite{Tretter} for an overview, and  works in this direction deal with the essential spectrum  \cite{Bat2005, Batkai_2012, Char_2014, Ibro_2016, Ibro_2017, Ibro_2017_2, Ibro_2013},
adjoints of block operator matrices \cite{Moe_2008}, closedness and self-adjointness \cite{Shk_2016},  invertibility of block operator matrices \cite{Hua_2019}, the 
quadratic numerical range \cite{Lan_2001}, or the operator Ricatti equation \cite{Kost_2005}, where the given references are just samples and far from being complete.
In most cases $2\times 2$-block operator matrices are discussed where  $2\times 2$-matrices serve as an inspiration. For many applications this is sufficient, and some larger block operator matrices can be treated iteratively by different $2\times 2$-block decompositions. 
However, also larger block structures such as 
$3\times 3$-  or general $n\times n$-block operator matrices have been studied as well, see \cite{Ben_2010} and also \cite[Section 1.11]{Tretter}, respectively. 
Spectral problems for block operator matrices and corresponding sesquilinear forms are discussed   in many works, just to mention a few see   \cite{Kost2013_1, Kost2013_2, Kost_2005, Kost_2007, Kost_2004, Seel_2016, Schmitz_2015, Seel_2016_2} and also the references therein.

We start by recapitulating some basic notions and facts in Section~\ref{sec:prelim}, and in the subsequent Section~\ref{sec:blockoperator} the setting for diagonally dominant block operator matrices is made precise. The main results are contained in Sections~\ref{sec:sec_Rsec} -- \ref{sec:extrapolation} and applications are discussed in Section~\ref{sec:applications}.

\section{Sectoriality, $\Rsec$-sectoriality and the bounded $H^\infty$-calculus}\label{sec:prelim}
\subsection{Notation}\label{subsec:notation}
The solution theory of the abstract Cauchy problem 
\begin{equation}\label{eq:CauchyProb}
\left\{
\begin{aligned}
\partial_t x(t)  +T\, x(t) &=f(t), \qquad t>0, \\ 
 x(0)&=x_0,
\end{aligned}
\right .
\end{equation}
for a linear unbounded operator $T$ in a Banach space $X$ over $\C$ is closely related to the  location of the spectrum of $T$ and resolvent estimates in a sector of the complex plane and its complement, 
%In this paper we will frequently use the notation
\begin{align*}
\Sigma_{\psi}\stackrel{{\rm def}}{=}\{z\in \C\setminus \{0\}\colon |\arg z|<\psi\}, \quad \text{ and }\quad
\complement\overline{\Sigma_{\psi}}\stackrel{{\rm def}}{=}\C\setminus \overline{\Sigma_{\psi}} \quad \hbox{for } \psi \in (0,\pi),
\end{align*}
respectively, where we follow the notation in \cite{Analysis2}. We will keep this notation throughout this paper and denote by $\sigma(T)$ and $\rho(T)$ the spectrum and the resolvent set of $T$, respectively, where we always consider spaces over $\C$. If real spaces are needed, these can be obtained for the relevant examples by restriction.  
By $\Do(T)$, $\Ran(T)$, and $\ker(T)$ we denote the domain, the range, and the kernel of an operator $T$, respectively. We denote by $\id$ the identity map on $X$ and also write to make it short  $\lambda-T$ instead of $\lambda\id-T$ for $\lambda\in \C$.

In examples and applications we denote
as usual for a domain $\Dom\subseteq \R^d$ by $L^p(\Dom)$ for $p\in [1,\infty]$ the Lebesgue spaces.
For $s>0$ and $p\in (1,\infty)$ we denote by $H^{s,p}(\R^d)$ the Bessel potential space and set
\begin{align*}
	H^{s,p}(\Dom)&\stackrel{{\rm def}}{=}\big\{f\in \mathcal{D}'(\Dom)\colon F|_{\Dom} =f\text{ for some }F\in H^{s,p}(\R^d)\big\}, \\
	H^{s,p}_0(\Dom)&\stackrel{{\rm def}}{=} \overline{C^{\infty}_0(\Dom)}^{H^{s,p}(\Dom)} \quad \text{ and }\quad
	H^{-s,p}(\Dom)\stackrel{{\rm def}}{=}(H_0^{s,p'}(\Dom))^* \quad \hbox{where } \tfrac{1}{p}+\tfrac{1}{p'}=1,
\end{align*}
and where $X^\ast$ denotes the dual space of a Banach space $X$.
For $s\in \N$ we also use the notation $W^{s,p}=H^{s,p}$,  $W_0^{s,p}=H_0^{s,p}$,  and $W^{1,\infty}(\Dom)$ denotes the space of Lipschitz-continuous functions on $\Dom$. In the case $p=2$ we simply write $H^{s}=H^{s,2}$ for $s\in \R$. 
%By $\dot{H}^{s,p}(\R^d)$ we denote the \emph{homogeneous} Sobolev spaces, see e.g.\ \cite[Subsection 6.2]{BeLo}. 
If not indicated otherwise, all functions in these spaces are complex valued. For vector valued function spaces 
we denote for a Banach space $X$ by $L^p(\Dom;X)$ the usual Bochner spaces and by $H^{s,p}(\Dom;X)$ the corresponding Bessel potential spaces.
We identify $L^p(\Dom)^m=L^p(\Dom;\C^m)$ and $H^{s,p}(\Dom)^m=H^{s,p}(\Dom;\C^m)$ for $m\in \N$, and correspondingly for other functions spaces.

The real and imaginary part of a complex number $\lambda \in \C$ are denoted by $\Re \lambda$ and $\Im \lambda$, respectively.  By $a\lesssim b$ for $a,b\in \R$ we mean that there is a constant $C>0$ independent of $a,b$ such that $a\leq C b$, and by $\eqsim$ we mean that $\lesssim$ and $\gtrsim$ hold.
Finally, we set $a\vee b=\max\{a,b\}$ and $a\wedge b=\min\{a,b\}$.

\subsection{Sectorial operators}\label{sec:preliminaries}
\label{ss:sectorial_operator}
We say that the operator 
$$T\colon\Do(T)\subseteq X\to X$$ with range $\Ran(T)\subseteq X$ is a \emph{sectorial operator}  on the Banach space $X$ provided there exists $\omega \in (0,\pi)$ such that 
\begin{align}\label{eq:sec}
\sigma(T)\subseteq \overline{\Sigma_{\omega}}, \quad \overline{\Do(T)}=\overline{\Ran(T)}=X, \quad \hbox{and} \quad \sup_{\lambda\in \complement\overline{\Sigma_{\omega}}} \|\lambda(\lambda-T)^{-1}\|_{\calL(X)}<\infty,
\end{align}
and the infimum over all such $\omega\in (0,\pi)$ is called the \emph{angle of sectoriality} $\omega(T)$ of $T$. By \cite[Proposition 10.1.7(3)]{Analysis2}, such $T$ is injective. The operator $T$ is called \emph{pseudo-sectorial} if the assumption $\overline{\Do(T)}=\overline{\Ran(T)}=X$ in  \eqref{eq:sec} is dropped, cf. e.g. \cite[Definition 3.1.1]{pruss2016moving}.
If $T$ is sectorial of angle smaller than $\pi/2$, then $-T$ is the generator of an analytic semigroup.
\begin{comment}
\begin{remark}[Further notions of sectoriality]\label{rem:sec}
	There are several -- slightly diverging -- notions of sectoriality in the literature. Here, we follow the usage in \cite{KKW} or \cite[Definition 1.1 f.]{DHP}. In contrast in \cite[Definition 10.1.1]{Analysis2} the condition of the density of the domain and the range is dropped. Moreover, in~\cite[Definition 4.1]{EngNag00} it is the generator itself rather than its negative which is called sectorial; while in~\cite[Chapter 3, \S 3.10]{Kat80} an operator on a Hilbert spaces is called sectorial if its numerical range lies in a sector. Also, there are slight disambiguities concerning the %closedness or openness of sectors. 
	definitions of sectors.
\end{remark}
\end{comment}

\subsection{Fractional powers and scales of sectorial operators}
\label{ss:fractional_powers}
For a sectorial operator $T$ on a Banach space $X$ fractional powers $T^\gamma$ with domain $\Do(T^\gamma)\subseteq X$ for $\gamma\in \R$ can be defined, see e.g. \cite[Section 2.2]{DHP} and \cite[Chapter 3]{Haase:2} for a basic construction and also \cite{Kom_I, Kom_II, Kom_III, Kom_IV, Kom_V, Kom_VI} for a detailed study of fractional powers.
%By the extended functional calculus 
One can show that $T^{\gamma}$ is also injective on $X$, compare e.g. \cite[Theorem 15.15]{KuWe}. Then, the corresponding homogeneous space is defined by
\begin{equation}
\label{eq:def_homogeneous_scale}
\Dd(T^{\gamma})\stackrel{{\rm def}}{=}\big(\Do(T^{\gamma}),\|T^{\gamma}\cdot\|_X\big)^{\sim}\quad \hbox{for } \gamma\in \R,
\end{equation}
where $\sim$ denotes the completion. %, and as usual $\Do(T^{\gamma})$ denotes the domain of the fractional power $T^{\gamma}$.  
In \cite{KKW} these spaces are denoted by $\dot{X}_{\gamma,T}=\dot{X}_{\gamma}=\Dd(T^{\gamma})$. 
One can check that (see \cite[Lemma 6.3.2 a)]{Haase:2})
\begin{equation}
\label{eq:fractional_powers_positive_intersection_X}
\Do(T^{\gamma})=\Dd(T^{\gamma})\cap X \quad \hbox{for } \gamma\geq 0,
\end{equation}
and in particular if $0\in \rho(T)$, then $\Do(T^{\gamma})=\Dd(T^{\gamma})$ for $\gamma\geq 0$.

Following \cite[Section 2]{KKW} or \cite[Subsection 6.3]{Haase:2}, let us recall that $T$ uniquely induces an operator on $\Dd(T^{\gamma})$ for all $\gamma\in \R$. Indeed, it is easy to check that $T^{\gamma}\colon\Do(T^{\gamma})\to \Ran(T^{\gamma})$ extends to an isomorphism $\widetilde{T^{\gamma}}\colon\Dd(T^{\gamma})\to X$ the inverse of which  $(\widetilde{T^{\gamma}})^{-1}$ is an extension of $T^{-\gamma}\colon\Ran(T^{\gamma})\to \Do(T^{\gamma})$. Note that in general $(\widetilde{T^{\gamma}})^{-1}\neq (\widetilde{T^{-\gamma}})$ since these operators might act on different spaces. 
For any $\gamma\in \R$ we define the operator $\dot{T}_{\gamma}$ on $\Dd(T^{\gamma})$ by
\begin{equation*}
\begin{tikzcd}
X\supseteq \Do(T) \arrow{r}{T} 
& X \arrow{d}{\widetilde{(T^{\gamma}})^{-1}} \\
\Dd(T^{\gamma})\supseteq \Do(\dot{T}_{\gamma}) \arrow{u}{\widetilde{T^{\gamma}}}\arrow{r}{\dot{T}_{\gamma}}  & \Dd(T^{\gamma})
\end{tikzcd}
\end{equation*}
that is
\begin{align}\label{eq:definition_Tdot}
\dot{T}_{\gamma}\stackrel{{\rm def}}{=}(\widetilde{T^{\gamma}})^{-1}T \, \widetilde{T^{\gamma}}, \quad \text{ with domain }\quad
\Do(\dot{T}_{\gamma})=\Dd({T}^{\gamma+1})\cap \Dd(T^{\gamma}).
\end{align}
In particular, $\dot{T}_{\gamma}$ is similar to $T$ and therefore it has the same spectral properties as $T$. For future convenience, let us note that for all $\gamma\in \R$ and $\lambda\in \rho(T)=\rho(\dot{T}_{\gamma})$, one has $(\lambda-\dot{T}_{\gamma})^{-1}=(\widetilde{T^{\gamma}})^{-1}(\lambda-T )^{-1} \widetilde{T^{\gamma}} $ and therefore
\begin{equation}
\label{eq:consistency_T_T_gamma}
(\lambda-\dot{T}_{\gamma})^{-1}|_{\Dd(T^{\gamma})\cap X}=
(\lambda-T )^{-1}|_{\Dd(T^{\gamma})\cap X}.
\end{equation}
Extrapolation scales play an important role in the proofs in Sections~\ref{sec:Hinfty} and \ref{sec:extrapolation}, and they have been used also in perturbation theory presented in \cite{HaKu2006, Amansag_etal2022, KKW, Kun_2015}.

%REF to the proofs below

\subsection{$\Rsec$-boundedness and maximal $L^p$-regularity}
\label{ss:rsec_max_reg}
A family $\mathscr{J}\subseteq \calL(E,F)$
is \emph{Rademacher}- or \emph{$\Rsec$-bounded} with \emph{$\Rsec$-bound} $\Rsec(\mathscr{J})<\infty$ if for any sequence $(\varepsilon_n)$ of Rademacher variables, i.e., $\{+1,-1\}$-valued independent random variables on a probability space $(\O,\mathscr{A},\P)$  with mean zero, one has for all $T_1,\ldots, T_m\subseteq \mathscr{J}$ and $x_1,\ldots, x_m\in E$ that
\begin{align*}
\E\Big\|\sum_{n=1}^m \varepsilon_n T_n x_n\Big\|_F^2 \leq 
\Rsec(\mathscr{J})^2\E\Big\|\sum_{n=1}^m \varepsilon_n x_n\Big\|_E^2,
\end{align*}
where the expectation in $(\O,\mathscr{A},\P)$ is denoted by $\E$, cf.\ e.g.\ \cite[Chapter 8]{Analysis2}.
In case $\mathscr{J}= \{J_{\lambda}\,:\, \lambda \in \Lambda\}$ for some index set $\Lambda$ is $\Rsec$-bounded, we write $\Rsec(J_{\lambda}\,:\, \lambda \in \Lambda)<\infty$ instead of 
$\Rsec(\{J_{\lambda}\,:\, \lambda \in \Lambda\})<\infty$.

 A sectorial  operator $T$ in $X$ is called \emph{$\Rsec$-sectorial} if for some $\sigma \in (\omega(T), \pi)$ the family
\begin{align*}
\{\lambda (\lambda-T)^{-1}\colon \lambda\in \complement\overline{\Sigma_{\sigma}}\}
\end{align*}
is $\Rsec$-bounded.
The \emph{angle of $\Rsec$-sectoriality} of $T$ is
\begin{align*}
\omega_{\Rsec}(T)\stackrel{{\rm def}}{=}\inf\Big\{\sigma \in (\omega(T), \pi)\colon \{\lambda (\lambda-T)^{-1}\colon \lambda\in \complement\overline{\Sigma_{\sigma}}\} \hbox{ is $\Rsec$-bounded} \Big\},
\end{align*}
compare e.g. \cite[Definition 10.3.1]{Analysis2}.
%, we write $\Rsec(\mathscr{J})<\infty$ provided $\mathscr{J}$ is $\Rsec$-bounded.

An operator $T$ in a Banach space $X$ 
has \emph{maximal $L^p$-regularity} for $p\in (1,\infty)$ on $[0,\tau)$ with $\tau\in (0,\infty]$ if $T$  is closed and densely defined
%an analytic semigroup, 
and for $f\in L^p(0,\tau;X)$, $u_0=0$ the solution to the abstract Cauchy problem~\eqref{eq:CauchyProb} is differentiable almost everywhere and there exists a constant $C_p>0$ such that for all such $f$
\begin{align*}
\norm{\partial_t u}_{L^p(0,\tau;X)} + \norm{T u}_{L^p(0,\tau;X)} \leq C_p \norm{f}_{L^p(0,\tau;X)},
\end{align*}
cf. e.g. \cite[Section 1.3]{KuWe}. Note that if $T$ has maximal $L^p$-regularity, then $-T$ generates an analytic semigroup, see \cite{Coulhon1986}.
%The generator $\A$ of an analytic semigroup 
In a UMD space $X$, a closed densely defined operator $T$ has  maximal $L^p$-regularity for $p\in (1,\infty)$ if and only if 
it is $\Rsec$-sectorial of angle smaller than $\pi/2$, compare  \cite[Theorem 4.2]{We}.

\subsection{Bounded $H^{\infty}$-calculus} 
For any $\psi\in (0,2\pi)$, we denote by $H^{\infty}_0(\Sigma_{\psi})$ the set of all holomorphic functions $f:\Sigma_{\psi}\to \C$ such that $|f(z)|\lesssim |z|^{\varepsilon}/(1+|z|)^{-2\varepsilon}$ for all $z\in \Sigma_{\psi}$ and some $\varepsilon>0$. For a sectorial operator $T$ and all $f\in H^{\infty}_0(\Sigma_{\psi})$ where $\psi>\om(T)$ the 
Dunford integral
\begin{equation}
\label{eq:dunford_H_calculus_definition}
f(T)\stackrel{{\rm def}}{=}
\frac{1}{2\pi i}\int_{\partial\Sigma_{\sigma}} f(\lambda)(\lambda-T)^{-1} \,\dd\lambda,
\end{equation}
is absolutely convergence and independent of $\sigma\in (\om(T),\psi)$. We say that $T$ has a bounded $H^\infty(\Sigma_\psi)$-calculus for $\psi \in (\omega(T), \pi)$ if there exists $C>0$ such that 
\begin{align*}
\norm{f(T)}_{\calL(X)} \leq C \|f\|_{H^{\infty}(\Sigma_{\psi})} \quad \hbox{for all } f\in H^\infty_0(\Sigma_\psi),\end{align*}
where $\|f\|_{H^{\infty}(\Sigma_{\psi})}\stackrel{{\rm def}}{=}\sup_{z\in\Sigma_{\psi}}|f(z)|$, 
and the \emph{$H^\infty$-angle} of $T$ is
\begin{align*}
\omega_{H^\infty}(T)\stackrel{{\rm def}}{=}
\inf\{\psi \in (\omega(T), \pi)\colon T \hbox{ has a bounded $H^\infty(\Sigma_\psi)$-calculus} \},
\end{align*}
compare e.g. \cite[Definition 10.2.10]{Analysis2}. Finally, we say that $T$ has a bounded $H^{\infty}$-calculus provided $T$ has a bounded $H^{\infty}(\Sigma_{\sigma})$-calculus for some $\sigma\in (0,\pi)$.

For an UMD space $X$, one has the inclusions
$$H^\infty(X) \subseteq \mathcal{SMR}(X) \subseteq \Ssec(X) \ \ \text{ and } \ \ 
H^\infty(X) \subseteq \Rsec(X) \subseteq \Ssec(X),
%\quad 
%\text{ and } \quad
%H^\infty(X) \subseteq \Rsec(X)  %\subseteq \Ssec(X),
$$
where $H^\infty(X)$, $\mathcal{SMR}(X)$, $\Rsec(X)$, and $\Ssec(X)$ stand for the classes of operators in a UMD space $X$ having a bounded $H^\infty$-calculus, 
 admitting stochastic maximal $L^p$-regularity,
being $\Rsec$-sectorial and sectorial, respectively, compare e.g \cite[Equation (2.15) and Theorem 4.5]{DHP}, \cite{AV19,MaximalLpregularity} and \cite[Section 6]{ALV23}.

By holomorphy, one can check that if $T$ has a bounded $H^{\infty}$-calculus, then the same holds for $\mu+T$ where $\mu>0$. The following is a partial converse of the latter observation which follows from \cite[Corollary 5.5.5]{Haase:2}. % see also \cite[Corollary 5.5.5]{Haase:2} and \cite[Proposition 6.10]{KKW} for results in the same direction. 

\begin{proposition}[$H^{\infty}$-calculus for shifted operators]
\label{prop:shift_H_calculus}
Let $T$ be a linear operator. Assume that $\mu_0+T$ has a bounded $H^{\infty}$-calculus for some $\mu_0> 0$. Suppose that $\rho(T)\supseteq \{0\}\cup \{z\in \C\,:\,|\arg z|\geq \sigma\}$  for some $\sigma>\angH(\mu_0+T)$.
 Then
$T$ has a bounded $H^{\infty}$-calculus of angle $\leq \sigma$.
\end{proposition}

\begin{proof}
	%[Proof of Proposition \ref{prop:shift_H_calculus}]
Let us begin by proving that $T$ is sectorial of angle $\omega (T)\leq  \sigma$. It is enough to show that $T$ is sectorial of angle $\omega(T)\leq \psi$ where $\psi>\sigma$ is arbitrary. The sectoriality of $\mu_0+T$ implies that $\rho(T)\supseteq -\mu_0+\complement\overline{\Sigma}_{\phi}$ for all $\phi\in [\psi,\angH(\mu_0+T))$.
 Fix 
 %$\phi\in (\psi,\angH(\mu_0+T))$ 
 such $\phi$
 and note that $\complement\overline{\Sigma}_{\psi}\setminus (-\mu_0+\complement\overline{\Sigma}_{\phi})$ is compact and it is contained in  $\{z\in \C\,:\,|\arg z|\geq \sigma\}$. Therefore $\rho(T)\supseteq \complement\overline{\Sigma}_{\psi}$ and the estimate
 $$
 \sup_{\lambda\in \complement\overline{\Sigma}_{\psi}} \|\lambda (\lambda-T)^{-1}\|_{\calL(X)}<\infty
 $$
 follows from the one of $\mu_0+T$ using \cite[Lemma 10.2.4]{Analysis2}.
 Now, applying \cite[Corollary 5.5.5]{Haase:2} with $A=T+\mu_0$ and $B=-\mu_0$, the statement follows. 
\end{proof}

\section{Block operator matrices}\label{sec:blockoperator}

\subsection{Diagonal dominance} The standing  assumption throughout this note is

\begin{assumption}[Diagonal dominance]	\label{ass:standing} Let $X_1,X_2$ be Banach spaces, and 
\begin{align*}
A&\colon \Do(A)\subseteq X_1\to X_1, \quad \text{ and }\quad 
D\colon \Do(D)\subseteq X_2\to X_2,	\\
B&\colon \Do(B)\subseteq X_2\to X_1, \quad \text{ and }\quad 
C\colon \Do(C)\subseteq X_1\to X_2
\end{align*}
be linear operators with domains $\Do(A)$, $\Do(D)$, $\Do(B)$, and $\Do(C)$, respectively, where
	\begin{enumerate}[{\rm(1)}]
		\item $A\colon\Do(A)\subseteq X_1\to X_1$ and $D\colon\Do(D)\subseteq X_2\to X_2$ are closed linear operators with dense domains;
		\item 
		$\Do(D) \subseteq \Do(B)$, $\Do(A) \subseteq \Do(C)$, and there exist $c_A,c_D,L\geq 0$ such that %, for all $$ and $y\in \Do(D)$, 
		\begin{align*}
		\|C x\|_{X_2}&\leq c_A \|Ax\|_{X_1}+ L \|x\|_{X_1} \quad \hbox{for all } x\in\Do(A),\\
		\|B y\|_{X_1}&\leq c_D \|D y\|_{X_2}+ L \|y\|_{X_2}\quad \hbox{for all } y\in\Do(D).
		\end{align*}
	\end{enumerate}
\end{assumption}  
With this assumption one sets
\begin{align*}
X\stackrel{{\rm def}}{=}X_1\times X_2 \quad \text{ and }\quad \Do(\A)\stackrel{{\rm def}}{=} \Do(A)\times \Do(D),
\end{align*}
compare \cite[Equation (2.2.3)]{Tretter}, and  
\begin{equation*}
\A:\Do(\A)\subseteq X\to X  \quad\hbox{with} \quad	\A
\begin{bmatrix}
x \\
y 
\end{bmatrix}
\stackrel{{\rm def}}{=}
\begin{bmatrix}
A & B  \\
C & D 
\end{bmatrix}
\begin{bmatrix}
x\\
y 
\end{bmatrix}, \quad \text{for all }\begin{bmatrix}
x\\
y 
\end{bmatrix}\in \Do(\A).
\end{equation*}
The operator $\A$ is called \emph{diagonally dominant} if Assumption~\ref{ass:standing} holds.  Here, compared to \cite[Definition 2.2.1]{Tretter}, closedness of $A$ and $D$ is assumed for simplicity instead of closability, see also \cite[Assumption 2.2]{Nag89} for a similar definition including closedness of $\A$. 
For $\A$ diagonally dominant we also write $\A = \diag + \B$, where
\begin{align}\label{eq:off_diagional_matrix}
\diag	\stackrel{{\rm def}}{=}	\begin{bmatrix}
A & 0  \\
0 & D 
\end{bmatrix} \quad\hbox{and}\quad \B	\stackrel{{\rm def}}{=}	\begin{bmatrix}
0 & B  \\
C & 0 
\end{bmatrix} \quad  \hbox{ with }\quad  \Do(\diag)= \Do(\B)=\Do(\A).
\end{align}
A diagonally dominant $\A$ is then obtained from the diagonal part $\diag$ by the relatively bounded perturbation $\B$. In particular its domain is already determined by the diagonal part $\diag$ which determines the ``degree of unboundedness'' of $\A$ as formulated in \cite[Assumption 2.2 f.]{Nag89}.

\begin{remark}[Boundedness of $B(\lambda-D)^{-1}$ and $C(\lambda-A)^{-1}$]\label{rem:ass_standing}
	By Assumption~\ref{ass:standing} %implies that % if $\lambda\in \rho(D)\neq \emptyset$, then 
	\begin{align*}
	B(\lambda-D)^{-1}\colon X_2\rightarrow X_1 \quad \hbox{if}\quad \lambda\in \rho(D)\neq \emptyset
	\end{align*}
%	
%	$B(\lambda-D)^{-1}\colon X_2\rightarrow X_1$, $\lambda\in \rho(D)$,
	is bounded, 
	since 	for all  $y\in X_2$
	\begin{align*}
	\norm{B(\lambda-D)^{-1}y}_{X_1} &\leq c_D\norm{D(\lambda-D)^{-1}y}_{X_2} + L \norm{(\lambda-D)^{-1}y}_{X_2} \\
	&\leq c_D(1+\norm{ \lambda(\lambda-D)^{-1}}_{\calL(X_2)}\norm{y}_{X_2} + L \norm{(\lambda-D)^{-1}}_{\calL(X_2)} \norm{y}_{X_2},
	\end{align*}
	and analogously, if  $\lambda\in\rho(A)\neq \emptyset$, then $C(\lambda-A)^{-1}\colon X_1\rightarrow X_2$,
	is bounded. 
	If $D$ is injective with dense range, and if Assumption~\ref{ass:standing} holds with $L=0$, then 
	\begin{align*}
	\norm{B D^{-1}y}_{X_2} \leq c_D\norm{y}_{X_2} \quad \hbox {for all } y\in \Ran(D),
	\end{align*}
	and since $\Ran(D)\subseteq X_2$ is dense, $B D^{-1}$ has a unique continuous extension in $\calL(X_2,X_1)$, and similarly for $C A^{-1}$ in $\calL(X_1,X_2)$ if $A$ is injective with dense range.
\end{remark}

\subsection{A factorization of diagonally dominant operators}\label{subsec:M}
For $\A$  diagonally dominant and $\lambda\in \rho(A)\cap \rho(D)=\rho(\diag)\neq \emptyset$,
one sets
\begin{equation}
\label{eq:def_M_A_0}
\M(\lambda)
\stackrel{{\rm def}}{=} (\lambda-\A)(\lambda-\diag)^{-1}.
\end{equation}
The operators defined by
\begin{align*}
S_1(\lambda)&
\stackrel{{\rm def}}{=} (\lambda-A) - B(\lambda-D)^{-1}C \quad \hbox{for } \quad \lambda \in \rho(D),\\
S_2(\lambda)&
\stackrel{{\rm def}}{=} (\lambda-D) - C(\lambda-A)^{-1}B \quad \hbox{for } \quad \lambda \in \rho(A),
\end{align*}
are called \emph{Schur-complements} of $\A$, see e.g. \cite[Definition 2.2.12]{Tretter}, and they serve in the theory of block operator matrices as a substitute to the determinant of $2\times2$-matrices.
Factoring out $(\lambda-A)$ and $(\lambda-D)$ respectively, one obtains the pair of operators
\begin{align}\label{eq:def_M_1_M_2}
\begin{split}
M_1(\lambda) &
\stackrel{{\rm def}}{=}
\id - B (\lambda-D)^{-1} C(\lambda- A)^{-1}  \quad \hbox{for } \quad \lambda \in \rho(A)\cap \rho(D),\\
M_2(\lambda) &
\stackrel{{\rm def}}{=}
\id - C (\lambda-A)^{-1} B(\lambda- D)^{-1} \quad \hbox{for } \quad \lambda \in \rho(A)\cap \rho(D).
\end{split}
\end{align}
These  are the main building blocks in the following factorization, where $M_1(\lambda)$ and $M_2(\lambda)$ encode the interaction of the different blocks of $\A$. These operators have been introduced in the context of the spectral theory of block operator matrices, see\cite[Lemma 2.1]{Nag89} for the case of bounded block operator matrices, \cite[Theorem 2.4]{Nag97}  for the unbounded diagonally dominant triangular case, and also e.g. \cite[Proposition 2.3.4 ff.]{Tretter} for the general diagonally dominant case.

\begin{proposition}[Factorization of diagonally dominant $\A$]
	\label{l:relation_resolvent_AD_calA}
	Let Assumption \ref{ass:standing} be satisfied, 
	and let $\lambda\in \rho(A)\cap \rho(D)\neq \emptyset$.  
	Then the following hold:
	\begin{enumerate}[{\rm(a)}]
		\item\label{eq:basic_identity_MA} 
		One has
		\begin{equation*}%
		\lambda-\A= \M(\lambda)(\lambda-\diag) \quad \hbox{and} \quad 
		\M(\lambda)=\begin{bmatrix}
		\id & -B(\lambda-D)^{-1}  \\
		-C(\lambda-A)^{-1} & \id 
		\end{bmatrix}.
		\end{equation*}
		\item\label{it:lambda_M} The operators $M_1(\lambda)$, $M_2(\lambda)$, and $\M(\lambda)$ are bounded on $X_1$, $X_2$, and $X$, respectively.
		%respectively.
		\item\label{it:lambda_resolvent_c} The following are equivalent
		\begin{enumerate}[{\rm(1)}]
			\item\label{it:lambda_resolvent_calA} $\lambda\in \rho(\A)$;
			\item \label{it:M_invertible} $\M(\lambda)$ is boundedly invertible;
			\item \label{it:M_1_invertible} $
			M_1(\lambda)
			$ is boundedly invertible;
			\item \label{it:M_2_invertible} $
			M_2(\lambda)
			$ is boundedly invertible.
		\end{enumerate}
		\item	\label{eq:resolvent_formula_block_matrix_I} If one of the previous conditions in \eqref{it:lambda_resolvent_c} is satisfied, then 
		\begin{align*}%\label{eq:resolvent_formula_block_matrix_I}  
		(\lambda-\A)^{-1}
		= (\lambda-\diag)^{-1}\M(\lambda)^{-1}
		\end{align*}
		where $\M(\lambda)^{-1}$ has the block matrix representations%\todo{Are the sign ok? See the proof of Lemma 4.4 below}
		\begin{align*}
		\M(\lambda)^{-1}&=\begin{bmatrix}
			M_1(\lambda)^{-1} & M_1(\lambda)^{-1}B(\lambda-D)^{-1}  \\
			M_2(\lambda)^{-1} C(\lambda-A)^{-1} & M_2(\lambda)^{-1}  
			\end{bmatrix}\\
		&=
		\begin{bmatrix}
		\id & B(\lambda-D)^{-1} M_2(\lambda)^{-1} \\
		0 & M_2(\lambda)^{-1}
		\end{bmatrix} 
		\begin{bmatrix}
		\id & 0  \\
		C(\lambda-A)^{-1} & \id 
		\end{bmatrix}
		\\
		&= 	\begin{bmatrix}
		M_1(\lambda)^{-1} & 0 \\
		C(\lambda-A)^{-1}M_1(\lambda)^{-1} & \id
		\end{bmatrix}
		\begin{bmatrix}
		\id & B(\lambda-D)^{-1}  \\
		0 & \id 
		\end{bmatrix}.
		\end{align*}
	\end{enumerate}
\end{proposition}

\begin{proof}[Proof of Proposition~\ref{l:relation_resolvent_AD_calA}]
	The factorization and the block representation in  \eqref{eq:basic_identity_MA} can be verified in a straightforward way on $\Do(\A)=\Do(\diag)$ using the block representations of $\lambda-\A$ and $(\lambda-\diag)^{-1}$.

	To prove \eqref{it:lambda_M}, note that
	by Remark~\ref{rem:ass_standing} the operators $M_1(\lambda)$ and $M_2(\lambda)$ are bounded if $\A$ is diagonally dominant. Similarly, using the block matrix representation in \eqref{eq:basic_identity_MA} of $\M(\lambda)$ it follows that $\M(\lambda)$ is bounded.

	The statement of \eqref{it:lambda_resolvent_c} is essentially given in \cite[Corollary 2.3.5]{Tretter}, see also \cite[Lemma 2.1]{Nag89} for the case of bounded operators and \cite[Theorem 2.4]{Nag97} for block triangular operator matrices. For the sake of completeness the proof is given here. % in a more elementary way. 
	The implication \eqref{it:lambda_resolvent_calA}$\Rightarrow$\eqref{it:M_invertible}
	follows 
	since for $\lambda\in \rho(\A)\cap \rho(\diag)$, \eqref{eq:basic_identity_MA} implies  that $\M(\lambda)$ is invertible and
	\begin{equation}
	\label{eq:inverse_Mcal_assuming_invertibility_calLA}
	\M(\lambda)^{-1}=(\lambda-\diag)(\lambda-\A)^{-1}.
	\end{equation}
	is bounded. 
	Note that by a row reduction one obtains the following factorizations
	\begin{align}\label{eq:factorization_M}
	\begin{split}
	\M(\lambda) &= 
	\begin{bmatrix}
	\id & 0  \\
	-C(\lambda-A)^{-1} & \id 
	\end{bmatrix}
	\begin{bmatrix}
	\id & -B(\lambda-D)^{-1}  \\
	0 & M_2(\lambda)
	\end{bmatrix} \\
	&= 
	\begin{bmatrix}
	\id & -B(\lambda-D)^{-1}  \\
	0 & \id 
	\end{bmatrix}
	\begin{bmatrix}
	M_1(\lambda) & 0 \\
	-C(\lambda-A)^{-1} & \id
	\end{bmatrix}.
	\end{split}
	\end{align}
	% and hence \eqref{eq:resolvent_formula_block_matrix_I} follows if $M_1(\lambda)$ and $M_2(\lambda)$ are invertible.
	From   \eqref{eq:factorization_M} one sees also that \eqref{it:M_invertible}, 	\eqref{it:M_1_invertible} and 
	\eqref{it:M_2_invertible} are equivalent. From~\eqref{eq:basic_identity_MA} it follows that if one of the conditions
	\eqref{it:M_invertible}, 	\eqref{it:M_1_invertible} or 
	\eqref{it:M_2_invertible} holds for $\lambda\in \rho(\diag)$, then $\lambda-\A$ is boundedly invertible and hence $\lambda\in \rho(\A)$.
	
	Part~\eqref{eq:resolvent_formula_block_matrix_I} follows from \eqref{eq:basic_identity_MA} and the representation \eqref{eq:factorization_M}.	
\end{proof}

\subsection{Comparison to further factorizations}\label{subsec:Schur}
In \cite{Engel_1995} also the operators
\begin{align*}
\begin{split}
N_1(\lambda) &
\stackrel{{\rm def}}{=}
\id - (\lambda- A)^{-1} B (\lambda-D)^{-1} C  \quad \hbox{for } \quad \lambda \in \rho(A)\cap \rho(D),\\
N_2(\lambda) &
\stackrel{{\rm def}}{=}
\id - (\lambda- D)^{-1} C (\lambda-A)^{-1} B \quad \hbox{for } \quad \lambda \in \rho(A)\cap \rho(D).
\end{split}
\end{align*}
have been introduced by Engel which give rise to a factorization of the type 
\begin{align*}
\lambda-\A= (\lambda-\diag)\mathcal{N}(\lambda), \quad \hbox{where}\quad \mathcal{N}(\lambda)= (\lambda-\diag)^{-1}(\lambda-\A),
\end{align*}
and its inverse can be represented in terms of $N_1(\lambda)$ and $N_2(\lambda)$. Both factorizations are equivalent for the setting discussed here, since for $\lambda \in \rho(A)\cap \rho(D)$
	\begin{align*}
	M_1(\lambda)= (\lambda-A)N_1(\lambda)(\lambda-A)^{-1}	 \quad \hbox{and}\quad M_2(\lambda)= (\lambda-D)N_2(\lambda)(\lambda-D)^{-1}.
	\end{align*}
	A technical advantage of the factorization $\lambda-\A= \M(\lambda) (\lambda-\diag)$ used here is that $\M(\lambda)$ is a bounded operator in $X$.

Furthermore, the \emph{Frobenius-Schur factorization} is a classical 
factorization for block operator matrices. Under the general assumptions 
$\Do(A)\subseteq \Do(C)$, $\rho(A)\neq \emptyset$, and that, for some (and hence for all) $\lambda \in  \rho(A)$, the operator $(\lambda-A)^{-1}B$ is bounded on $\Do(B)$, one has -- assuming here for simplicity that $\A$ is closed -- that 
\begin{align*}
\lambda -\A = 
\begin{bmatrix}  \id & 0 \\ C(\lambda-A)^{-1}& \id  \end{bmatrix} 
\begin{bmatrix}  \lambda -A & 0 \\ 0 & S_2(\lambda) \end{bmatrix} 
\begin{bmatrix}  \id  &\overline{(\lambda-A)^{-1} B} \\ 0& \id  
\end{bmatrix},
\end{align*}
compare \cite[Theorem 2.2.14]{Tretter}. If $\A$ is diagonally dominant, then $\Do(S_2(\lambda))=\Do(D)$ for $\lambda \in \rho(A)$, cf. \cite[Rem. 2.2.13]{Tretter}.   Analogous statements hold for $S_1(\lambda)$. 	For $\A$ satisfying Assumption~\ref{ass:standing}, $\lambda \in \rho(\diag)$, and $S_j(\lambda)$ closed for $j\in \{1,2\}$,   one has that bounded invertibility of $M_j(\lambda)$ and $S_j(\lambda)$  are equivalent. On the one hand, if for instance $M_1(\lambda)$ is boundedly invertible, then $S_1(\lambda)^{-1}=(\lambda-A)^{-1}M_1(\lambda)^{-1}$ is bijective and bounded. On the other hand, if $S_1(\lambda)$ is boundedly invertible, then $M_1(\lambda)=S_1(\lambda)(\lambda-A)^{-1}$ is closed and bijective, and hence boundedly invertible. The analogous argument applies to the case $j=2$.

The factorization in Subsection~\ref{subsec:M} is valid only for diagonally dominant operators with non-empty resolvent set, whereas the Frobenius-Schur decomposition applies to a general class of  block operator matrices.
The Frobenius-Schur decomposition needs the assumption that
$(\lambda-A)^{-1}B$ is bounded on $\Do(B)$, and this is quite restrictive for the purpose of the analysis presented below, and it does not hold automatically for diagonally dominant operators as discussed in the following Example~\ref{ex:cond_onB}.
	The condition that $(\lambda-A)^{-1}B$ is bounded on $\Do(B)$ is related to the orders of $A,D$ and $B$. Heuristically, since $B(\lambda-D)^{-1}$ is bounded for $\A$ diagonally dominant the order of $B$ is at most the order of $D$, and similarly, that $(\lambda-A)^{-1}B$ is bounded implies  that the order of $B$ is at most the order of $A$. 
 %Therefore, for studying diagonally dominant sectorial operators the factorization given in Subsection~\ref{subsec:M} seems better suited. 
\begin{example}[The condition on $(\lambda-A)^{-1}B$]\label{ex:cond_onB}
In  $L^p(\R^d)\times L^p(\R^d)$ for $p\in (1,\infty)$ %the operator 
	\begin{align*}
	\A_\alpha= \begin{bmatrix}
	-\Delta & (-\Delta)^\alpha \\ 0 & (-\Delta)^\alpha
	\end{bmatrix} \quad  \hbox{ with } 
	\Do(\A_\alpha)=H^{2,p}(\R^d)\times H^{2\alpha,p}(\R^d), \quad \alpha >0
	\end{align*}
	is diagonally dominant for all $\alpha>0$, but for $\alpha >1$
	\begin{align*}
	(\lambda-A)^{-1}B = (\lambda+\Delta)^{-1}(-\Delta)^\alpha, \quad \lambda \in \rho(-\Delta)=\C\setminus [0,\infty)
	\end{align*}
	does not extend to a bounded operator on  $L^p(\R^d)$. Nevertheless in this situation the factorization given in Subsection~\ref{subsec:M} applies.
\end{example}

\section{Sectoriality and $\mathcal{R}$-sectoriality for block operator matrices}\label{sec:sec_Rsec}
In this section we present our main results concerning the sectoriality and $\Rsec$-sectoriality of block operator matrices. Further perturbation results for block operators with smallness conditions are given in Subsection~\ref{subsec:perturbative_Rsec}.
The proofs of Theorem~\ref{t:rsec_necessary_sufficient_condition_II},  Corollary~\ref{t:rsec_necessary_sufficient_condition}, and Proposition \ref{l:one_side_A_0_A_estimate}   will be given in Subsection \ref{ss:proofs_sectoriality} below.

\begin{theorem}[Characterization of sectoriality and $\Rsec$-sectoriality]
	\label{t:rsec_necessary_sufficient_condition_II}
	Suppose that Assumption \ref{ass:standing} with $L=0$ holds and that \begin{equation}
	\label{eq:equivalence_norm_A_0_A}
	\|\diag x\|_X\lesssim \|\A  x\|_{X}\ \ \ \text{ for all }x\in \Do(\diag)=\Do(\A).
	\end{equation}
	\begin{enumerate}[{\rm(a)}]
		\item\label{it:A_sectorial_a} If $A$ and $D$ are sectorial operators, then 
		for each $\psi\in  [\omega(A)\vee \omega(D),\pi)$ the following are equivalent:
		\begin{enumerate}[{\rm(1)}]
			\item\label{it:A_sectorial} $\A$ is sectorial of angle $\psi$;
			\item\label{it:schur_compleements_bounded} $\overline{\Ran(\A)}=X$, and for all $\phi>\psi$ and for one $j\in \{1,2\}$  %
			\begin{align*}
			M_j(\lambda)^{-1}\in \calL(X_j) \hbox{ for all } \lambda\in \complement\overline{\Sigma_{\phi}}, \hbox{ and } \sup \{\norm{M_j(\lambda)^{-1}}\colon \lambda\in   \complement\overline{\Sigma_{\phi}}\}<\infty.
			\end{align*}
		\end{enumerate}
		\item\label{it:A_R_sectorial_equivalence} If $A$ and $D$ are $\mathcal{R}$-sectorial operators, then 
		for each $\psi\in  [\angR(A)\vee \angR(D),\pi)$ the following are equivalent:
		\begin{enumerate}[{\rm(1)}]
			\item\label{it:A_R_sectorial} $\A$ is $\Rsec$-sectorial of angle $\psi$;
			\item\label{it:schur_compleements_R_bounded}
			$\overline{\Ran(\A)}=X$, and for all $\phi>\psi$ and for one $j\in \{1,2\}$  %the operator 
			\begin{align*}
			M_j(\lambda)^{-1}\in \calL(X_j) \hbox{ for all } \lambda\in \complement\overline{\Sigma_{\phi}}, \hbox{ and } \Rsec(M_j(\lambda)^{-1}\colon \lambda\in   \complement\overline{\Sigma_{\phi}})<\infty.
			\end{align*}
		\end{enumerate}
		\item \label{it:X_reflexive} Finally, if $X$ is reflexive, then the condition $\overline{\Ran(\A)}=X$ in \eqref{it:schur_compleements_bounded} and \eqref{it:schur_compleements_R_bounded} can be removed.
	\end{enumerate}
	
\end{theorem}

\begin{remark}[Optimality of the angle]
	If $C=B=0$, then $\A=\diag$ and therefore (in general) the inequalities $\omega(\A)\geq \omega(A)\vee \omega(D)=\omega(\diag)$ and $\angR(\A)\geq \angR(A)\vee \angR(D)=\angR(\diag)$ cannot be  improved.
\end{remark}
\begin{remark}[Closedness of $\A$]
	The closedness of $\A$ does not follow from Assumption~\ref{ass:standing}. In fact, it can be characterized by the Schur complements, cf. \cite[Theorem 2.2.14]{Tretter}, and there seems to be no such characterization by the operators $M_1(\cdot)$ and $M_2(\cdot)$ besides that $\A$ is closed if one of the conditions in Proposition~\ref{l:relation_resolvent_AD_calA}~\eqref{it:lambda_resolvent_c} holds. 
	However, \eqref{eq:equivalence_norm_A_0_A} 
	together with Assumption~\ref{ass:standing} implies that 
	\begin{equation*}
	\|\A  x\|_{X} +\|x\|_X \eqsim \|\diag x\|_X + \|x\|_X \quad \text{for all }x\in \Do(\A)=\Do(\diag),
	\end{equation*}
	and hence
	together with the closedness of  $\diag$, the closedness of $\A$ follows. In particular Proposition~\ref{l:one_side_A_0_A_estimate} below implies already the closedness of $\A$. For further conditions ensuring the closedness of $\A$ see also \cite[Theorem 2.2.8]{Tretter}.
\end{remark}

\begin{corollary}[Characterization of sectoriality and $\Rsec$-sectoriality for invertible $\diag$]
	\label{t:rsec_necessary_sufficient_condition}
	Let Assumption \ref{ass:standing} be satisfied, and assume that $A$ and $D$ are boundedly invertible.
	\begin{enumerate}[{\rm(a)}]
		\item If $A$ and $D$ are sectorial operators,
		then for each $\psi\in [\omega(A)\vee \omega(D),\pi)$ the following are equivalent:
		\begin{enumerate}[{\rm(1)}]
			\item\label{it:A_sectorial_cor} $\A$ is an invertible sectorial operator of angle $\psi$;
			\item\label{it:schur_compleements_bounded_cor} For all $\phi>\psi$  and for one $j\in \{1,2\}$  %the operator 
			\begin{align*}
			M_j(\lambda)^{-1}\in \calL(X_j) \hbox{ for all } \lambda\in \complement\overline{\Sigma_{\phi}}\cup\{0\}, \hbox{ and } \sup \{\norm{M_j(\lambda)^{-1}}\colon \lambda\in   \complement\overline{\Sigma_{\phi}}\}<\infty.
			\end{align*}
		\end{enumerate}
		\item If $A$ and $D$ are $\mathcal{R}$-sectorial operators,
		then for each $\psi\in [\angR(A)\vee \angR(D),\pi)$ the following are equivalent:
		\begin{enumerate}[{\rm(1)}]
			\item\label{it:A_R_sectorial_cor} $\A$ is an invertible $\Rsec$-sectorial operator of angle $\psi$;
			\item\label{it:schur_compleements_R_bounded_cor} 
			For all $\phi>\psi$  and for one $j\in \{1,2\}$
			\begin{align*}
			M_j(\lambda)^{-1}\in \calL(X_j) \hbox{ for all } \lambda\in \complement\overline{\Sigma_{\phi}}\cup\{0\}, \hbox{ and } 	\Rsec(M_j(\lambda)^{-1}\colon \lambda\in   \complement\overline{\Sigma_{\phi}})<\infty.
			\end{align*}
		\end{enumerate}
	\end{enumerate}
\end{corollary}

Paraphrasing the above results one has that sectoriality and $\Rsec$-sectoriality for angle larger than $\omega(A)\vee \omega(D)$ and $\angR(A)\vee \angR(D)$, respectively, 
of a block operator matrix $\A$ is solely determined by one of the \emph{bounded} operators $M_j(\lambda)$ in $X_j$ for $j\in\{1,2\}$ defined in \eqref{eq:def_M_1_M_2}, and by \eqref{eq:factorization_M} one sees that if the condition holds for one of the operators $M_j(\lambda)$, $j\in\{1,2\}$, then  it also holds for the other.

In the study of long-time behavior of solutions to nonlinear partial differential equations, %\todo{Antonio: I have to include some references. Amru: As a suggestion, I took the book by Prüss and Simonett, possibly add some reference on polynomial decay},
see e.g. \cite[Section 5.3]{pruss2016moving}, the assumption $0\in \rho(A)\cap \rho(D)$ in Corollary~\ref{t:rsec_necessary_sufficient_condition} is too restrictive, though it simplifies the formulation of the statement considerably,  and in fact it would exclude even cases such as the Laplace operator on $\R^d$. To avoid this limitation, we assumed in Theorem~\ref{t:rsec_necessary_sufficient_condition_II} instead condition \eqref{eq:equivalence_norm_A_0_A}.
Next, we give sufficient conditions for   \eqref{eq:equivalence_norm_A_0_A}  to hold. %Note that, if Assumption \ref{ass:standing} is satisfied with $L=0$, then the operators $C A^{-1}$ and $B D^{-1}$ admits an extension as bounded linear operator which will be denoted by $G$ and $H$, respectively provided $\overline{\Ran(A)}=X_1$ and $\overline{\Ran(D)}=X_2$.

\begin{proposition}[Criteria for condition \eqref{eq:equivalence_norm_A_0_A}]
	\label{l:one_side_A_0_A_estimate}
	Let Assumption \ref{ass:standing} be satisfied with $L=0$, and assume that $\overline{\Ran(A)}=X_1$ and $\overline{\Ran(D)}=X_2$. Then set
	%the operators $C A^{-1}$ and $B D^{-1}$ are closable and one has
	\begin{equation}
	\label{eq:def_GH_lemma}
	G\stackrel{{\rm def}}{=} \overline{C A^{-1}}\in \calL(X_1,X_2) \quad\text{and} \quad H\stackrel{{\rm def}}{=}\overline{B D^{-1}}\in \calL(X_2,X_1),
	\end{equation}
	and if one of the operators $$
	\id - H G \quad \hbox{and}\quad\id - GH$$
	is boundedly invertible, then \eqref{eq:equivalence_norm_A_0_A} holds.\ In particular, \eqref{eq:equivalence_norm_A_0_A} holds if for  $\varepsilon> 0$ 
	\begin{align*}
	\sup\{\norm{B(t+D)^{-1}C(t+A)^{-1}}\colon t\in (0,\varepsilon)\}<1,&\text{ or }\\
	 \sup\{\norm{C(t+A)^{-1}B(t+D)^{-1}}\colon t\in (0,\varepsilon)\}<1.&
	\end{align*}
\end{proposition}

\begin{remark}[Density of $\Ran(\A)$]\label{rem:RanA_dense}
	By  \eqref{eq:def_GH_lemma} the operator $\M(0)\stackrel{{\rm def}}{=}\A\diag^{-1}$ extends to a bounded linear operator on $X$, and a  factorization analogous to the one in Proposition~\ref{l:relation_resolvent_AD_calA}~\eqref{eq:resolvent_formula_block_matrix_I}
	holds also for $\lambda=0$.
Then the condition that one of the operators $
	\id - H G$ and $ \id - GH$
	is boundedly invertible implies that $\overline{\M(0)}$ is boundedly invertible. Thus $\Ran (\A)=\Ran (\M(0)\diag)$ and $\overline{\Ran(\A)}=X$ in case $\overline{\Ran(\diag)}=X$.
\end{remark}

% In the next section we derive some useful consequences.  

The next proposition can be seen as a variation of \cite[Theorem 4.4.4]{pruss2016moving} and \cite[Theorem 4.2]{We} where $\Rsec$-bounds on the relevant operators appear only on subsets %of 
\begin{align*}
\ell_{\theta}\stackrel{{\rm def}}{=}\{r e^{ i\theta}\colon r>0\}\cup \{r e^{- i\theta}\colon r>0\}\subseteq \complement\overline{\Sigma_{\psi}} \subseteq\C \quad  \hbox{for } \theta\in (0,\pi) \hbox{ and } \psi \in (0,\theta).
\end{align*}
%\begin{align*}
%\ell_{\theta}\stackrel{{\rm def}}{=}\{r e^{ i\theta}\colon r>0\}\cup \{r e^{- i\theta}\colon r>0\}\subseteq \complement\overline{\Sigma_{\theta}} \subseteq\C \quad  \hbox{for } \theta\in (0,2\pi).
%\end{align*}
%\todo{Antonio: Depending on the position we will chose for this result the proof can be shorten. The argument used has been used elsewhere in the manuscript. I am not sure which could be the correct position for the following result. } 
%To formulate this we set for  $\theta\in (0,2\pi)$

\begin{proposition}
	\label{t:extrapolation_R_sectoriality}
	Let Assumption \ref{ass:standing} be satisfied and $A,D$ be $\Rsec$-sectorial. Fix $\theta\in (\angR(A)\vee \angR(D),\pi)$, and $a\in (1,\infty)$. 
	Assume that for one $j\in \{1,2\}$  and for each $\lambda\in \complement\overline{\Sigma_{\theta}}$, $M_j(\lambda)$ %defined in \eqref{eq:def_M_1_M_2} 
	is boundedly invertible, and  
	\begin{equation}
		\label{eq:uniform_bound_M}
		\sup_{\lambda\in  \complement\overline{\Sigma_{\theta}}}\|M_j(\lambda)^{-1}\|_{\calL(X_j)}<\infty.
	\end{equation}
	Then there exists $\xi\in (\angR(A)\vee \angR(D),\theta)$ for which the following hold.
	\begin{enumerate}[{\rm(1)}]
		\item\label{it:extrapolation_sectoriality_small_sector} $M_j(\lambda)$ is invertible for all $\lambda\in \complement\overline{\Sigma_{\xi}}$ and  $\sup_{\lambda\in  \complement\overline{\Sigma_{\xi}}}\|M_j(\lambda)^{-1}\|_{\calL(X_j)}<\infty$.
		%\eqref{eq:uniform_bound_M} holds for $\theta$ replaced by $\xi$;
		\item\label{it:extrapolation_R_sectoriality_small_sector} $\Rsec(M_j(\lambda)^{-1}\colon\lambda\in \complement\overline{\Sigma_{\xi}})<\infty$ provided
		\begin{equation}
			\label{eq:sequence_R_bound}
			\sup_{\lambda\in \ell_{\theta}}\Rsec\big(M_j(a^{k} \lambda)^{-1}\colon k\in \Z\big)<\infty.
		\end{equation}
	\end{enumerate}
\end{proposition}

Note that if \eqref{eq:sequence_R_bound} holds for one $j\in \{1,2\}$, \eqref{eq:equivalence_norm_A_0_A} holds, and $\overline{\Ran(\A)}=X$, then Theorem \ref{t:rsec_necessary_sufficient_condition_II} and Proposition~\ref{t:extrapolation_R_sectoriality}
ensure that $\A$ is $\Rsec$-sectorial of angle $<\theta$. %. As in Theorem \ref{t:rsec_necessary_sufficient_condition_II}, 
Here the condition $\overline{\Ran(\A)}=X$ is redundant if $X$ is reflexive, and if $0\in \rho(\A)$ then one can also remove the condition \eqref{eq:equivalence_norm_A_0_A}. 
%\begin{remark}
	In applications one typically checks the  stronger condition
	\begin{align*}
		\Rsec(M_j(\lambda)^{-1}\colon\lambda\in \ell_{\theta})<\infty
	\end{align*}
	instead of  \eqref{eq:sequence_R_bound}.
	Thus to prove $\Rsec$-sectoriality of $\A$ one needs to show $\Rsec$-bounds on $\ell_{\theta}$ instead of the much larger set $ \complement\overline{\Sigma_{\theta}}$.
	In applications to parabolic problems, the choice $\theta=\frac{\pi}{2}$ is particularly handy. 
%\end{remark}

\begin{proof}[Proof of Proposition \ref{t:extrapolation_R_sectoriality}]
	%Let $a,\theta$ be as in the statement of  Theorem \ref{t:extrapolation_R_sectoriality}, and 
	Fix $j\in \{1,2\}$. Up to the choice of a smaller $\xi$, \eqref{it:extrapolation_R_sectoriality_small_sector} follows from \eqref{it:extrapolation_sectoriality_small_sector} and \cite[Proposition 8.5.8(2)]{Analysis2} applies (up to a rotation) to the holomorphic function $M_j^{-1}:\complement\overline{\Sigma_{\xi}}\to \calL(X_j)$. Therefore, it remains to prove \eqref{it:extrapolation_sectoriality_small_sector}.

	Let $\psi\in (\theta,\pi)$. Due to our assumptions and  \eqref{eq:uniform_bound_M}, it remains to show that  there exists $\varepsilon>0$ independent of $\psi$ such that $M_j(\lambda)$ is invertible for all 
	\begin{align*}
		\lambda\in L_{\psi}(\varepsilon), \quad \hbox{where } L_{\psi}(\varepsilon)\stackrel{{\rm def}}{=}
		\{z\in \C\setminus\{0\}: |\arg(z)-\psi|<\varepsilon\}.
	\end{align*}
	To prove the above claim, we employ a Neumann series argument. %For the reader's convenience we include some details. 
	To this end, let $\delta\in (0,\psi-\angR(A)\vee \angR(D))$. Note that the rotation map 
	\begin{align*}
		\Psi_{\delta}\colon L_{\psi}(\delta) \rightarrow \Sigma_{\delta}\cup (-\Sigma_{\delta}), \quad \lambda\mapsto e^{- i \psi}\lambda
	\end{align*}
	is a bi-holomorphism.
	%	$\Psi_{\delta}$ given by $\lambda\mapsto e^{- i \psi}\lambda$ is an bi-holomorphism between $L_{\psi}(\delta)$ and $\Sigma_{\delta}\cup (-\Sigma_{\delta})$. 
	Below we prove the claim for $L_{\psi}(\delta)$ replaced by $L_{\psi}^{\pm}(\delta)=\Psi_{\delta}^{-1}(\pm\Sigma_{\delta})$, the general case follows similarly.
	For notational convenience, we let $M_{1}(\lambda)=\id-T_{e^{-i\psi}\lambda}$ on $L^+_{\psi}(\delta)$ where   
	$$
	T_{\lambda}: \Sigma_{\delta}\to \Sigma_{\delta} \ \ \text{ with }\ \  
	T_{\lambda} \stackrel{{\rm def}}{=} B (\lambda e^{i\psi}-D)^{-1}C(\lambda e^{ i\psi}-A)^{-1}.
	$$ 
	Next we prove that $\id-T_{\lambda}$ is invertible for all $\lambda\in \Sigma_{2\varepsilon}$ where $\varepsilon>0$ is  independent of $\psi$. Fix $\lambda\in \Sigma_{2\varepsilon}$ and let $t=\Re \lambda$. Let $c_{\theta}$ be the supremum in \eqref{eq:uniform_bound_M}. If 
	\begin{equation}
		\label{eq:T_t_T_lambda_smallness}
		\|T_t-T_{\lambda}\|_{\calL(X_1)}<c_{\theta}^{-1},
	\end{equation}
	then we can write
	\begin{align*}
		(\id-T_{\lambda })^{-1}
		&=(\id-T_{t})^{-1}(\id+(\id-T_{t})^{-1}(T_{t}-T_{\lambda}))^{-1}\\
		&=(\id-T_{t})^{-1}\sum_{k \geq 0}(-1)^k(\id-T_{t})^{-k}(T_{t}-T_{\lambda})^k.
	\end{align*}
	To check \eqref{eq:T_t_T_lambda_smallness}, one can argue as follows. By the assumption in \eqref{eq:uniform_bound_M} one has
	 $\|(\id-T_{\lambda})^{-1}\|_{\calL(X_1)}\leq c_{\theta}$ and, for each $\lambda,\mu\in \complement\overline{\Sigma_{\psi}}$,
	\begin{align*}
		&B(\lambda-D)^{-1}C(\lambda-A)^{-1}
		-
		B(\mu-D)^{-1}C(\mu-A)^{-1}\\
		&  
		=B(\lambda-D)^{-1}C[(\lambda-A)^{-1}-(\mu-A)^{-1}]
		- 
		B[(\lambda-D)^{-1}-(\mu-D)^{-1}]C(\mu-A)^{-1}\\
		&  
		=\frac{\lambda-\mu}{\mu}\Big[ B(\lambda-D)^{-1}C(\lambda-A)^{-1}[\mu(\mu-A)^{-1}]
		\\
		&\qquad \qquad\qquad \qquad- 
		B(\lambda-D)^{-1}[\mu(\mu-D)^{-1}]C(\mu-A)^{-1}\Big],
	\end{align*}
	where we used the resolvent identity. Applying the previous with $\mu=t$ and using sectoriality of $A$ and $D$ as well as Assumption \ref{ass:standing} for $L=0$ one gets
	$
	\|T_{t}-T_{\lambda}\|_{\calL(X_1)}\leq K \frac{|t-\lambda|}{t}  
	$
	where $K$ depends only on the sectoriality constants of $A,D$ on $\complement\overline{\Sigma_{\theta}}$. Since $\frac{|t-\lambda|}{t}   \leq \tan (2\varepsilon)$, \eqref{eq:T_t_T_lambda_smallness} follows by choosing $\varepsilon \stackrel{{\rm def}}{=}\frac{1}{2}\arcsin (\frac{c_{\theta}}{K})$ which is independent of $\psi$. 
\end{proof}

\subsection{Perturbative type results for block operators}\label{subsec:perturbative_Rsec}
In this subsection, as a consequence of Theorem~\ref{t:rsec_necessary_sufficient_condition_II} and Corollary \ref{t:rsec_necessary_sufficient_condition}, employing perturbative arguments we show sectoriality and $\Rsec$-sectoriality of block operator matrices if the coupling is small, this holds in particular if $C$ is small enough, while $B$ can be \emph{large}.
Therefore, this goes beyond the
standard perturbation theory. 

\begin{proposition}[Sectoriality and $\Rsec$-sectoriality for small couplings]
	\label{prop:pert_block_operators}
	Let Assumption \ref{ass:standing} be satisfied with $L=0$. 
	\begin{enumerate}[{\rm(a)}]
		\item\label{it:A_sectorial_smallC} Let $A,D$ be sectorial operators, $\psi\in (\omega(A)\vee\omega(D),\pi)$,
		 and assume that 
		\begin{align*}
		%	\label{eq:assumption_epsilon_small_S}
		\sup\{\norm{B (\lambda-D)^{-1}C(\lambda-A)^{-1}}\colon \lambda\in \C\setminus \Sigma_{\psi}\}&<1 \hbox{ or }\\
			\sup\{\norm{C(\lambda-A)^{-1}B (\lambda-D)^{-1}}\colon \lambda\in \C\setminus \Sigma_{\psi}\}&<1.
		\end{align*}
		Then $\A$ is sectorial on $X$ of angle $\omega(\A)\leq \psi$.
		\item\label{it:A_R_sectorial_smallC} Let $A,D$ be $\Rsec$-sectorial operator, $\psi\in (\angR(A)\vee\angR(D),\pi)$, and   
		\begin{align*}
		%	\label{eq:assumption_epsilon_small_R}
		\Rsec(B (\lambda-D)^{-1}C(\lambda-A)^{-1}\colon \lambda\in \C\setminus \Sigma_{\psi})&<1 \hbox{ or }\\
		\Rsec(C(\lambda-A)^{-1}B (\lambda-D)^{-1}\colon \lambda\in \C\setminus \Sigma_{\psi})&<1
		%c_A <\frac{1}{c_D \non^\Rsec_{\psi}(A)\non_{\psi}^\Rsec(D)}. 
		\end{align*}
		Then $\A$ is $\Rsec$-sectorial on $X$ of angle $\angR(\A)\leq \psi$.
	\end{enumerate}
\end{proposition}

\begin{proof}%[Proof of Proposition \ref{prop:pert_block_operators}]
	To prove the claim in \eqref{it:A_R_sectorial_smallC} we check the condition in Theorem~\ref{t:rsec_necessary_sufficient_condition_II}~\eqref{it:schur_compleements_R_bounded}.   %\ref{t:rsec_necessary_sufficient_condition}~\eqref{it:schur_compleements_R_bounded}. 
	First note that by Proposition~\ref{l:one_side_A_0_A_estimate} condition~\eqref{eq:equivalence_norm_A_0_A} holds and by  Remark~\ref{rem:RanA_dense} $\Ran(\A)\subseteq X$ is dense.
	%Let $\psi\in (\angR(A)\vee \angR(D),\pi)$ be as in the statement. 
	
	We provide the required estimate for $M_1(\cdot)$, the one for $M_2(\cdot)$ being similar. 
	%Since $\varepsilon$ satisfies \eqref{eq:assumption_epsilon_small}, we have
	Having
	\begin{equation}
	\label{eq:R_boundedness_max_reg_less_one}
	\Rsec(B (\lambda-D)^{-1}C(\lambda-A)^{-1}\colon \lambda\in \C\setminus \Sigma_{\psi})<1,
	\end{equation}
	%	and in particular 
	%	$$\|B (\lambda-D)^{-1}C(\lambda-A)^{-1}\|_{\calL(X_1)}<1 \quad \hbox{for all }\lambda\in \C\setminus \Sigma_{\psi}.$$ 
	%	Then by 
	by a Neumann series argument one obtains that
	$$
	M_1(\lambda)^{-1}=\sum_{n\geq 0} [B (\lambda-D)^{-1}C(\lambda-A)^{-1}]^n 
	\quad \text{ for all }\lambda\in \C\setminus \Sigma_{\psi},
	$$
	where the series converges absolutely in $\calL(X_1)$. The previous expression implies
	$$
	\Rsec(M_1(\lambda)^{-1}\colon\lambda\in \C\setminus \Sigma_{\psi})\leq
	\sum_{n\geq 0}\Big[ \Rsec(B (\lambda-D)^{-1}C(\lambda-A)^{-1}\colon \lambda\in \C\setminus \Sigma_{\psi})\Big]^n <\infty,
	$$
	where in the last inequality we have used \eqref{eq:R_boundedness_max_reg_less_one}. The claim in \eqref{it:A_sectorial_smallC} follows analogously replacing $\Rsec$-bounds by norm-bounds.
\end{proof}
As usual, the condition $L=0$ in Proposition \ref{prop:pert_block_operators} can be removed up to a shift.  
For a sectorial operator $T$ on a Banach space $X$ we set
\begin{equation*}
\non_{\psi}^{\Ssec}(T)\stackrel{{\rm def}}{=}\sup\{\norm{T (\lambda-T)^{-1}}\colon \lambda\in  \complement\overline{\Sigma_{\psi}}\} \quad \text{for}\quad\psi>\omega(T),
\end{equation*}
and for an $\Rsec$-sectorial operator $T$
\begin{equation}
\label{eq:def_non_R_sec}
\non_{\psi}^\Rsec(T)\stackrel{{\rm def}}{=}\Rsec\big(T (\lambda-T)^{-1}\colon \lambda\in  \complement\overline{\Sigma_{\psi}}\,\big)
\quad \text{for}\quad\psi>\angR(T).
\end{equation}

\begin{corollary}[Sectoriality and $\Rsec$-sectoriality for small $C$]
	\label{cor:pert_block_operators}
	Let Assumption \ref{ass:standing} be satisfied, and let $c_A$ and $c_D$ be the relative bounds in Assumption~\ref{ass:standing}.
	\begin{enumerate}[{\rm(a)}]
		\item\label{cor:pert_block_operators_a} If $A$ and $D$ are sectorial, 
		then for any $\psi\in (\omega(A)\vee\omega(D),\pi)$ there are
		$$\varepsilon_0(c_D ,\non^{\Ssec}_{\psi}(A),\non^{\Ssec}_{\psi}(D))>0\ \  \text{ and } \ \ 
		\nu_0(c_D ,\non_{\psi}^{\Ssec}(A),\non^{\Ssec}_{\psi}(D),L)>0$$ 
		such that if $\nu> \nu_0$ and $c_A<\varepsilon_0$, then 
		$\nu + \A$ is sectorial on $X$ of angle $\leq  \psi$. % provided $\nu> \nu_0$ and $c_D<\varepsilon_0$.
		\item\label{cor:pert_block_operators_b} If $A$ and $D$ are $\Rsec$-sectorial, 
		then for any $\psi\in (\angR(A)\vee\angR(D),\pi)$ there are
		$$\varepsilon_0(c_D ,\non^{\Rsec}_{\psi}(A),\non^{\Rsec}_{\psi}(D))>0\ \  \text{ and } \ \ 
		\nu_0(c_D ,\non^{\Rsec}_{\psi}(A),\non^{\Rsec}_{\psi}(D),L)>0$$ 
		such that if $\nu> \nu_0$ and $c_A<\varepsilon_0$, then $\nu + \A$ is $\Rsec$-sectorial on $X$ of angle $\leq  \psi$.
	\end{enumerate}	
In particular, if $C\in \calL(\Do(A^{\gamma}),X_2)$ for some  $\gamma\in (0,1)$, %. Indeed, under this assumption by the Young and moment inequality (see e.g.\ \cite[Theorem 3.3.5]{pruss2016moving}) 
  then there are $\varepsilon_0>0$ and $\nu_0>0$ such that 
in the situations of \eqref{cor:pert_block_operators_a} and \eqref{cor:pert_block_operators_b} the conditions on $c_A$ are  satisfied. 			
\end{corollary}

\begin{remark}\label{rem:small_cA}
		Let us stress that $\varepsilon_0$ does not depend on $L>0$ which will become clear from \eqref{eq:var_nu_0} below. The conditions in  Proposition~\ref{prop:pert_block_operators}~\eqref{it:A_sectorial_smallC} and \eqref{it:A_R_sectorial_smallC} hold provided that
	\begin{equation}
	\label{eq:perturbation_L_0_simplified_condition}
	c_A <\frac{1}{c_D \non^\Ssec_{\psi}(A)\non^\Ssec_{\psi}(D)} \quad \hbox{and} \quad
	c_A <\frac{1}{c_D \non^\Rsec_{\psi}(A)\non_{\psi}^\Rsec(D)},
	\end{equation}
	respectively, that is if $L=0$, then $\nu=0$, and the above estimates give lower bounds on $\varepsilon_0$.	
	 However, the more general assumptions in Proposition~\ref{prop:pert_block_operators} as compared to Corollary~\ref{cor:pert_block_operators} are useful as well. For instance consider a  diagonally dominant block operator $\A$ on $X=X_1\times X_2$ with $X_2=X_{2}^1\times X_{2}^2$, $B_{12}\colon \Do(D_{22}) \subseteq X_{2}^2\rightarrow X_1$ and $C_{21}\colon \Do(A) \subseteq X_{1}\rightarrow X_2^1$ of the form
	\begin{align*}
	\A =
	\begin{bmatrix}
	\begin{array}{c|cc}
	A & 0 & B_{12} \\
	\hline
	C_{21} & D_{11} & 0 \\
	0 & 0 & D_{22} \\
	\end{array}
	\end{bmatrix}  = \begin{bmatrix}
	A & B \\ C & D
	\end{bmatrix}. %\quad \hbox{with } \Do(\A)=H^2(\R^d)\times H^2(\R^d)^2,
	\end{align*}
	Then $c_A$ can be larger than $\varepsilon_0$, but for $\lambda \in \rho(\diag)$
	\begin{align*}
	B (\lambda-D)^{-1}C(\lambda-A)^{-1} &= 
	\begin{bmatrix}
	0 & B_{12}
	\end{bmatrix}
	\begin{bmatrix}
	(\lambda-D_{11})^{-1} & 0 \\
	0 & (\lambda-D_{22})^{-1}
	\end{bmatrix}
	\begin{bmatrix}
	C_{21} \\ 0 
	\end{bmatrix} (\lambda-A_{11})^{-1}\\
	&=0.
	\end{align*} 
\end{remark}

\begin{proof}[Proof of Corollary \ref{cor:pert_block_operators}]
	The idea is to apply Proposition \ref{prop:pert_block_operators} by checking the condition \eqref{eq:perturbation_L_0_simplified_condition} with $A$ and $D$ replaced by $\nu+A$ and $\nu+D$. 
	Let us begin by noticing that, for all $\nu>0$, 
	\begin{equation}
	\label{eq:monotonicity_N_non}
	\non_{\psi}^{\ast}(\nu+ A)\leq \non^{\ast}_{\psi}(A)\quad \text{ and }\quad
	\non_{\psi}^{\ast}(\nu+ D)\leq \non_{\psi}^{\ast}(D), \quad \ast\in \{\Ssec,\Rsec\},
	\end{equation}
	and by Assumption~\ref{ass:standing} for all $y\in X_2$
	\begin{equation*}
	\begin{aligned}
	\|B (\nu+D)^{-1} y\|_{X_1}
	\leq C_B \|D(\nu+D)^{-1} y\|_{X_2} + L\|(\nu+D)^{-1}y\|_{X_1}.
	\end{aligned}
	\end{equation*}
	In particular, for all $\nu>0$ and $x\in \Do(D)$, there is a $y\in X_2$ with $x=(\nu+D)^{-1} y$ and hence
	\begin{equation}
	\label{eq:estimate_B_relatively_small}
	\|B x\|_{X_1}\leq \Big(C_B \non^{\ast}_{\psi}(D) + (1+\non^{\ast}_{\psi}(D)) \frac{L}{\nu}\Big)\|(\nu+D)x\|_{X_2}.
	\end{equation}
	Next, set $$\nu_0'\stackrel{{\rm def}}{=}\frac{L}{C_B}\Big(1+\frac{1}{\non_\psi^\ast(D)}\Big),$$ and note that the constant on the right hand side of \eqref{eq:estimate_B_relatively_small} is less than  $2C_B \non_{\psi}^\ast(D)$ provided $\nu\geq \nu_0'$. Analogously, for all $\nu>0$ and $x\in \Do(A)$, 
	$$
	\|C  x\|_{X_2}
	\leq \Big(\varepsilon \non^\ast_{\psi}(A) + (1+\non^\ast_{\psi}(A)) \frac{L}{\nu}\Big)\|(\nu+A)x\|_{X_1}.
	$$
	By \eqref{eq:monotonicity_N_non}, \eqref{eq:estimate_B_relatively_small} and Proposition  \ref{prop:pert_block_operators}, the claim follows provided $\nu_0\geq \nu_0'$ and $\varepsilon_0>0$ satisfy  
	$$
	\varepsilon_0\non_{\psi}^\ast(A) +   (1+\non_{\psi}^\ast(A))\frac{L}{\nu_0} \leq 
	\frac{1}{2 C_B \non_{\psi}^\ast(D) \non_{\psi}^\ast(A)\non_{\psi}^\ast(D)}\stackrel{{\rm def}}{=} R ,
	$$
	%Since $\non_{\psi}^\ast(A)\leq \non_{\psi}(A) $ and $\non_{0}(D)\leq \non_{\psi}(D) $, 
	and a possible choice is given by
	\begin{equation}
	\label{eq:var_nu_0}
	\varepsilon_0 = \frac{R}{\non_{\psi}^\ast(A)}, \ \ \text{ and }\ \ 
	\nu_0=\nu_0'\vee \Big(\frac{L}{R}(1+\non_{\psi}^\ast(A))\Big).
	\end{equation}
	
	It remains to prove the last statement.  Thus, we assume that  $C\in \calL(\Do(A^{\gamma}),X_2)$ for some $\gamma\in (0,1)$. Under this assumption by the Young and moment inequality (see e.g.\ \cite[Theorem 3.3.5]{pruss2016moving}) for any $\varepsilon>0$ there is an $C_{\varepsilon,\gamma}>0$ such that 
	$$
	\|Cx\|_{X_2}\leq \varepsilon \|A x\|_{X_1}+ C_{\varepsilon,\gamma} \|x\|_{X_1} \quad\text{for all }x\in \Do(A),
	$$
	and in particular there exists $\varepsilon_0>0$ such that $c_D<\varepsilon_0$.
\end{proof}

\subsection{Proofs of Theorem \ref{t:rsec_necessary_sufficient_condition_II}, Corollary \ref{t:rsec_necessary_sufficient_condition}, and Proposition \ref{l:one_side_A_0_A_estimate}}
\label{ss:proofs_sectoriality}

\begin{proof}[Proof of Theorem \ref{t:rsec_necessary_sufficient_condition_II}]
	For part \eqref{it:A_R_sectorial_equivalence}, let $\psi\in (\angR(A)\vee \angR(D),\pi)$ be fixed.

	\eqref{it:A_R_sectorial}$\Rightarrow$\eqref{it:schur_compleements_R_bounded}: Fix $\phi>\psi$. Since $\phi>\angR(A)\vee\angR(D)$, we have $\rho(A)\cap \rho(D)\supseteq \complement\overline{\Sigma_{\phi}}$. By Proposition~\ref{l:relation_resolvent_AD_calA} and our assumptions, we have that the representations in Proposition \ref{l:relation_resolvent_AD_calA}\eqref{eq:resolvent_formula_block_matrix_I} hold for all $\lambda\in \complement\overline{\Sigma_{\phi}}$.
	We claim that 
	\begin{equation}
	\label{eq:R_sectoriality_Minverse}
	\Rsec(\M(\lambda)^{-1}\colon\lambda\in  \complement\overline{\Sigma_{\phi}})<\infty 
	\end{equation}
	where $\M(\lambda)$ is defined in \eqref{eq:def_M_A_0}.
	%\begin{equation}
	%\label{eq:def_Mone}
	%\Mone(\lambda)\stackrel{{\rm def}}{=}
	%\begin{bmatrix}
	%M_1(\lambda)^{-1} & M_1(\lambda)^{-1}B(\lambda-D)^{-1}  \\
	%M_2(\lambda)^{-1} C(\lambda-A)^{-1} & M_2(\lambda)^{-1}  
	%\end{bmatrix}.
	%\end{equation}
	Note that \eqref{eq:R_sectoriality_Minverse} is in fact stronger than \eqref{it:schur_compleements_R_bounded}.
	
	Next we prove \eqref{eq:R_sectoriality_Minverse}. Note that by \eqref{eq:def_M_A_0}%the representation in Proposition \ref{l:relation_resolvent_AD_calA}~\eqref{eq:resolvent_formula_block_matrix_I} can be rewritten as 
	\begin{equation}
	\label{eq:Mone_identity}
	\M(\lambda)^{-1}=(\lambda-\diag)(\lambda-\A)^{-1}  \ \text{ for all }\lambda\in \complement\overline{\Sigma_{\phi}}.
	\end{equation}
	Recall that $\diag$ is injective since $A,D$ (and thus $\diag$) are sectorial operators. By \eqref{eq:equivalence_norm_A_0_A} and $\Do(\diag)=\Do(\A)$, it follows that $\A$ is injective as well.
	In particular, $\A^{-1}:\Ran(A)\to \Do(A)$ is well-defined and $\|\diag\A^{-1} x\|_{X}\lesssim \|x\|_{X}$ for all $x\in \Ran(\A)$. By $\overline{\Ran(A)}=X$ we infer
	$$
	\overline{\diag\A^{-1}}\stackrel{{\rm def}}{=} K \in \calL(X)  \ \ \text{ and } \ \ \diag=K \A\text{ on }\Do(\A).
	$$
	Combining the latter with \eqref{eq:Mone_identity},
	$$
	\M(\lambda)^{-1}= \lambda(\lambda-\A)^{-1} - K \A (\lambda-\A)^{-1} \ \text{ for all }\lambda\in \complement\overline{\Sigma_{\phi}}.
	$$
	By \eqref{it:A_R_sectorial} and the previous identity, we get \eqref{eq:R_sectoriality_Minverse}, and by Proposition \ref{l:relation_resolvent_AD_calA}~\eqref{eq:resolvent_formula_block_matrix_I} the claim follows, where one uses that for $\lambda_0\in \rho(A) \cap \rho(D)$
	\begin{align*}
	\Rsec (C(\lambda-A)^{-1}\colon \lambda\in \complement\overline{\Sigma_{\phi}}) &\leq \norm{C(\lambda_0-A)^{-1}} \Rsec \big((\lambda_0-A)(\lambda-A)^{-1}\colon \lambda\in \complement\overline{\Sigma_{\phi}}\,\big)<\infty, \\
	\Rsec (B(\lambda-D)^{-1}\colon \lambda\in \complement\overline{\Sigma_{\phi}})&\leq \norm{B(\lambda_0-D)^{-1}} \Rsec \big((\lambda_0-D)(\lambda-D)^{-1}\colon \lambda\in \complement\overline{\Sigma_{\phi}}\,\big)<\infty.
	\end{align*}

	\eqref{it:schur_compleements_R_bounded}$\Rightarrow$
	\eqref{it:A_R_sectorial}: Let $\phi>\psi$. %and $\Mone$ be as in \eqref{eq:def_Mone}. 
	Note that $\rho(\diag)\supseteq \complement\overline{\Sigma_{\phi}}$ and one has the representations in Proposition~\ref{l:relation_resolvent_AD_calA}~\eqref{eq:resolvent_formula_block_matrix_I}. Due to the $\Rsec$-sectoriality of $A,D$ and the choice of $\phi$, it remains to prove that $\Rsec(\M(\lambda)^{-1}\colon\lambda\in  \complement\overline{\Sigma_{\phi}})<\infty 
	$. By \eqref{it:schur_compleements_R_bounded}, the latter holds provided
	$$
	\Rsec(B(\lambda-D)^{-1}\colon\lambda\in  \complement\overline{\Sigma_{\phi}})<\infty  \ \ \text{ and } \ \ 
	\Rsec(C(\lambda-A)^{-1}\colon\lambda\in  \complement\overline{\Sigma_{\phi}})<\infty .
	$$
	Again, these bounds follow from the $\Rsec$-sectoriality of $A$ and $D$, the choice of $\phi$ and Assumption \ref{ass:standing} with $L=0$.
	
	Part~\eqref{it:A_sectorial_a} follows analogously, replacing $\Rsec$-bounds by norm-bounds.
	
	Proof of part \eqref{it:X_reflexive}: If $X$ is reflexive, then $X=\mathsf{N}(T)\oplus \overline{\mathsf{R}(T)}$, for any (pseudo-) sectorial operator $T$ (see \cite[Proposition 10.1.9]{Analysis2}). Reasoning as in the implication 
	\eqref{it:A_R_sectorial}$\Rightarrow$\eqref{it:schur_compleements_R_bounded}, \eqref{eq:equivalence_norm_A_0_A} yields $\mathsf{N}(\A)=\{0\}$ and therefore  $\overline{\mathsf{R}(\A)}=X$.
\end{proof}

\begin{proof}[Proof of Corollary \ref{t:rsec_necessary_sufficient_condition}]
	The claim follows by Theorem \ref{t:rsec_necessary_sufficient_condition_II}, noticing that if $0\in \rho(\A)$ then \eqref{eq:equivalence_norm_A_0_A} follows from $\Do(\A)=\Do(\diag)$ and that Assumption \ref{ass:standing} with $L=0$ holds since $0\in \rho(A)\cap \rho(D)$. 
\end{proof}

%It remains to prove 

\begin{proof}[Proof of Proposition \ref{l:one_side_A_0_A_estimate}]
This proof %Proposition~\ref{l:one_side_A_0_A_estimate}  is 
resembles the one of Proposition \ref{l:relation_resolvent_AD_calA}.	The operators $H$ and $G$ are well-defined and bounded by Remark~\ref{rem:ass_standing}.
	By \eqref{eq:def_GH_lemma}, we have
	\begin{equation}
	\label{eq:A_0_A_G_relation_lemma_proof}
	\A =\overline{\M(0)} \diag,
	\ \ \text{ where }\ \ 
	\overline{\M(0)}=
	\begin{bmatrix}
	\id & H\\
	G & \id
	\end{bmatrix}
	\end{equation}
as in Remark~\ref{rem:RanA_dense}.
	To fix the idea, we assume that $\id-HG$ is invertible. 
	Reasoning as in the proof of Proposition \ref{l:relation_resolvent_AD_calA}, one can check that $\overline{\M(0)}$ is invertible with inverse given by
	$$
	\overline{\M(0)}^{\,-1}=\begin{bmatrix}
	(\id - HG )^{-1} &0 \\
	- G(\id-HG)^{-1} & \id
	\end{bmatrix}
	\begin{bmatrix}
	\id &- H\\
	0 & \id
	\end{bmatrix}
	\in \calL(X).
	$$
	Thus \eqref{eq:A_0_A_G_relation_lemma_proof} gives $\diag=\overline{\M(0)}^{\,-1} \A$ and therefore $\|\diag x\|_{X}\leq \| \overline{\M(0)}^{\,-1}\|_{\calL(X)} \|\A x\|_{X}$ for $x\in \Do(\diag)$, as desired.
	
	It remains to prove the last assertion, where it suffices to show in the first case that $\sup\{\norm{B(t+D)^{-1}C(t+A)^{-1}}\colon t\in (0,\varepsilon)\}<1$ implies $\| H G \|_{\calL(X_1)}< 1$. 
	This follows since for all $x\in X$  $$\lim_{t\downarrow 0} B(t+D)^{-1}C(t+A)^{-1} x = \lim_{t \downarrow 0} H D(t+D)^{-1}G A(t+A)^{-1} x =HG x$$ by \cite[Proposition 10.1.7 (2)]{Analysis2}, and %the claim follows, where 
the other case follows similarly.
\end{proof}

\section{$H^{\infty}$-calculus for block operator matrices}\label{sec:Hinfty}
In this section %we have a look at the $H^{\infty}$-calculus for block operator matrices and 
we give some sufficient condition to check the boundedness of the $H^{\infty}$-calculus for block operator matrices $\A$. %Further conditions will be investigated in the next section.
These results %of this section 
will be formulated using %the following

%Test~\nameref{ass:51_+}
\begin{assumption}%[$\pm$]
	\label{ass:fractions_pm}
	Let Assumption \ref{ass:standing} be satisfied. Suppose that $A$ and $D$ are sectorial operators.
	%Assume that $A$ and $D$ have bounded $H^{\infty}$-calculus.
	\begin{description}%[{\rm(1)}]
		\item[{\rm$(+)$}\label{ass:fractions_+}] We say that Assumption \ref{ass:fractions_pm}\nameref{ass:fractions_+} holds if  there exists $\delta\in (0,1)$ such that%, 
%		for all $(x,y)\in \Do(A)\times \Do(D)$,
		\begin{equation*}
		\begin{aligned}
		C(\Do(A^{1+\delta}))&\subseteq \Do(D^{\delta}) \ \ \text{ and } \ \ &
		\|D^{\delta}C x\|_{X_2}\lesssim \| A^{1+\delta} x\|_{X_1}& \quad \hbox{for all } x\in \Do(A^{1+\delta}),
		\\
		B(\Do(D^{1+\delta}))&\subseteq \Do(A^{\delta}) \ \ \text{ and } \ \ &
		\|A^{\delta}B y\|_{X_1}\lesssim \| D^{1+\delta} y\|_{X_2} & \quad \hbox{for all } y\in \Do(D^{1+\delta}).
		\end{aligned}
		\end{equation*}
		\item[{\rm$(-)$}\label{ass:fractions_-}] We say that Assumption \ref{ass:fractions_pm}\nameref{ass:fractions_-} holds if there exists $\delta\in (0,1)$ such that%, 
%		for all $(x,y)\in \Do(A)\times \Do(D)$,
		\begin{equation*}
		\begin{aligned}
		\ \ \ \  \qquad
		\Ran(C)&\subseteq \Ran(D^{\delta}) \ \ \text{ and } \ \ &
		\|D^{-\delta}C x\|_{X_2}\lesssim \| A^{1-\delta} x\|_{X_1}&\quad \hbox{for all } x\in \Do(A), 
		\\
		\ \  \ \ \qquad
		\Ran(B)&\subseteq \Ran(A^{\delta}) \ \ \text{ and } \ \ &
		\|A^{-\delta}B y\|_{X_1}\lesssim \| D^{1-\delta} y\|_{X_2}& \quad \hbox{for all } y\in \Do(D).
		\end{aligned}
		\end{equation*}
	\end{description}
\end{assumption}
\begin{remark}
	If $0\in \rho(A)$, then $\Ran(A^{\delta})=X_1$ for $\delta>0$. Thus the condition $\Ran(B)\subseteq \Ran(A^{\delta}) $ in Assumption \ref{ass:fractions_pm}\nameref{ass:fractions_-} becomes redundant in this case. A similar consideration holds for $\Ran(C)\subseteq \Ran(D^{\delta})$ if $0\in \rho(D)$.
\end{remark}
Having in mind applications to differential operators, Assumption~\ref{ass:standing} implies a relation between the orders of $C$ and $A$, and the orders of $B$ and $D$. Assumptions~\ref{ass:fractions_pm}$(\pm)$ now impose additional relations between the orders of $C$, $D$ and $A$, and the orders of $B$, $A$ and $D$. This is illustrated by the following
\begin{example}
	Consider
	in  $L^p(\R^d)\times L^p(\R^d)$ for $p\in (1,\infty)$ the operator 
	\begin{align*}
	\A_\alpha= \begin{bmatrix}
	-\Delta+\id & -(-\Delta+\id)^\alpha \\ -\Delta+\id &\ \  (-\Delta+\id)^\alpha
	\end{bmatrix} \quad  \hbox{ with } 
	\Do(\A_\alpha)=H^{2,p}(\R^d)\times H^{2\alpha,p}(\R^d), \quad \alpha >0.
	\end{align*}
	It is easy to see that  $\A_\alpha$	is diagonally dominant and satisfies Assumption~\ref{ass:standing} for all $\alpha>0$. However,  in the case \nameref{ass:fractions_+}, using that $\Do(A^\gamma)=H^{2\gamma,p}(\R^d)$ and $\Do(D^\gamma)=H^{2\alpha\gamma,p}(\R^d)$ for $\gamma\geq 0$,
	\begin{align*}
	C(\Do(A^{1+\delta}))&=H^{2\delta,p}(\R^d) \subseteq \Do(D^{\delta}) = H^{2\alpha\delta,p}(\R^d)  &\hbox{only if } \alpha\geq  1,  \\
	B(\Do(D^{1+\delta}))&=H^{2\alpha\delta,p}(\R^d) \subseteq \Do(A^{\delta}) = H^{2\delta,p}(\R^d)   &\hbox{only if } \alpha\leq  1,
	\end{align*}
	and hence the inclusions Assumption~\ref{ass:fractions_pm}\nameref{ass:fractions_+} holds only for $\alpha=1$. 
	The same argument also proves that Assumption~\ref{ass:fractions_pm}\nameref{ass:fractions_-} holds only for $\alpha=1$. Hence,  Assumption~\ref{ass:fractions_pm} requires that $A$ and $D$ have the same orders. 
\end{example}

We begin by providing a sufficient condition for the boundedness of the $H^{\infty}$-calculus of $\A$. For the notion of the \emph{type} of a space we refer to \cite[Chapter 7]{Analysis2} recapped here in the Appendix~\ref{sec:appendix}. In particular %Recall that 
%any 
$L^p$-spaces with $p\in (1,\infty)$ and their closed subspaces have non-trivial type. Note that spaces of non-trivial type are exactly the $K$-convex spaces by Pisier's theorem, see e.g. \cite[Theorem 7.4.23]{Analysis2} and also \cite[Section 7.4]{Analysis2} for the definition and properties of $K$-convex spaces.

\begin{theorem}[Boundedness of the $H^\infty$-calculus for $\Rsec$-sectorial $\A$]
	\label{t:calculus_bounded_no_smallness_condition}
Suppose that $X_1$ and $X_2$ are reflexive Banach spaces with non-trivial type. 	Let Assumption \ref{ass:standing} with $L=0$, estimate \eqref{eq:equivalence_norm_A_0_A}, and Assumption \ref{ass:fractions_pm}~\nameref{ass:fractions_+}and \nameref{ass:fractions_-} be satisfied. Then the following implication holds:

	If $A$ and $D$ have a bounded $H^{\infty}$-calculus of angle $\angH(A)$ and $\angH(D)$, respectively, and $\A$ is $\Rsec$-sectorial on $X$ with angle $\angR(\A)$,
	then $\A$ has a bounded $H^{\infty}$-calculus of angle $\angH(\A)=\angR(\A)$.
\end{theorem}
\begin{proof}
%The above result follows from 
This follows directly from the more general transference result  Theorem~\ref{t:transference_appendix} -- given in the appendix -- applied with $T=\diag$ and $S=\A$, where the assumption \eqref{eq:mapping_property_AB} translates into the Assumptions~\ref{ass:fractions_pm}$(\pm)$.
\end{proof}

Having only one of the two Assumptions \ref{ass:fractions_pm}~\nameref{ass:fractions_+} and~\nameref{ass:fractions_-}, one can still prove the boundedness of the $H^\infty$-calculus, where 
as in Proposition \ref{prop:pert_block_operators}, we only require the coupling to be small. The proofs of Theorem \ref{t:calculus_pert_block_operators} and Theorem ~\ref{t:calculus_pert_block_operators_notype} will be given in Subsection \ref{ss:proof_pert_H_calculus} below.  %and therefore one cannot deduce the following using standard perturbation arguments.

\begin{theorem}[Boundedness of the $H^\infty$-calculus  for small couplings]
	\label{t:calculus_pert_block_operators} 
	Let Assumption \ref{ass:standing} be satisfied with $L=0$, and assume that  Assumption \ref{ass:fractions_pm}~\nameref{ass:fractions_+} or Assumption \ref{ass:fractions_pm}~\nameref{ass:fractions_-} holds. 
	Let $A $ and $D$ have a bounded $H^{\infty}$-calculus and fix $\psi\in (\angH(A)\vee\angH(D),\pi)$. Assume that  $X_1$ or $X_2$ has non-trivial type and
	\begin{align}
	\label{eq:assumption_epsilon_small_calculus}
	\begin{split}
	c_1^{\Rsec}\stackrel{{\rm def}}{=}\Rsec(B (\lambda-D)^{-1}C(\lambda-A)^{-1}\colon \lambda\in \complement \Sigma_{\psi})&<1/K_{X_1} \hbox { or} \\
	c_2^{\Rsec}\stackrel{{\rm def}}{=}\Rsec(C(\lambda-A)^{-1}B (\lambda-D)^{-1}\colon \lambda\in \complement  \Sigma_{\psi})&<1/K_{X_2},
	\end{split}
	%c_A <\frac{1}{c_D \non_{\psi}(A)\non_{\psi}(D)}. 
	\end{align}
receptively, 
where $K_{X_j}$ are the $K$-convexity constants of $X_j$ for $j\in \{1,2\}$, 
	then $\A$ has a bounded $H^{\infty}$-calculus on $X$ of angle $\angH(\A)< \psi$.
\end{theorem}

The geometric conditions on $X_1,X_2$ can be avoided  assuming a particular smallness condition on $C$ which allows one to interpolate pairs of operators rather than $\Rsec$-bounded sets.
Recall that $\non^{\Rsec}_{\psi}$ is defined in \eqref{eq:def_non_R_sec}.

\begin{theorem}[Boundedness of the $H^\infty$-calculus  in space of trivial type and small $C$]
	\label{t:calculus_pert_block_operators_notype}
	Let Assumption \ref{ass:standing} be satisfied with $L=0$.
	Assume that $A $ and $D$ have a bounded $H^{\infty}$-calculus and that  Assumption \ref{ass:fractions_pm}\nameref{ass:fractions_-} or Assumption \ref{ass:fractions_pm}\nameref{ass:fractions_+} holds. 
	Fix $\psi\in (\angH(A)\vee\angH(D),\pi)$. If for $c_A$ and $c_D$, the relative bounds in Assumption~\ref{ass:standing}, one has
	\begin{equation}
		\label{eq:assumption_epsilon_small_calculus_notype}
		c_A <\frac{1}{c_D \non^{\Rsec}_{\psi}(A)\non^{\Rsec}_{\psi}(D)},
	\end{equation}
	then $\A$ has a bounded $H^{\infty}$-calculus on $X$ of angle $\angH(\A)< \psi$.
\end{theorem}

Again, arguing as in the proof of Corollary \ref{cor:pert_block_operators}, we may allow lower order terms at the expense of a shifting. Note that  Corollary~\ref{cor:calculus_pert_block_operators_lower_order} goes beyond the well-known lower order perturbation theorems for the $H^\infty$-calculus, cf. e.g. \cite[Theorem~2.4]{AHS1994}. 
The results presented in this section admit even perturbations of the same order, and even if $C$ is of lower order $B$ can be of the same order as $A$ and $D$. General perturbations of the same order under additional assumptions on mapping properties in domains of fractional powers of the unperturbed operator are discussed in \cite[Theorem 3.2]{DDHPV}  and \cite[Section 5]{KKW}.

\begin{corollary}[Boundedness of the $H^\infty$-calculus of $\A$ for small $C$]
	\label{cor:calculus_pert_block_operators_lower_order}
	Let Assumption \ref{ass:standing} be satisfied, and 
	assume that  Assumption \ref{ass:fractions_pm}~\nameref{ass:fractions_+} or Assumption \ref{ass:fractions_pm}~\nameref{ass:fractions_-} holds. If $A $ and $D$ have a bounded $H^{\infty}$-calculus, then for any $\psi\in (\angH(A)\vee\angH(D),\pi)$ there exist
	$$
	\varepsilon_0=\varepsilon_0(c_D ,\non_{\psi}(A),\non_{\psi}(D))>0\ \  \text{ and } \ \ 
	\nu_0=\nu_0(c_D ,\non_{\psi}(A),\non_{\psi}(D),L)>0
	$$
	such that if $c_A<\varepsilon_0$ and $\nu>\nu_0$,
	then
	$\nu+\A$ has a bounded $H^{\infty}$-calculus on $X$ of angle $\angH(\A)\leq \psi$. %provided $c_A<\varepsilon_0$ and $\nu>\nu_0$, 
	In particular, if $C\in \calL(\Do(A^{\gamma}),X_2)$ for some  $\gamma\in (0,1)$, %. Indeed, under this assumption by the Young and moment inequality (see e.g.\ \cite[Theorem 3.3.5]{pruss2016moving}) 
	then there exists $\varepsilon_0>0$ and $\nu_0>0$ such that %there is an $C_{\varepsilon,\gamma}>0$ such that 
	%$$
	%\|Cx\|_{X_2}\leq \varepsilon \|A x\|_{X_1}+ C_{\varepsilon,\gamma} \|x\|_{X_1} \quad\text{for all }x\in \Do(A),
	%$$
	%and therefore 
	 the conditions on $c_A$ are  satisfied.
\end{corollary}

\begin{remark}[Assumptions~\ref{ass:fractions_pm}$(\pm)$ are sufficient but not necessary]\label{rem:McIntosh_Yagi}
On the one hand, the	Assumptions~\ref{ass:fractions_pm}$(\pm)$ cannot be avoided in general even for the triangular case with $C=0$. A counterexample for the triangular case is constructed by McIntosh and Yagi in the proof of \cite[Theorem 3]{MY_1990}.
On the other hand, there are block operator matrices with bounded $H^\infty$-calculus which violate both Assumptions~\ref{ass:fractions_pm}$(\pm)$.
	Considering for instance a diagonally dominant $\A$ with $A=\id$, then this  already implies that $C$ is bounded. Now, in Assumptions~\ref{ass:fractions_pm}\nameref{ass:fractions_+} 
	the inclusion $C(\Do(A^{1+\delta}))\subseteq \Do(D^{\delta})$ would be violated for $C$  surjective and $D$ unbounded.
	In \nameref{ass:fractions_-} the estimate $\|A^{-\delta}B y\|_{X_2}\lesssim \| D^{1-\delta} y\|_{X_2}$ would fail if $B$ is of the order of $D$. So, considering for instance
	in  $L^p(\R^d)\times L^p(\R^d)$ for $p\in (1,\infty)$ the operator 
	\begin{align*}
	\A_{\mu}= \begin{bmatrix}
	\mu+\id & \Delta \\ \id & \mu-\Delta
	\end{bmatrix} \quad  \hbox{ with } 
	\Do(\A_{\mu})=L^p(\R^d)\times H^{2,p}(\R^d), \quad \mu>0,
	\end{align*}
	then this is diagonally dominant, but it violates both Assumptions~\ref{ass:fractions_pm}$(\pm)$. Nevertheless, it is shown in Corollary~\ref{prop:Azero} below, that $\A_{\mu}$ has a bounded $H^\infty$-calculus for $\mu>0$ sufficiently large.
\end{remark}

\subsection{Fractional powers}
One of the advantages of the $H^{\infty}$-calculus is that it implies the boundedness of imaginary powers (BIP) and therefore allows for a complete description of  $\Do(\A^{\theta})$ for $\theta\in (0,1)$, cf. e.g. \cite[Sections 2.3 and 2.4]{DHP}. The description of the fractional powers of negative orders is more delicate (although sometimes useful in applications to nonlinear (stochastic) partial differential equations, see e.g.\ \cite{AV19_QSEE_1}). For further results in this direction we  refer to Proposition \ref{prop:fractions_Hilbert}  and Theorem \ref{t:extrapolation_H} below.
%for a result in this direction.

\begin{proposition}[Fractional powers of $\A$ and $\diag$]
	\label{prop:fractional_powers_positive}
	Let Assumption \ref{ass:standing} be satisfied. 
	Suppose that $\A$ and $\diag$ have a bounded $H^{\infty}$-calculus and that \eqref{eq:equivalence_norm_A_0_A} holds.
	Then for all $\theta\in (0,1)$
	$$
	\Do(\A^{\theta})=\Do(\diag^{\theta}) \ \text{ and }\ 
	\|\diag^{\theta}x\|_{X}\eqsim \|\A^{\theta} x\|_X
	\ \text{ for all }x\in \Do(\diag^{\theta}).
	$$ 
	In particular $
	\Dd(\A^{\theta})=\Dd(\diag^{\theta})$ for all $\theta\in (0,1)$.
\end{proposition}

\begin{proof}
	Since $\A$ has a bounded $H^{\infty}$-calculus and $\Do(\diag)=\Do(\A)$, we have $\Do(\A^{\theta})=\Do(\diag^{\theta})$ and $\Dd(\A^{\theta})=\Dd(\diag^{\theta})$ for all $\theta\in [0,1]$, compare \cite[Proposition 2.2]{KKW}. Here, in the second identification we have also used that \eqref{eq:equivalence_norm_A_0_A} holds. The remaining estimate follows from these identifications and the definition of these spaces in \eqref{eq:def_homogeneous_scale}.
\end{proof}

\subsection{Proofs of Theorem \ref{t:calculus_pert_block_operators} and Theorem~\ref{t:calculus_pert_block_operators_notype}
} 
\label{ss:proof_pert_H_calculus}
%The proof of Theorems  \ref{t:calculus_bounded_no_smallness_condition} and \ref{t:calculus_pert_block_operators} are based on a more general result which will be presented below. 
%
We begin by proving the following key lemma.
\begin{lemma}
	\label{l:estimate_f_functional_calculus}
	Assume that operators $S$ and $T$ have a bounded $H^{\infty}$-calculus on Banach spaces $E$ and $F$, respectively. Fix $\sigma> \angH(S)\vee\angH(T)$ and 
	let $\mathcal{J} :\C\setminus \Sigma_{\sigma}\to \mathscr{J}$ be a strongly continuous map where $\mathscr{J}\subseteq \calL(E,F)$ is assumed to be an $\Rsec$-bounded set. Then, for each $f\in H^{\infty}_0(\Sigma_{\sigma})$,
	$$
	\Big\|\int_{\Gamma} f(\lambda) T^{\delta} (\lambda-T)^{-1} \mathcal{J}(\lambda) S^{1-\delta}(\lambda-S)^{-1}\, \dd\lambda \Big\|_{\calL(E,F)}\lesssim \Rsec(\mathscr{J}) \|f\|_{H^{\infty}(\Sigma_{\sigma})},
	$$
	where $\delta\in (0,1)$ and $\Gamma=\partial\Sigma_{\theta}$ with $\theta\in (\angH(S)\vee\angH(T),\sigma)$.
\end{lemma}

	The proof of Lemma \ref{l:estimate_f_functional_calculus} is based on standard randomization techniques, see e.g.\ \cite[Theorem 4.5.6]{pruss2016moving}. For the reader's convenience, we provide some details.
	
\begin{proof}
	Set $\Gamma^{\pm}\stackrel{{\rm def}}{=} \{r e^{\pm i\theta}: r>0\}$, then $\Gamma=\Gamma^+\cup \Gamma^-$, and we prove the claimed estimate for $\Gamma$ replaced by $\Gamma^+$, the one for $\Gamma^-$ follows similarly. For notational convenience, we set $\mathcal{J}_f(\lambda)\stackrel{{\rm def}}{=}f(\lambda)\mathcal{J}(\lambda)$. Note that
	\begin{equation}
	\label{eq:representation_formula_functional_calculus_f}
	\begin{aligned}
	&\int_{\Gamma^+}  f(\lambda) T^{\delta}(\lambda-T)^{-1} \mathcal{J}(\lambda )S^{1-\delta}(\lambda-S)^{-1}\, \dd\lambda\\
	&=\lim_{N\uparrow \infty} \sum_{j=-N}^{N-1}\int_{2^{j}}^{2^{j+1}} 
	T^{\delta}(r e^{i\theta}-T)^{-1} \mathcal{J}_f(re^{i\theta} )S^{1-\delta}(r e^{i\theta}-S)^{-1}\, \dd r\\
	&=\lim_{N\uparrow \infty} \sum_{j=-N}^{N-1}\int_{2^{j}}^{2^{j+1}} 
	\Big(\frac{T}{r}\Big)^{\delta}\Big(e^{i\theta}-\frac{T}{r}\Big)^{-1} \mathcal{J}_f(e^{i\theta} r)\Big(\frac{S}{r}\Big)^{1-\delta}
	\Big( e^{i\theta}-\frac{S}{r}\Big)^{-1}\, \frac{\dd r}{r}\\
	&=\lim_{N\uparrow \infty} \int_{1}^{2} \sum_{j=-N}^{N-1}
	\Big(\frac{T}{t 2^j}\Big)^{\delta}\Big(e^{i\theta}-\frac{T}{t2^j }\Big)^{-1} \mathcal{J}_f(e^{i\theta} t2^j)
	\Big(\frac{S}{t2^j }\Big)^{1-\delta}
	\Big( e^{i\theta}-\frac{S}{t2^j}\Big)^{-1}\, \frac{\dd t}{t}\\
	&=\lim_{N\uparrow \infty} \int_{1}^{2} \sum_{j=-N}^{N-1}h_{\delta}(2^j t T) \mathcal{J}_f(e^{i\theta} 2^j t) h_{1-\delta}(2^j t S)\,\frac{\dd t}{t},
	\end{aligned}
	\end{equation}
	where $h_{\rho}(z)=z^{\rho}(e^{i\theta}-z)^{-1}$. %defines a function in $H^{\infty}_0(\Sigma_{\sigma'})$ for all $\rho\in (0,1)$ and $\sigma'\in (\sigma,\angH(S)\vee\angH(T))$. 
	We estimate the operators appearing in the above integral by standard randomization techniques. To this end, let us 
	note that, by the contraction principle (see e.g.\ %\cite[Lemma 4.1.7]{pruss2016moving} or
	 \cite[Theorem 6.1.13 (ii)]{Analysis2}), 
	\begin{equation}
	\label{eq:R_bounded_range_J_f}
	\Rsec(\mathcal{J}_{f}(r e^{i\theta}):r>0)\leq \tfrac{\pi}{2}  \|f\|_{H^{\infty}(\Sigma_{\sigma})}\Rsec(\mathscr{J}).
	\end{equation} 
	Since each $S$ and $T$ has  a bounded $H^{\infty}$-calculus and  $h_{\rho}(z)=z^{\rho}(e^{i\theta}-z)^{-1}$ defines a function in $H^{\infty}_0(\Sigma_{\sigma'})$ for all $\rho\in (0,1)$ and $\sigma'\in (\angH(S)\vee\angH(T),\theta)$, by  \cite[Proposition 10.2.20 and Lemma 10.3.8(1)]{Analysis2} there exists a constant $C$ such that for all finite collections of scalars $(\alpha_k)_{k\in \Z},(\beta_k)_{k\in \Z}$ such that $|\alpha_k|,|\beta_k|\leq 1$,
	\begin{equation}
	\label{eq:unconditionality_H_calculus} 
		\sup_{t>0}\Big(\Big\|\sum_{j=-N}^{N-1} \beta_k h_{\delta}(2^jt T^* )\Big\|_{\calL(F^*)}
		+
		\Big\|\sum_{j=-N}^{N-1} \alpha_k h_{1-\delta}(2^jt S )\Big\|_{\calL(E)}\Big)
	\leq C.
	\end{equation}
	Here $C$ is a constant independent of  $t$ and $N$ (see also below Definition 10.2.12 and Proposition H.2.3 in \cite{Analysis2}).
	Let $(\varepsilon_j)_{j\in \Z}$ be a Rademacher sequence, see Subsection \ref{ss:rsec_max_reg}. Then, for any $x\in E$ and $y^*\in F^*$ we have
	\begin{align*}
	&\Big|\sum_{j=-N}^{N-1}
	\langle y^*, h_{\delta}(2^j t T) \mathcal{J}_f(e^{i\theta} 2^j t) h_{1-\delta}(2^j t S) x \rangle\Big|\\
	=&\Big|\sum_{j=-N}^{N-1}\langle h_{\delta}(2^j t T^*)  y^*, \mathcal{J}_f(e^{i\theta} 2^j t) h_{1-\delta}(2^j t S)x \rangle\Big|\\
	\stackrel{(i)}{=}&
	\Big| \E\Big\langle \sum_{j=-N}^{N-1} \varepsilon_j h_{\delta}(2^j t T^*)  y^*, \sum_{n=-N}^{N-1} \varepsilon_n\mathcal{J}_f(e^{i\theta} 2^n t) h_{1-\delta}(2^n t S)x \Big\rangle\Big|\\
	\leq& \Big\|\sum_{j=-N}^{N-1}\varepsilon_j h_{\delta}(2^j t T^*)  y^*\Big\|_{L^2(\O;F^*)}
	\Big\|\sum_{j=-N}^{N-1}\varepsilon_j \mathcal{J}_f(e^{i\theta} 2^j t) h_{1-\delta}(2^j t S)x \Big\|_{L^2(\O;E)}\\ 
	\stackrel{(ii)}{\leq}&\tfrac{\pi}{2} C^2 \|f\|_{H^{\infty}(\Sigma_{\sigma})} \Rsec(\mathscr{J})  \|y^*\|_{F^*} \|x\|_{E},
	\end{align*}
	where in $(i)$ we used $\E[\varepsilon_j \varepsilon_n]=\delta_{j,n}$, here $\delta_{j,n}$ denotes Kronecker's delta, and in $(ii)$ the Equations~\eqref{eq:R_bounded_range_J_f} and \eqref{eq:unconditionality_H_calculus}. Hence,
	\begin{align*}	\Big\|
\sum_{j=-N}^{N-1}	h_{\delta}(2^j t T) \mathcal{J}_f(e^{i\theta} 2^j t) h_{1-\delta}(2^j t S)
	\Big\|_{\calL(E,F)}\leq 
	\tfrac{\pi}{2} C^2 \|f\|_{H^{\infty}(\Sigma_{\sigma})} \Rsec(\mathscr{J}).
\end{align*}
Combining the latter with \eqref{eq:representation_formula_functional_calculus_f}, one gets the desired estimate.
%\todo{Amru: I changed $\|f\|_{H^{\infty}(\Sigma_{\nu})}$ to $\|f\|_{H^{\infty}(\Sigma_{\sigma})}$ in the proof - correct?}
\end{proof}

%Next we prove Theorem \ref{t:calculus_pert_block_operators}. First we assume that Assumption \ref{ass:fractions_pm}$(-)$ holds.

\begin{proof}[Proof of Theorem \ref{t:calculus_pert_block_operators} in case that Assumption \ref{ass:fractions_pm}~\nameref{ass:fractions_-} holds]
	Let us begin by collecting some useful facts. By Proposition \ref{prop:pert_block_operators}, $\A$ is $\Rsec$-sectorial of angle $<\psi$. Fix $\phi\in (\angH(\A),\psi)$. Then, by Theorem \ref{t:rsec_necessary_sufficient_condition_II},  for all $|\arg\lambda|>\angR(\A)$ the resolvent $(\lambda-\A)^{-1}$  is given by the factorization in    Proposition~\ref{l:relation_resolvent_AD_calA} \eqref{eq:resolvent_formula_block_matrix_I} with $M_1(\lambda)$ and $M_2(\lambda)$  as in \eqref{eq:def_M_1_M_2}. 
	
	Next we look at the functional calculus. 
	Fix $f\in H^{\infty}_0(\Sigma_{\phi})$. Consider w.l.o.g. the case where $X_1$ has non-trivial type and $c_1^{\Rsec}
	<1/K_{X_1}$, then one has to estimate  
	\begin{align}
\label{eq:estimate_1_f_small_H_calculus}
I_{11}&\stackrel{{\rm def}}{=}\Big\| 
\int_{\Gamma}  f(\lambda) (\lambda-A)^{-1} M_1(\lambda)^{-1} \dd\lambda\Big\|_{\calL(X_1)},\\
\label{eq:estimate_2_f_small_H_calculus}
I_{21}&\stackrel{{\rm def}}{=}\Big\| 
\int_{\Gamma}  f(\lambda)  (\lambda-D)^{-1} C(\lambda-A)^{-1}M_1(\lambda)^{-1} \dd\lambda\Big\|_{\calL(X_1,X_2)}, \\
\label{eq:estimate_3_f_small_H_calculus}
I_{12}&\stackrel{{\rm def}}{=}\Big\| 
\int_{\Gamma}  f(\lambda)  (\lambda-A)^{-1} M_1(\lambda)^{-1}B(\lambda-D)^{-1} \dd\lambda\Big\|_{\calL(X_2,X_1)}, \\
\label{eq:estimate_4_f_small_H_calculus}
I_{22}&\stackrel{{\rm def}}{=}\Big\| 
\int_{\Gamma}  f(\lambda) (\lambda-D)^{-1}\left[ \id +C(\lambda-A)^{-1} M_1(\lambda)^{-1}B(\lambda-D)^{-1}\right] \dd\lambda\Big\|_{\calL(X_2)}
\end{align}
	by $\lesssim\|f\|_{H^{\infty}(\Sigma_{\psi})}$,
	where the implicit constants depend only on %$C_1,C_2,C_B,\varepsilon,
	 the $\Rsec$-bound $c_1^{\Rsec}$ in \eqref{eq:assumption_epsilon_small_calculus} and the constants of the $H^\infty$-calculus of $A,D$.  
%	For the reader's convenience, 
We split the proof into three steps.

	\emph{Step 1:} Assumption \ref{ass:fractions_pm}\nameref{ass:fractions_-} implies that $C$ and $B$ extend uniquely to  bounded linear operators
 	\begin{align*}
C_{\delta}\in \calL(\Dd(A^{1-\delta}), \Dd(D^{-\delta})) \quad \hbox{and} \quad B_{\delta}\in \calL(\Dd(D^{1-\delta}),\Dd(A^{-\delta})).
 \end{align*}
	Reasoning as in the proof of Proposition \ref{prop:fractional_powers_positive}, by complex interpolation one has for all $\eta\in [0,\delta]$ that $C$ and $B$ induce uniquely the operators
\begin{equation}
\label{eq:mapping_B_1_dot_space}
C_{\eta}\in \calL(\Dd(A^{1-\eta}), \Dd(D^{-\eta})) \quad \hbox{and} \quad B_{\eta}\in \calL(\Dd(D^{1-\eta}),\Dd(A^{-\eta})),
\end{equation}  
	satisfying $C_{\eta} x=C x$ and $B_{\eta} y=B y$ for all $(x,y)\in \Do(A)\times\Do(D)$. 
	Hence, using \cite[Theorem 15.14 b)]{KuWe}, we have for all $\eta\in [0,\delta]$, 
	  $x\in \Do(A)\cap \Ran(A)$ and $y\in \Do(D)\cap \Ran(D)$,
			\begin{align*}
	\|D^{-\eta}C A^{\eta-1} x\|_{X_2}\lesssim \|x\|_{X_1} \quad \hbox{and} \quad \|A^{-\eta}B D^{\eta-1} y\|_{X_1}\lesssim \|y\|_{X_2}.
	\end{align*}
		Since 
	$\Do(A)\cap \Ran(A)\embed X_1$ and $\Do(D)\cap \Ran(D)\embed X_2$ are dense (see e.g.\ \cite[Theorem 3.1.2(iv)]{pruss2016moving}),
	this ensures that $D^{-\eta}CA^{\eta-1}$ and $A^{-\eta}BD^{\eta-1}$ admit a unique extension to bounded operators, namely
	\begin{equation}
	\label{eq:def_G_H_perturbation_H_calculus}
	G_{\eta}\stackrel{{\rm def}}{=}\overline{A^{-\eta}BD^{\eta-1}}\in \calL(X_2,X_1) \quad\text{and}\quad
	H_{\eta}\stackrel{{\rm def}}{=}\overline{D^{-\eta}CA^{\eta-1}}\in \calL(X_1,X_2).
	\end{equation}

\emph{Step 2:}
	The estimate~\eqref{eq:assumption_epsilon_small_calculus}  
	ensures that 
	$$
	M_1(\lambda)^{-1}=\sum_{n\geq 0} [B(\lambda-D)^{-1}C(\lambda-A)^{-1}]^n
	$$ 
	convergences absolutely  in $\calL(X_1)$, because $K_{X_1}\geq 1$ for any $K$-convex space, and that $\{M_1(\lambda)^{-1}\colon \lambda \in \C\setminus\Sigma_{\psi}\}$ is $\Rsec$-bounded. 
	Next we rewrite the series conveniently. Note that, for $\eta \in [0,\delta]$ and on $\Ran(A^{\eta})$,
	\begin{align*}
	\mathcal{T}(\lambda)&\stackrel{{\rm def}}{=}B(\lambda-D)^{-1}C(\lambda-A)^{-1}\\
	&= A^{\eta} (A^{-\eta} B D^{\eta-1}) D^{1-\eta} (\lambda-D)^{-1} D^{\eta}(D^{-\eta}C A^{\eta-1})A^{1-\eta} (\lambda-A)^{-1} \\
	&=A^{\eta} G_{\eta} D(\lambda-D)^{-1} H_{\eta}A (\lambda-A)^{-1} A^{-\eta} \\
	&=A^{\eta}\mathcal{S}_\eta(\lambda)A^{-\eta},
	\end{align*}
	where $\mathcal{S}_\eta(\lambda)\stackrel{{\rm def}}{=}G_{\eta} D(\lambda-D)^{-1} H_{\eta}A (\lambda-A)^{-1}\in \calL(X_1)$.
	Hence $\mathcal{T}(\lambda)$ extends to an operator $\mathcal{T}_\eta(\lambda)\in \calL(\Dd(A^{-\eta}))$ and since $\{\mathcal{S}_\eta(\lambda)\colon \lambda \in \C\setminus\Sigma_{\psi}\}$ is $\Rsec$-bounded in $\calL(X_1)$, also
	\begin{align*}
	\{\mathcal{T}_\eta(\lambda)\colon \lambda \in \C\setminus\Sigma_{\psi}\} \subseteq \calL(\Dd(A^{-\eta}))  \hbox{ satisfies }  c_{\mathcal{T}_\eta}^{\Rsec}\stackrel{{\rm def}}{=}\Rsec \big(\mathcal{T}_\eta(\lambda)\colon \lambda \in \C\setminus\Sigma_{\psi}\big)<\infty.
	\end{align*}
	  By \cite[Proposition 8.4.4]{Analysis2} $\Rsec$-boundedness interpolates assuming $K$-convexity. Here, by assumption, $X_1$ is $K$-convex and since $\Dd(A^{-\eta})$ is isomorphic to $X_1$ it is also $K$-convex.
	 Hence by complex interpolation for $\eta=\theta\delta$ and $\theta \in (0,1)$
		\begin{align*}
 c_{\mathcal{T}_{\theta\delta}}^{\Rsec} \leq
 K_{\Dd(A^{-\eta})}^{1-\theta}K_{X_1}^{1-\theta}
  (c_{\mathcal{T}_0}^{\Rsec})^{\theta} (c_{\mathcal{T}_\delta}^{\Rsec})^{1-\theta},
	\end{align*}
where $K_{\Dd(A^{-\eta})}$ and $K_{X_1}$ are the $K$-convexity constants of the spaces $\Dd(A^{-\eta})$ and $X_1$, respectively.
	Since  $c_{\mathcal{T}_0}^{\Rsec}=c_1^{\Rsec}<1/K_{X_1}$, there exists an $\eta\in (0,\delta]$ with $c_{\mathcal{T}_{\eta}}^{\Rsec}<1$. In particular for this $\eta$ the series
	\begin{align*}
	\mathcal{U}(\lambda)\stackrel{{\rm def}}{=}\sum_{n\geq 0} \mathcal{S}_\eta(\lambda)^n = A^{-\eta}\Big(\sum_{n\geq 0} \mathcal{T}_\eta(\lambda)^n\Big)A^{\eta}
	\end{align*}
	converges absolutely in $\calL(X_1)$ and  $\{\mathcal{U}(\lambda)\colon \lambda \in \C\setminus\Sigma_{\psi}\}$ is $\Rsec$-bounded.

\emph{Step 3:}	Finally we can estimate \eqref{eq:estimate_1_f_small_H_calculus}--\eqref{eq:estimate_4_f_small_H_calculus}.
		Using the above argument, one can check that
		\begin{align*}
			(\lambda-A)^{-1}M_1(\lambda)^{-1} = 
		(\lambda-A)^{-1} +A^{\eta}(\lambda-A)^{-1}\mathcal{J}_{11}(\lambda)A^{1-\eta}(\lambda-A)^{-1},
		\end{align*}
		where $\{\mathcal{J}_{11}(\lambda)\colon \lambda \in \C\setminus\Sigma_{\psi}\}$ is $\Rsec$-bounded and
		\begin{align*}
	 \mathcal{J}_{11}(\lambda)\stackrel{{\rm def}}{=} \Big(\sum_{n\geq 2}\mathcal{S}_\eta(\lambda)^n\Big)\Big(G_\eta D(\lambda-D)^{-1}H_\eta\Big).
		\end{align*}
			Thus, \eqref{eq:estimate_1_f_small_H_calculus} follows %from \eqref{eq:identity_A_M_1_smallness_H_infty_calculus}, \eqref{eq:R_boundedness_S} and 
			by Lemma \ref{l:estimate_f_functional_calculus} applied with $S=T=A$ and the assumption on $A$.
			
To show \eqref{eq:estimate_2_f_small_H_calculus},	we write similar to the above 
	\begin{align*}
	(\lambda-D)^{-1} C(\lambda-A)^{-1}M_1(\lambda)^{-1}%&=(\lambda-D)^{-1}D^{\eta} (D^{-\eta}C A^{\eta-1}) A^{1-\eta}(\lambda-A)^{-1} \\
	%+	(\lambda-D)^{-1}D^{\eta} (D^{-\eta}& C A^{\eta-1}) A(\lambda-A)^{-1}\mathcal{J}_{11}(\lambda)A^{1-\eta}(\lambda-A)^{-1} \\
	&= (\lambda-D)^{-1}D^{\eta} \mathcal{J}_{21}(\lambda)A^{1-\eta}(\lambda-A)^{-1},
	\end{align*}
	where $\{\mathcal{J}_{21}(\lambda)\colon \lambda \in \C\setminus\Sigma_{\psi}\}$ is $\Rsec$-bounded for
	$\mathcal{J}_{21}(\lambda)\stackrel{{\rm def}}{=} H_{\eta} + H_{\eta}A(\lambda-A)^{-1}\mathcal{J}_{11}(\lambda)$,
	so that Lemma \ref{l:estimate_f_functional_calculus} applies with $T=D$ and $S=A$. 
	
For \eqref{eq:estimate_3_f_small_H_calculus}
we rewrite 
\begin{align*}
(\lambda-A)^{-1} M_1(\lambda)^{-1}B(\lambda-D)^{-1} 
%&=
%(\lambda-A)^{-1}A^{\eta}(A^{-\eta}BD^{\eta-1})D^{1-\eta}(\lambda-D)^{-1} \\ +A^{\eta}(\lambda-A)^{-1}&\mathcal{J}_{11}(\lambda)A(\lambda-A)^{-1}(A^{-\eta}BD^{\eta-1})D^{1-\eta}(\lambda-D)^{-1} \\
&=(\lambda-A)^{-1}A^{\eta}\mathcal{J}_{12}(\lambda)D^{1-\eta}(\lambda-D)^{-1} 
\end{align*}
where $\{\mathcal{J}_{12}(\lambda)\colon \lambda \in \C\setminus\Sigma_{\psi}\}$ is $\Rsec$-bounded with
$\mathcal{J}_{12}(\lambda)\stackrel{{\rm def}}{=} G_{\eta} + \mathcal{J}_{11}(\lambda)A(\lambda-A)^{-1}G_{\eta}$.

In \eqref{eq:estimate_4_f_small_H_calculus}
the first addend can be estimated by the assumption on $D$ and for the second we write using the previous computation 
\begin{multline*}
(\lambda-D)^{-1}C(\lambda-A)^{-1} M_1(\lambda)^{-1}B(\lambda-D)^{-1}
%=  \\
%(\lambda-D)^{-1}D^{\eta}(D^{-\eta}CA^{\eta-1})A(\lambda-A)^{-1}\mathcal{J}_{12}(\lambda)D^{1-\eta}(\lambda-D)^{-1} \\
=(\lambda-D)^{-1}D^{\eta}\mathcal{J}_{22}(\lambda)D^{1-\eta}(\lambda-D)^{-1},
\end{multline*}	
	where $\mathcal{J}_{22}(\lambda)\stackrel{{\rm def}}{=}H_\eta A(\lambda-A)^{-1}\mathcal{J}_{12}(\lambda)$
	and $\{\mathcal{J}_{22}(\lambda)\colon \lambda \in \C\setminus\Sigma_{\psi}\}$ is $\Rsec$-bounded.
\end{proof}

\begin{proof}[Proof of Theorem \ref{t:calculus_pert_block_operators} for the case of Assumption \ref{ass:fractions_pm}\nameref{ass:fractions_+}] Following \cite[Corollary 6.5]{KKW}, we prove the claim by employing a shift argument on a scale of spaces and Theorem \ref{t:calculus_pert_block_operators} for the already proven case with Assumption \ref{ass:fractions_pm}\nameref{ass:fractions_-}. 	
%	Let us begin by collecting some useful facts. 

By Proposition \ref{prop:pert_block_operators}, $\A$ is $\Rsec$-sectorial with angle $\angR(\A)<\psi$.
Below we use the notation introduced in Subsection \ref{ss:sectorial_operator}. In particular $\dot{A}_{\delta}$ and $\dot{D}_{\delta}$ are sectorial operators with bounded $H^{\infty}$-calculus on $\Dd(A^{\delta})$ and $\Dd(D^{\delta})$, respectively, which follows by similarity from the Definition in~\eqref{eq:definition_Tdot}.
Moreover, Assumption \ref{ass:fractions_pm}\nameref{ass:fractions_+} ensures that $B$ and $C$ %can be uniquely  
%restricted to 
uniquely induce  operators
	\begin{align*}
	B_{\delta}&\colon \Dd(D^{1+\delta})\to \Dd(A^{\delta}), \qquad \hbox{with} \qquad  B_{\delta}|_{\Dd(D^{1+\delta})\cap \Do(D) }= B, \hbox{ and }\\
	C_{\delta}&\colon \Dd(A^{1+\delta})\to \Dd(D^{\delta}), \qquad \hbox{with} \qquad  C_{\delta}|_{\Dd(A^{1+\delta})\cap \Do(A) }= C.
	\end{align*}
	By density of $\Do(A)\subseteq \Dd(A)$ and $\Do(D)\subseteq \Dd(D)$, Assumption~\ref{ass:standing} implies that also
	\begin{align*}
	\norm{B_0y} \leq c_D \norm{ y}_{\Dd(D)} \hbox{ for all } y\in \Dd(D), \hbox{ and }  \norm{C_0x} \leq c_A \norm{ x}_{\Dd(A)} \hbox{ for all } x\in \Dd(A)
	\end{align*}
	hold. By interpolation for $\eta \in (0,\delta)$
	\begin{align*}
	\norm{B_\eta} \leq c_D^{\eta/\delta} \norm{B_{\delta}}^{1-\eta/\delta} \quad \hbox{and} \quad \norm{C_\eta} \leq c_A^{\eta/\delta} \norm{C_{\delta}}^{1-\eta/\delta}.
	\end{align*}
	So
%	By repeating the interpolation argument used in Step 1 in the proof of the case of Assumption~\ref{ass:fractions_pm}\nameref{ass:fractions_-} above, 
we may assume, at the expense of choosing $\delta\in (0,1)$ small enough, 
	\begin{equation}
	\label{eq:smallness_shifted_scale}
	\|C_{\delta} x\|_{ \Dd(D^{\delta})}\leq c_A' \| \dot{A}_{\delta} x\|_{\Dd(A^{\delta})},
	\ \ \text{ and } \ \ 
	\|B_{\delta} y\|_{ \Dd(A^{\delta})}\leq c_D' \| \dot{D}_{\delta} y\|_{\Dd(D^{\delta})}, 
	\end{equation}
	for all %$(x,y)\in \Dd(\diag^{1+\delta})\cap \Dd(\A_{0}^{\delta})$ 
	$x\in \Dd(A^{1+\delta})$ and $y\in \Dd(D^{1+\delta})$,
	and some $c_D'\geq c_D$ and $c_A'\geq c_A$, respectively.

Consider then the block operator matrix
\begin{align*}
\wh{\A}_{\delta}
\stackrel{{\rm def}}{=}
\begin{bmatrix}
\dot{A}_{\delta} & B_{\delta}\\
C_{\delta} & \dot{D}_{\delta}
\end{bmatrix}\colon
\Dd(\diag^{1+\delta})\cap \Dd(\diag^{\delta})
\subseteq \Dd(\diag^{\delta})\to \Dd(\diag^{\delta}), \ \ 
 \Dd(\diag^{\delta})=\Dd(A^{\delta})\times \Dd(D^{\delta}).
\end{align*}
 By \eqref{eq:smallness_shifted_scale} $\wh{\A}_{\delta}$ satisfies Assumption~\ref{ass:standing} with $L=0$ for $X_1=\Dd(A^{\delta})$ and $X_2=\Dd(D^{\delta})$.
Note that \textit{a priori} $\wh{\A}_{\delta}$ is  \emph{not} equal to the extrapolated operator $\dot{\A}_{\delta}$, since the state space for $\wh{\A}_{\delta}$ is $\Dd(\diag^{\delta})$ which may differ from the state space of $\dot{\A}_{\delta}$, that is $\Dd(\A^{\delta})$. 
%	\begin{equation}
%	\label{eq:smallness_shifted_scale}
%	c_D'\geq c_D, \qquad c_A'<\frac{1}{c_D' \non_{\psi}(A)\non_{\psi}(D)}.
%	\end{equation} 
The relative bounds in  Assumption \ref{ass:standing} for $\A$ imply for $B_{\delta}$ and $C_{\delta}$ that
\begin{equation}
\label{eq:ass_standing_shifted_scale}
\begin{aligned}
\|(\dot{D}_{\delta})^{-\delta} C_{\delta}x \|_{\Dd(D^{\delta})}
&\leq c_A
\|(\dot{A}_{\delta})^{1-\delta} x\|_{\Dd(A^{\delta})}, \quad x\in \Do(A^{1+\delta}), \\
\|(\dot{A}_{\delta})^{-\delta} B_{\delta}y \|_{\Dd(A^{\delta})}
&\leq c_D
\|(\dot{D}_{\delta})^{1-\delta} y\|_{\Dd(D^{\delta})}, \quad y\in \Do(D^{1+\delta}),
\end{aligned}
\end{equation}
%where $(x,y)\in \Dd(\diag^{1+\delta})\cap \Dd(\A_{0}^{\delta})$.
which is the estimate in  Assumption~\ref{ass:fractions_pm}\nameref{ass:fractions_-} for $\wh{\A}_{\delta}$, and the range conditions hold by construction. Consider w.l.o.g.\ the case where $c_1^{\Rsec}<1/K_{X_1}$. Then 
\begin{align}
c_{1,\delta}^{\Rsec}\stackrel{{\rm def}}{=}\Rsec(B_{\delta} (\lambda-\dot{D}_\delta)^{-1}C_\delta(\lambda-\dot{A}_\delta)^{-1}\colon \lambda\in \C\setminus \Sigma_{\psi})<\infty
%c_A <\frac{1}{c_D \non_{\psi}(A)\non_{\psi}(D)}. 
\end{align}
and by repeating the interpolation argument used in Step 2 in the proof of the case of Assumption~\ref{ass:fractions_pm}\nameref{ass:fractions_-} above, at the expense of choosing $\delta\in (0,1)$ small enough, 
we may assume that $c_{1,\delta}^{\Rsec}<1/K_{X_1}$. Hence Theorem \ref{t:calculus_pert_block_operators} for the case of Assumption~\ref{ass:fractions_pm}\nameref{ass:fractions_-}, which has been proven already, ensures that $\wh{\A}_{\delta}$ has a bounded $H^{\infty}$-calculus of angle $<\psi$.

%
% By \eqref{eq:smallness_shifted_scale}--\eqref{eq:ass_standing_shifted_scale} and $\non_{\psi}(\dot{T}_{\delta})=\non_{\psi}(T)$ for $T\in\{A,D\}$, 
	Next, let us point out the %an important 
	relation between $\wh{\A}_{\delta}$ and $\A$. 
	To this end, let us recall that by \eqref{eq:assumption_epsilon_small_calculus}, the fact that $K_{X_1}\geq 1$ and the last claim in Lemma \ref{l:one_side_A_0_A_estimate},
	\begin{equation}
	\label{eq:domains_equality_proof_plus}
	\Do(\A)=\Do(\diag)\quad \text{ and } \quad \Dd(\A)=\Dd(\diag).
	\end{equation}
	By construction we have $\wh{\A}_{\delta}|_{\Do(\diag^{1+\delta})}=\A|_{\Do(\diag^{1+\delta})}$, and  by the   Proposition~\ref{l:relation_resolvent_AD_calA} 
\begin{multline*}
%	\wh{\M}_\delta(\lambda)^{-1}
(\lambda-\wh{\A}_{\delta})^{-1}	= \\
	\begin{bmatrix}
(\lambda-\dot{A}_\delta)^{-1} & 0  \\
0 & (\lambda-\dot{D}_\delta)^{-1} 
\end{bmatrix} 	\begin{bmatrix}
	\wh{M}_{1,\delta}(\lambda)^{-1} & 0 \\
	C_\delta(\lambda-\dot{A}_\delta)^{-1}	\wh{M}_{1,\delta}(\lambda)^{-1} & \id
	\end{bmatrix}
	\begin{bmatrix}
	\id & B_\delta(\lambda-\dot{D}_\delta)^{-1}  \\
	0 & \id 
	\end{bmatrix},
	\end{multline*}
	% and the proof of the \nameref{ass:fractions_-}-case also shows 
	where  $\wh{M}_{1,\delta}(\lambda)= \id - B_{\delta} (\lambda-\dot{D}_\delta)^{-1}C_\delta(\lambda-\dot{A}_\delta)^{-1}$ 
	and 
since $c_{1,\delta}^{\Rsec}<1/K_{X_1}$ one has $$\wh{M}_{1,\delta}(\lambda)^{-1}=\sum_{n\geq 0} [B_{\delta} (\lambda-\dot{D}_\delta)^{-1}C_\delta(\lambda-\dot{A}_\delta)^{-1}]^n.$$ Using the mapping properties of $\dot{A}_\delta, \dot{D}_\delta, B_\delta, C_\delta$, it follows that $(\lambda-\wh{\A}_{\delta})^{-1}$ restricts to an operator on $\Do(\diag^{\delta})=\Dd(\diag^{\delta})\cap X$ and $(\lambda-\wh{\A}_{\delta})^{-1}\Do(\diag^{\delta})\subseteq \Dd(\diag^{1+\delta}) \cap \Do(\diag^{\delta})= \Do(\diag^{1+\delta})$ (cf.\ \eqref{eq:fractional_powers_positive_intersection_X}).
%%%%
 %	\todo{Antonio: This observation should be made more precise. Here I means that $\mathcal{U}$ maps $\Dd(A^{\delta})$ into itself. Probably we could write a remark on this fact.
 %	Amru: Is the factorization enough to illustrate this, or should we add more details?} 
%%%%%	
Hence, for all $x\in \Do(\diag^{\delta})$,
	\begin{align*}
	(\lambda-\A)^{-1}x 
	&= (\lambda-\A)^{-1} (\lambda-\wh{\A}_{\delta})(\lambda-\wh{\A}_{\delta})^{-1} x\\
	&=(\lambda-\A)^{-1} (\lambda-\A)(\lambda-\wh{\A}_{\delta})^{-1} x
	=(\lambda-\wh{\A}_{\delta})^{-1} x.
	\end{align*}	
	Since $\Do(\A)\stackrel{\eqref{eq:domains_equality_proof_plus}}{=} \Do(\diag)\embed \Do(\diag^{\delta})$, the previous display yields
	\begin{equation}
	\label{eq:equality_resolvent_A_wh_A_1}
	(\lambda-\wh{\A}_{\delta})^{-1}|_{\Do(\A)} =(\lambda-\A)^{-1}|_{\Do(\A)}
	\stackrel{\eqref{eq:consistency_T_T_gamma}}{=} 
	(\lambda-\dot{\A}_{1})^{-1}|_{\Do(\A)}\text{ for all }\lambda\in \complement\overline{\Sigma_{\psi}}.
	\end{equation}
	%
%To proceed further, note that $(\wh{\A}_{\delta})^{\dot{}}_{1}$ has bounded $H^{\infty}$-calculus since is similar to $\wh{\A}_{\delta}$. 
	Next, we prove that 
\begin{equation}
\label{eq:domain_wh_A}
	\Dd(\wh{\A}_{\delta})= \Dd(\diag^{1+\delta}), \ \ \text{ and } \ \ 
	\|\wh{\A}_{\delta} x\|_{\Dd(\diag^{\delta})}\eqsim \| x\|_{\Dd(\diag^{1+\delta})} 
	\end{equation}
	 for all $x\in \Dd(\diag^{1+\delta})\cap \Dd(\diag^{\delta})$.
	By Assumption \ref{ass:fractions_pm}\nameref{ass:fractions_+} we have
	$$
	\|\diag^{\delta} \A x\|_{X}\lesssim \|\diag^{1+\delta} x\|_{X} \quad \text{for all}\quad x\in \Do(\diag^{1+\delta}).
	$$
	Thus $\Dd(\diag^{1+\delta})\embed \Dd(\wh{\A}_{\delta})$. The reverse inclusion follows from the last claim in Lemma \ref{l:one_side_A_0_A_estimate} applied with $\A$ replaced by $\wh{\A}_{\delta}$ and $X=\Dd(\diag^{\delta})$ using that $c^{\Rsec}_{1,\delta}<1$ by assumption. 
	
	Note that $(\wh{\A}_{\delta})^{\dot{}}_{1-\delta}$ has a bounded $H^{\infty}$-calculus on 
	$\Dd((\wh{\A}_{\delta})^{1-\delta})$ since it is similar to $\wh{\A}_{\delta}$. 
	By \cite[Proposition 2.2]{KKW} and \eqref{eq:domain_wh_A},
	\begin{equation}
	\label{eq:equivalence_power_whA_1_delta_diag}
	\Dd((\wh{\A}_{\delta})^{1-\delta})
	=
	[\Dd(\diag^{\delta}),\Dd(\diag^{1+\delta})]_{1-\delta}=\Dd(\diag)
	\stackrel{\eqref{eq:domains_equality_proof_plus}}{=}\Dd(\A).
	\end{equation} 
	Since 
	 $\Do(\A)\supseteq \Do(\diag^{1+\delta})\stackrel{d}{\embed} \Dd((\wh{\A}_{\delta})^{1-\delta})$, we also have 
	$$
	(\lambda-(\wh{\A}_{\delta})^{\dot{}}_{1-\delta})^{-1}|_{\Do(\A)}
	\stackrel{\eqref{eq:consistency_T_T_gamma}}{=}(\lambda-\wh{\A}_{\delta})^{-1}|_{\Do(\A)}.
	$$
	By $\Do(\A)\stackrel{d}{\embed}\Dd(\A)$, \eqref{eq:equality_resolvent_A_wh_A_1} and \eqref{eq:equivalence_power_whA_1_delta_diag}, we infer that $\dot{\A}_1$ has a bounded $H^{\infty}$-calculus. Therefore $\A$ has a bounded $H^{\infty}$-calculus on $X$ by similarity with $\dot{\A}_1$. 
\end{proof}

\begin{proof}[Proof of Theorem \ref{t:calculus_pert_block_operators_notype}]
	The proof is almost analogous to the proof of Theorem~\ref{t:calculus_pert_block_operators}. The necessary modifications are in Step 2 and 3 of the part of the proof with Assumption \ref{ass:fractions_pm}\nameref{ass:fractions_-}. There, one has that 
	$C_{\delta}\in \calL(\Dd(A^{1-\delta}), \Dd(D^{-\delta}))$ and $B_{\delta}\in \calL(\Dd(D^{1-\delta}),\Dd(A^{-\delta}))$,
and then by interpolation $C$ and $B$ extend for $\eta\in [0,\delta]$  to operators
	\begin{align*}
	C_{\eta}\in \calL(\Dd(A^{1-\eta}), \Dd(D^{-\eta})) \quad \hbox{with} \quad \norm{C_{\eta}} \leq  c_A^{1-\eta/\delta}\norm{C_{\delta}}^{\eta/\delta}, \\
	B_{\eta}\in \calL(\Dd(D^{1-\eta}),\Dd(A^{-\eta})) \quad \hbox{with} \quad \norm{B_{\eta}} \leq  c_D^{1-\eta/\delta}\norm{B_{\delta}}^{\eta/\delta}.
	\end{align*}  
	By assumption one has for $\eta$ 
	by choosing $\delta\in (0,1)$ small enough 
	\begin{equation}\label{eq:assumption_epsilon_small_calculus_notype2}
	\|D^{-\eta}C_{\eta} x\|_{X_2}\leq c_A' \| A^{1-\eta}x\|_{X_{1}},
	\ \ \text{ and } \ \ 
	\|A^{-\eta}By\|_{X_1}\leq c_D' \| D^{1-\eta} y\|_{X_2}, 
	\end{equation}
	for all %$(x,y)\in \Dd(\diag^{1+\delta})\cap \Dd(\A_{0}^{\delta})$ 
	$x\in \Do(A)$ and $y\in \Do(D)$,
	and some $c_D'\geq c_D$ and $c_A'\geq c_A$, respectively, with
	\begin{equation*}
	c_A' <\frac{1}{c_D' \non^{\Rsec}_{\psi}(A)\non^{\Rsec}_{\psi}(D)}.
	\end{equation*}
	Since $\Do(A)\cap \Ran(A)\embed X_1$ and $\Do(D)\cap \Ran(D)\embed X_2$ are dense, one has from the above that $D^{-\eta}CA^{\eta-1}$ and $A^{-\eta}BD^{\eta-1}$ admit unique extensions %to bounded operators, namely
	\begin{equation*}
	G\stackrel{{\rm def}}{=}\overline{A^{-\eta}BD^{\eta-1}}\in \calL(X_2,X_1)\ \ \text{ and } 
	H\stackrel{{\rm def}}{=}\overline{D^{-\eta}CA^{\eta-1}}\in \calL(X_1,X_2).
	\end{equation*} 
	where  by \eqref{eq:assumption_epsilon_small_calculus_notype2}
	\begin{equation}
	\label{eq:smallness_eta_1_eta_2}
	\|G\|_{\calL(X_2,X_1)}\leq c_D' \quad \text{ and }\quad \|H\|_{\calL(X_1,X_2)}\leq c_A'.
	\end{equation}
	By assumption one has absolute convergence in $\calL(X_1)$ of
	$$
	M_1(\lambda)^{-1}=\sum_{n\geq 0} [B(\lambda-D)^{-1}C(\lambda-A)^{-1}]^n.
	$$ 
	 Next,  we rewrite the series conveniently using that 
	\begin{align*}
	B(\lambda-D)^{-1}C(\lambda-A)^{-1}
	=A^{\eta} G D(\lambda-D)^{-1} HA^{\eta-1} (\lambda-A)^{-1},
	\end{align*}
	and iterating the above argument, one can check that, for any $n\geq 1$,
	\begin{align*}
	&(\lambda-A)^{-1}[B(\lambda-D)^{-1}C(\lambda-A)^{-1}]^n\\
	&=
	A^{\eta}(\lambda-A)^{-1} [G D(\lambda-D)^{-1} H A(\lambda-A)^{-1}]^{n-1} 
	[G D(\lambda-D)^{-1} H] A^{1-\eta}(\lambda-A)^{-1}.
	\end{align*}
	Therefore
	\begin{equation}
	\label{eq:identity_A_M_1_smallness_H_infty_calculus}
	(\lambda-A)^{-1}M_1(\lambda)^{-1}=(\lambda-A)^{-1} +A^{\eta}(\lambda-A)^{-1}\mathcal{S}(\lambda)
	A^{1-\eta}(\lambda-A)^{-1},
	\end{equation}
	where
	$$
	\mathcal{S}(\lambda)\stackrel{{\rm def}}{=}
	\Big(\sum_{n\geq 0} [G D(\lambda-D)^{-1} H A(\lambda-A)^{-1}]^{n}\Big)  [G D(\lambda-D)^{-1} H].
	$$
	Reasoning as in the proof of Proposition \ref{prop:pert_block_operators}, one can check that \eqref{eq:smallness_eta_1_eta_2} implies that the above series expansion is absolutely convergent in $\calL(X_1)$ and
	\begin{equation}
	\label{eq:R_boundedness_S}
	\Rsec(\mathcal{S}(\lambda) : \lambda\in \C\setminus \Sigma_{\psi})<\infty.
	\end{equation}
	Thus, \eqref{eq:estimate_1_f_small_H_calculus} follows from \eqref{eq:identity_A_M_1_smallness_H_infty_calculus}, \eqref{eq:R_boundedness_S} and Lemma \ref{l:estimate_f_functional_calculus} applied with $S=T=A$. The rest of the proof is as in the one of Theorem~\ref{t:calculus_pert_block_operators}. The modifications above interpolate operators rather than $\Rsec$-bounded sets, and thereby the $K$-convexity assumption has been avoided. 
\end{proof}	

\subsection{$H^{\infty}$-calculus on Hilbert spaces}
\label{ss:H_calculus_Hilbert_space}
In this subsection we investigate the boundedness of the $H^{\infty}$-calculus assuming that $X$ is a Hilbert space, 
which we emphasize by
writing $H,H_1$ and $H_2$ instead of $X,X_1 $ and $X_2$, where the respective scalar products are denoted by $(\cdot|\cdot)_{H}$, $(\cdot|\cdot)_{H_1}$, and $(\cdot|\cdot)_{H_2}$. 

It is known that if $-T$ generates a $C_0$--semigroup of contractions on a  Hilbert space, then $T$ has a bounded $H^{\infty}$-calculus. %(see e.g.\ \cite[Theorem 10.4.22]{Analysis2}). 
The latter result is optimal if $-T$ generates an analytic semigroup. Indeed, a sectorial operator $T$ of angle smaller than $\pi/2$ on a  Hilbert space has a bounded $H^{\infty}$-calculus
if and only if $-T$ generates a contraction semigroup w.r.t.\ an equivalent Hilbertian norm (see e.g.\ \cite[Theorem 10.4.22]{Analysis2}).
In light of the Lumer-Phillips Theorem, compare e.g. \cite[Corollary G.4.5]{Analysis2}, we get the following criteria.

\begin{proposition}[Generation of $C_0$--semigroups of contractions]
	\label{prop:hilbert_space_positivity}
	Let Assumption \ref{ass:standing} be satisfied. Suppose that $-A$ and $-D$ generate $C_0$--semigroup of contractions.
	Let  $-1\in \rho(\A)$. Suppose that there exists $\gamma\in (0,\infty)$ such that
	\begin{equation}
	\label{eq:positivity_H_B_C}
	\gamma\,\Re(B h_2|  h_1)_{H_1}+ \Re(C h_1| h_2)_{H_2}\geq -\gamma\,\Re(A h_1| h_1)_{H_1}-\Re(D h_2| h_2)_{H_2} 
	\end{equation}
	 for all $h=(h_1,h_2)\in \Do(\A)$.
	Then $-\A$ generates a $C_0$--semigroup of contractions  on $H$. 
	In particular, $\A$ has a bounded $H^{\infty}$-calculus of angle $\angH(\A)=\om(\A)\leq\frac{\pi}{2}$.
\end{proposition}
\begin{proof}
Proposition \ref{prop:hilbert_space_positivity} follows from the  Lumer-Phillips Theorem because \eqref{eq:positivity_H_B_C} implies 
$$
\Re(\A h| h)_{H, \gamma}\geq 0 \quad \text{ for all }h\in H
$$ 
where  
$
(h|k)_ {H, \gamma} \stackrel{{\rm def}}{=} \gamma  (h_1|k_1)_{H_1}+(h_2|k_2)_{H_2} 
$
for $h=(h_1,h_2)$, $k=(k_1,k_2)\in H$. The statement on the bounded $H^\infty$-calculus follows by \cite[Theorem 10.2.24]{Analysis2}.
\end{proof}
%Let us collect further observation in the following

\begin{remark}\label{eq:cancellation_H_B_C}
%	\begin{enumerate}[(1)]
	%	\item By taking $(h_1,0)$ and $(0,h_2)$ in \eqref{eq:positivity_H_B_C}, it follows that $A$ and $D$ are dissipative operators. Hence, if $\rho(A)$ and $\rho(D)$ are not empty, then $-A$ and $-D$ generate $C_0$--semigroups of contractions. 
	%	\item 
	If $A$ and $D$ are dissipative operators, then \eqref{eq:positivity_H_B_C} holds provided
		\begin{equation*}
		\gamma\,\Re(B y|  x)_{H_1}+ \Re(C x| y)_{H_2}=0 \ \ \text{ for all }(x,y)\in \Do(\A).
		\end{equation*}	
%	\end{enumerate}
\end{remark}
The block operator matrix $\A$ is \emph{$\mathcal{J}$-symmetric} if $\mathcal{J} \A$ is symmetric, 
cf. \cite[Section 2.6]{Tretter}, where
\begin{align*}
\mathcal{J} \stackrel{{\rm def}}{=} \begin{bmatrix}
\id & 0 \\ 0 & -\id
\end{bmatrix}.
\end{align*}
 For $\mathcal{J}$-symmetric  operators $\A$ one has  $C\subseteq -B^\ast$, compare \cite[Proposition 2.6.1]{Tretter}. Combining this with Remark~\ref{eq:cancellation_H_B_C} and Proposition~\ref{prop:hilbert_space_positivity} one gets the following
\begin{corollary}[Bounded $H^\infty$-calculus for $\mathcal{J}$-symmetric operators]\label{cor:Jsymmetry}
	If $\A$ satisfies Assumption~\ref{ass:standing}, $A=A^\ast$, $D=D^\ast$, and it is $\mathcal{J}$-symmetric, then $\A$ has a bounded $H^\infty$-calculus with $\angH(\A)=\om(\A)\leq\frac{\pi}{2}$.
\end{corollary}

The following proposition describes negative fractional powers of $\A$ in terms of the ones of $\diag$. Below in Section~\ref{sec:extrapolation}, it will also be used to extrapolate the $H^{\infty}$-calculus in an $L^q$--setting.

\begin{proposition}[Negative fractional powers of $\A$ and $\diag$]
	\label{prop:fractions_Hilbert}
	Let Assumption \ref{ass:standing} be satisfied.
	Suppose that \eqref{eq:equivalence_norm_A_0_A} holds.
	Assume that $-\diag$ and $-\A$ generate  $C_0$--semigroups of contraction. Then
	$$
	\Ran(\A^{\beta})=
	\Ran(\diag^{\beta}) \ \text{ and }\ 
	\|\A^{-\beta} h\|_{H}\eqsim 
	\|\diag^{-\beta} h\|_{H} \text{ for all }h\in \Ran(\diag^{\beta})
	$$
	for all $\beta\in [0,\frac{1}{2})$. In particular $\Dd(\diag^{\gamma})=\Dd(\A^{\gamma})$ for all $\gamma\in (-\frac{1}{2},0]$.
\end{proposition}

\begin{proof}
	Let $\beta\in [0,\frac{1}{2})$. By \cite[Theorem 1.1]{Kato},  $\Do((\A^*)^{\beta})=\Do(\A^{\beta})$ and $\Do((\diag^*)^{\beta})=\Do(\diag^{\beta})$ with corresponding homogeneous estimates, namely
	$$
	\|\diag^{\beta} h\|_H \eqsim \|(\diag^*)^{\beta} h\|_H 
	\ \ \text{ and } \ \ 
	\|\A^{\beta} h\|_H \eqsim \|(\A^*)^{\beta} h\|_H
	$$
	for all $h\in H$.
	Combining the latter with Proposition \ref{prop:fractional_powers_positive}, one has $\Do((\A^*)^{\beta})=\Do((\diag^*)^{\beta})$ with a corresponding homogeneous estimate $\|\A^{\beta}h\|_{H}\eqsim \|(\diag^*)^{\beta}h \|_H$ for all $h\in \Do((\A^*)^{\beta})$. The claim now follows from \cite[Proposition 11]{Stokes}.
\end{proof}

\section{Extrapolation for consistent families of block operators}\label{sec:extrapolation}

In applications to (stochastic) partial differential equations many relevant  operators
can be studied not only in one setting but in a range of spaces. These
 typically depend on several parameters, e.g.\ integrability powers and/or Sobolev smoothness. Often there are certain special values of such parameters for which $\Rsec$-sectoriality and the boundedness of the $H^{\infty}$-calculus are  easier to investigate than for the general case. In this section we provide results which allow to extrapolate 
$\Rsec$-sectoriality and the boundedness of the $H^{\infty}$-calculus knowing the corresponding property for certain values of the parameters.
 
 \subsection{Assumptions and consistency}

 We begin by introducing the concept of a consistent family of operators. Let $S$ and $T$ be sectorial operators on $Y$ and $Z$, respectively. The operators $S,T$ are said to be \emph{consistent} if $(Y,Z)$ is an interpolation couple, i.e.\  $Y\embed V$ and $Z\embed V$ where  $V$ is a topological vector space, cf. e.g. \cite{Tr1}, and
$$
(\lambda-S)^{-1}|_{Y\cap Z}=  (\lambda -T)^{-1}|_{Y\cap Z}\qquad \text{ for all }\lambda<0.
$$
A family of sectorial operators $(T_{\theta})_{\theta\in I}$ on a family of Banach spaces $(Y_{\theta})_{\theta\in I}$, where $I\subseteq \R$ is an interval, is  \emph{consistent} if $T_{\theta}$ and $T_{\varphi}$ are consistent for all $\theta,\varphi\in I$.

\begin{remark}
\label{r:suff_condition_concistency}
In case that $Y\embed Z$, then $S$ and $T$ are consistent provided 
$
S=T|_{\Do(S)}.
$
This follows by noticing that by $D(S) = \{x \in  D(T)\cap Y \colon
Tx \in  Y\}$ one has $D(S)\cap D(T) = D(S)$ and  $(\lambda-T)^{-1}|_{Y}=(\lambda-S)^{-1}$.
\end{remark}

The following assumption is in force throughout this section.

\begin{assumption}
\label{ass:consistent_family}
Let $I=[\alpha,\beta]$ for $-\infty<\alpha<\beta<\infty$.
\begin{enumerate}[{\rm(1)}]
\item\label{it:space_complex_scale} 
Let $(X_{i,\theta})_{\theta\in I}$ for $i\in \{1,2\}$ be a family of Banach spaces such that $X_{i,\alpha}$ have non-trivial type for all $\theta\in I$ and $i\in \{1,2\}$. Moreover, for  $i\in \{1,2\}$ assume that  for the complex interpolation spaces
%, $\varphi,\theta\in I$ and $\gamma\in (0,1)$,
\begin{align*}
[X_{i,\theta},X_{i,\varphi}]_{\gamma}=X_{i,(1-\gamma)\theta+\gamma \varphi} \quad \hbox{for all } \varphi,\theta\in I \hbox{ and }	\gamma\in (0,1).
\end{align*}

\item\label{it:standing_theta} For each $\theta\in I$ the following hold:
\begin{itemize}
\item $A_{\theta}$ and $D_{\theta}$ are sectorial operators on $X_{1,\theta}$ and $X_{2,\theta}$, respectively;
\item $C_{\theta}\colon\Do(C_{\theta})\subseteq X_{1,\theta}\to X_{2,\theta}$ and $B_{\theta}:\Do(B_{\theta})\subseteq X_{2,\theta}\to X_{1,\theta}$ are linear operators with $\Do(A_{\theta})\subseteq \Do(C_{\theta})$ and $\Do(D_{\theta})\subseteq \Do(B_{\theta})$ and there exist $c_{D,\theta},c_{A,\theta},L_{\theta}\geq 0$ such that%, 
%for all $x\in \Do(A_{\theta})$ and $y\in \Do(B_{\theta})$,
\begin{align*}
\|C_{\theta} x\|_{X_{2,\theta}}&\leq c_{A,\theta} \| A_{\theta} x\|_{X_{1,\theta}} + L_{\theta} \|x\|_{X_{1,\theta}} \quad \hbox{for all } x\in \Do(A_{\theta}),\\
\|B_{\theta} y\|_{X_{1,\theta}}&\leq c_{D,\theta} \| D_{\theta} y\|_{X_{2,\theta}} + L_{\theta} \|y\|_{X_{2,\theta}} \quad \hbox{for all } y\in \Do(D_{\theta}).
\end{align*}
\end{itemize}
\item\label{it:consistency} $(A_{\theta})_{\theta\in I}$, $(D_{\theta})_{\theta\in I}$ are consistent families of operators.
\item\label{it:consistency_BC}
For all $\theta_1,\theta_2 \in I$,
 $x\in \Do(A_{\theta_1})\cap \Do(A_{\theta_2})$ and $y\in \Do(D_{\theta_1})\cap \Do(D_{\theta_2})$,
$$
B_{\theta_1} y =
B_{\theta_2} y, \qquad 
C_{\theta_1} x =
C_{\theta_2} x.
$$
\end{enumerate}
\end{assumption}
%\todo{Amru: Consistency to be defined for any, also non-sectorial, operators? Then conditions (3) and (4) on $A,B,C,D$ are more unified..}
In this subsection, we extend our standard notation as follows: For each $\theta\in I$ set $X_{\theta}\stackrel{{\rm def}}{=}X_{1,\theta}\times X_{2,\theta}$, $\Do(\diag_{\theta})=\Do(\A_{\theta})=\Do(A_{\theta})\times \Do(D_{\theta})$,  
\begin{align*}
&\diag_{\theta}
\stackrel{{\rm def}}{=}
\begin{bmatrix}
A_{\theta} & 0\\
0 &D_{\theta}
\end{bmatrix}
\quad \text{ and } \quad
\A_{\theta}
\stackrel{{\rm def}}{=}
\begin{bmatrix}
A_{\theta} & B_{\theta}\\
C_{\theta} &D_{\theta} 
\end{bmatrix}.
\end{align*}
\begin{remark} Assumption \ref{ass:consistent_family}\eqref{it:space_complex_scale} says that the families $(X_{\theta})_{\theta\in I}$ and $(X_{i,\theta})_{\theta\in I}$, $i\in\{1,2\}$, are  complex interpolation scales,
	compare e.g. \cite[Section V.1]{Am}. Assumption \ref{ass:consistent_family}\eqref{it:standing_theta} implies that $A_{\theta},B_{\theta},C_{\theta}$ and $D_{\theta}$ satisfy Assumption \ref{ass:standing} for all $\theta\in I$, and Assumption \ref{ass:consistent_family}\eqref{it:consistency} ensures that $(\diag_{\theta})_{\theta\in I}$ is a consistent family of operators.
\end{remark}

 Next we provide a sufficient condition for the consistency of $(\A_{\theta})_{\theta\in I}$ in terms of the operators %$M_j$. To this end, 
extending \eqref{eq:def_M_1_M_2}, where we set for all $\theta\in (0,1)$ and $\lambda\in \rho(A_\theta)\cap\rho(D_\theta)$ %, %we set
\begin{equation}
\label{eq:def_M_1_M_2_theta}
\begin{aligned}
M_{1,\theta}(\lambda) &
\stackrel{{\rm def}}{=}
\id - B_{\theta}(\lambda-D_{\theta})^{-1} C_{\theta}(\lambda- A_{\theta})^{-1} \in \calL(X_{1,\theta}),\\
M_{2,\theta}(\lambda) &
\stackrel{{\rm def}}{=}
\id -  C_{\theta}(\lambda- A_{\theta})^{-1}B_{\theta}(\lambda-D_{\theta})^{-1} \in \calL(X_{2,\theta}).
\end{aligned}
\end{equation}
%\todo{Antonio: Should we formulate all the results for one $M$?}

\begin{lemma}[Consistency of $(\A_{\theta})_{\theta\in I}$]
\label{l:consistency_A_theta}
Let Assumptions \ref{ass:consistent_family} be satisfied. Assume that $\A_{\theta}$ is sectorial for all $\theta\in I$, and that
\begin{equation}
\label{eq:equivalence_norm_A_0_A_theta}
\|\diag_{\theta} x\|_{X_{\theta}}\lesssim_{\theta} \|\A_{\theta} x\|_{X_{\theta}}\quad \text{ for all }x\in \Do(\diag_{\theta})\text{ and } \theta\in I.
\end{equation}
Then, for all $\theta_1,\theta_2\in I$, $\lambda<0$ and $j\in \{1,2\}$,
\begin{equation}
\label{eq:consistency_M_j}
M_{j,\theta_1}(\lambda) |_{X_{j,\theta_1}\cap X_{j,\theta_2}}=
M_{j,\theta_2}(\lambda) |_{X_{j,\theta_1}\cap X_{j,\theta_2}}.
\end{equation}
Moreover the family $(\A_{\theta})_{\theta\in I}$ is consistent if one of the following holds:
\begin{enumerate}[{\rm(1)}]
\item\label{it:consistency_M_j} For all $\lambda<0$, $j\in \{1,2\}$ and $\theta_1,\theta_2\in I$,
$$
M_{j,\theta_1}(\lambda)^{-1} |_{X_{j,\theta_1}\cap X_{j,\theta_2}}=
M_{j,\theta_2}(\lambda)^{-1} |_{X_{j,\theta_1}\cap X_{j,\theta_2}}.
$$
\item\label{it:consistency_for_nested_family}
For one $j\in \{1,2\}$ and for all $\theta_1,\theta_2\in I$
$$
%\text{ either } \quad 
X_{\theta_1,j} \embed X_{\theta_2,j} 
\quad \text{ or }\quad 
X_{\theta_2,j} \embed X_{\theta_1,j} .
$$
\end{enumerate}
\end{lemma}

Recall that $M_{j,\theta}(\lambda)$ is invertible for all $\lambda<0$ in case that $\A_{\theta}$ is sectorial by Corollary \ref{t:rsec_necessary_sufficient_condition}. In particular, the inverse in \eqref{it:consistency_M_j} is bounded. 
Condition \eqref{it:consistency_for_nested_family} can be easily checked in case that $X_{i,\theta}$ is an $L^p$-space on a finite measure space. Note that \eqref{it:consistency_for_nested_family} does \emph{not} follow from Remark \ref{r:suff_condition_concistency} since the claimed condition holds only for \emph{one} $j\in \{1,2\}$.

\begin{remark}[Optimality of \eqref{it:consistency_M_j} in Lemma \ref{l:consistency_A_theta}]
If $(\A_{\theta})_{\theta\in I}$ is consistent and  
$$
A_{\theta_1}|_{\Do(A_{\theta_1})\cap \Do(A_{\theta_2})} = A_{\theta_2}|_{\Do(A_{\theta_1})\cap \Do(A_{\theta_2})}, \quad
D_{\theta_1}|_{\Do(D_{\theta_1})\cap \Do(D_{\theta_2})} = D_{\theta_2}|_{\Do(D_{\theta_1})\cap \Do(D_{\theta_2})},
$$
then \eqref{it:consistency_M_j} in Lemma \ref{l:consistency_A_theta} holds. The claim follows from the identity
$
M_{j,\theta}(\lambda)^{-1}= \res_j(\lambda-\diag_{\theta})(\lambda-\A_{\theta})^{-1}\e_j 
$, compare Proposition~\ref{l:relation_resolvent_AD_calA}, where $\e_j\colon X_j\rightarrow X$ and $\res_j\colon X\rightarrow X_j$ are the extension and restriction operators,  respectively.
\end{remark}

\begin{remark}[Extrapolation of \eqref{eq:equivalence_norm_A_0_A_theta}]
\label{r:extrapolation_equivalence_norm_A_0_A}
The condition in \eqref{eq:equivalence_norm_A_0_A_theta} is the analogue of \eqref{eq:equivalence_norm_A_0_A}. If \eqref{eq:equivalence_norm_A_0_A_theta} holds for some $\theta=\theta^{\star}\in [\alpha,\beta]$ (e.g.\ if $0\in\rho(\A_{\theta^{\star}})$), then it also holds for all $\theta\in [\alpha,\beta]\cap (\theta^{\star}-\varepsilon,\theta^{\star}+\varepsilon)$ where $\varepsilon>0$ is small. This follows from Assumption \ref{ass:consistent_family}, Lemma \ref{l:one_side_A_0_A_estimate} and  Sneiberg's lemma, cf. \cite{S74}. For more details we refer to the proof of Theorem \ref{t:extrapolation_R} where a similar situation appears. 
\end{remark}

\begin{proof}[Proof of Lemma \ref{l:consistency_A_theta}]
We begin by proving \eqref{eq:consistency_M_j}, where it is enough to show that, for all $\lambda<0$,
\begin{align}
\label{eq:C_A_consistency}
C_{\theta_1}(\lambda-A_{\theta_1})^{-1}|_{X_{1,\theta_1}\cap X_{1,\theta_2}}
&=
C_{\theta_2}(\lambda-A_{\theta_2})^{-1}|_{X_{1,\theta_1}\cap X_{1,\theta_2}},\\
\label{eq:B_D_consistency}
B_{\theta_1}(\lambda-D_{\theta_1})^{-1}|_{X_{2,\theta_1}\cap X_{2,\theta_2}}
&=
B_{\theta_2}(\lambda-D_{\theta_2})^{-1}|_{X_{2,\theta_1}\cap X_{2,\theta_2}}.
\end{align}
 Note that, for all $\lambda<0$ $$(\lambda-A_{\theta_1})^{-1}x=(\lambda-A_{\theta_2})^{-1}x\in \Do(A_{\theta_1})\cap \Do(A_{\theta_2}) \quad \hbox{for } x\in X_{1,\theta_1}\cap X_{1,\theta_2},$$ 
by consistency of $(A_{\theta})_{\theta\in I}$. Therefore, by Assumption \ref{ass:consistent_family}\eqref{it:consistency_BC}, 
\begin{align*}
C_{\theta_1}(\lambda-A_{\theta_1})^{-1} x 
=C_{\theta_1}(\lambda-A_{\theta_2})^{-1} x
=C_{\theta_2}(\lambda-A_{\theta_2})^{-1} x \quad  \hbox{for } x\in X_{1,\theta_1}\cap X_{1,\theta_2}
\end{align*}
which proves \eqref{eq:C_A_consistency}, and \eqref{eq:B_D_consistency} follows similarly.

Statement \eqref{it:consistency_M_j} holds due to Proposition~\ref{l:relation_resolvent_AD_calA}\eqref{eq:resolvent_formula_block_matrix_I} for $\A=\A_{\theta}$, the claim follows from the assumption, the consistency of $(A_{\theta})_{\theta\in I}$, $(D_{\theta})_{\theta\in I}$, \eqref{eq:C_A_consistency}, and \eqref{eq:B_D_consistency}.

To prove \eqref{it:consistency_for_nested_family} %: Here 
we check the condition in \eqref{it:consistency_M_j}. Fix $\lambda<0$ and $j\in \{1,2\}$ and $\theta_1,\theta_2\in I$ %. Without loss of generality we assume 
 with $\theta_1<\theta_2$. For simplicity assume the first case, that is,  $X_{j,\theta_1}\cap X_{j,\theta_2}= X_{j,\theta_1}$, the second case follows similarly. Thus, by \eqref{eq:consistency_M_j},
we have $
M_{j,\theta_1}(\lambda)=
M_{j,\theta_2}(\lambda)|_{X_{j,\theta_1}}
$. Since $\A_{\theta_i}$ is sectorial, $M_{j,\theta_i}(\lambda)$ is invertible by Theorem \ref{t:rsec_necessary_sufficient_condition_II} and therefore
\begin{equation}
\label{eq:consistency_inverse_simple_case}
M_{j,\theta_2}(\lambda)^{-1}|_{X_{j,\theta_1}}=M_{j,\theta_1}(\lambda)^{-1}
.
\end{equation}
Again, using that $X_{j,\theta_1}\cap X_{j,\theta_2}= X_{j,\theta_1}$, the claim follows from \eqref{eq:consistency_inverse_simple_case} and \eqref{it:consistency_M_j}.
%
%The claim \eqref{it:consistency_for_nested_family_inverse_direction} follows from \eqref{it:consistency_for_nested_family} applied to the family $(\A_{-\theta})_{\theta\in [-\beta,-\alpha]}$.
\end{proof}

\subsection{Extrapolation results} Here we list the main results of this section, the proof will be given in Subsection \ref{ss:proofs_extrapolation} below. We begin by analyzing $\Rsec$-sectoriality.

\begin{theorem}[Extrapolation of $\Rsec$-sectoriality]
\label{t:extrapolation_R}
Let Assumption \ref{ass:consistent_family} with $L_{\theta}=0$ be satisfied for all $\theta\in I$. Suppose that the following are satisfied for some $\psi\in (0,2\pi)$: 
\begin{enumerate}[{\rm(a)}]
\item $\diag_{\theta}$ is $\Rsec$-sectorial of angle $<\psi$ for all $\theta\in (\alpha,\beta)$; 
\item There exists a $\thesp\in (\alpha,\beta)$ such that $\A_{\thesp}$ is $\Rsec$-sectorial of angle $< \psi$;
\item\label{cond:t:extrapolation_R _c}$\overline{\Ran(\A_{\theta})}=X_{\theta}$ for all $\theta\in I$. 
\end{enumerate}
Then there exists $\varepsilon>0$ such that %\todo{Antonio: Problem: Should we also prove that the corresponding family is consistent? I can prove it if the following is true $[X\cap Y,Y]_{\theta}=[X,Y]_{\theta}\cap Y$. Does it hold in general?
%Amru: Yes, \cite[1.17.1 Theorem 1]{Tr1} provided there is a bounded projection onto $Y$. 
%}
$$
\A_{\theta} \text{ is $\Rsec$-sectorial of angle $\leq \psi$ for all $|\theta-\thesp|<\varepsilon$}.
$$
Finally, if $X_{j,\alpha},X_{j,\beta}$ are reflexive for $j\in\{1,2\}$, then condition~\eqref{cond:t:extrapolation_R _c} %$\overline{\Ran(\A_{\theta})}=X_{\theta}$ 
can be omitted.
\end{theorem}
\begin{remark}
Fackler showed in \cite[Corollary 6.4]{F14_extrapolation_MR} that in general $\Rsec$-sectoriality  does \emph{not} extrapolate. Therefore Theorem~\ref{t:extrapolation_R} 
is somewhat surprising,  and
it heavily relies on the block structure of the operator $\A_{\theta}$ and the assumptions on $\diag_{\theta}$.
\end{remark}
Next we turn our attention to the $H^{\infty}$-calculus. In contrast to Theorem \ref{t:extrapolation_R} the following result requires  conditions on the angle. Fortunately, this is always the case in applications %to parabolic problems, that is,  
with $\omega_{H^\infty}(A)\vee \omega_{H^{\infty}}(D)< \pi/2$ (see \cite[Theorem 10.4.22]{Analysis2}).% (S)PDEs.
%\todo{Amru: $\psi\leq \pi/2$}

\begin{theorem}[Extrapolation of the $H^{\infty}$-calculus]
\label{t:extrapolation_H}
Let Assumption \ref{ass:consistent_family} be satisfied. Assume that $X_{i,\alpha}$ is a Hilbert space for $i\in \{1,2\}$. Suppose that $(\A_{\theta})_{\theta\in I}$ is a consistent family of sectorial operators and that 
\begin{equation}
\label{eq:estimate_A_A_0_theta_alpha_beta}
\|\diag_{\theta} x\|_{X_{\theta}}\lesssim \|\A_{\theta} x\|_{X_{\theta}}\quad \text{ for all }x\in \Do(\A_{\theta})\text{ and } \theta\in \{\alpha,\beta\}.
\end{equation}
Let the following be satisfied:
\begin{enumerate}[{\rm(a)}]
\item $\diag_{\alpha}$ and $\diag_{\beta}$ have a bounded $H^{\infty}$-calculus.
\item $-\A_{\alpha}$ generates a $C_0$-semigroup of contractions.
\item $\A_{\beta}$ is $\Rsec$-sectorial.
\end{enumerate}
Then for all $\theta\in [\alpha,\beta)$ the following hold.
\begin{enumerate}[{\rm(1)}]
\item\label{it:extrapolation_H_A_theta} $\A_{\theta}$ has a bounded $H^{\infty}$-calculus of angle 
$$
\omega_{H^\infty}(\A_{\theta})\leq \Big(\frac{\beta-\theta}{\beta-\alpha}\Big)\omega(\A_{\alpha})+\Big(\frac{\theta-\alpha}{\beta-\alpha}\Big)\angR(\A_{\beta}).
$$
\item\label{it:description_negative_fractional_powers} for all $\delta\in (0,\frac{1}{2}\frac{\beta-\theta}{\beta-\alpha})$, one has $\Ran(\A_{\theta}^{\delta})=
\Ran(\diag_{\theta}^{\delta})$ and
\begin{equation*}
\|\A_{\theta}^{-\delta} x\|_{X_{\theta}}\eqsim 
\|\diag_{\theta}^{-\delta} x\|_{X_{\theta}}\ \  	\text{ for all }x\in \Ran(\diag_{\theta}^{\delta}).
\end{equation*}
In particular $\Dd(\A_{\theta}^{-\delta})=
\Dd(\diag_{\theta}^{-\delta})$ for all $\delta\in (0,\frac{1}{2}\frac{\beta-\theta}{\beta-\alpha})$. 
\end{enumerate}
\end{theorem}

A condition for a block operator matrix  to generate a $C_0$-semigroup of contractions has been discussed in Subsection \ref{ss:H_calculus_Hilbert_space}. %Recall that $-\A_{\alpha}$ generates a $C_0$--semigroup of contraction (up to re-norming $X_{\alpha}$) in case it has a bounded $H^{\infty}$-calculus of angle $<\frac{\pi}{2}$. 
Conditions for the consistency of the family $(\A_{\theta})_{\theta \in I }$ have been given in Lemma \ref{l:consistency_A_theta}. 
Note that Theorem \ref{t:extrapolation_H}\eqref{it:description_negative_fractional_powers} complements Propositions \ref{prop:fractional_powers_positive} and \ref{prop:fractions_Hilbert} since it also holds in a non-Hilbertian setting. 

%Before starting with the proofs, we give a consequence of Theorems \ref{t:extrapolation_R} and \ref{t:extrapolation_H}.

\begin{corollary}[Extrapolation of $H^{\infty}$-calculus for small $\theta$]
\label{cor:extrapolation}
Let Assumption \ref{ass:consistent_family} be satisfied with $L_{\theta}=0$. 
 Suppose that $(\A_{\theta})_{\theta\in I}$ is a consistent family of sectorial operators and that \eqref{eq:estimate_A_A_0_theta_alpha_beta} holds. 
 Let the following be satisfied for some $\thesp\in I$ and $\psi\in (0,\frac{\pi}{2}]$: 
\begin{enumerate}[{\rm(a)}]
\item $\diag_{\theta}$ has a bounded $H^{\infty}$-calculus of angle $< \psi$ for all $\theta \in I$;
\item $X_{i,\gamma}$ is a Hilbert space for $i\in \{1,2\}$;
\item $-\A_{\thesp}$ generates a $C_0$-contraction semigroup; 
\item $\A_{\thesp}$ is sectorial of angle $<\psi$.
\end{enumerate}
Then there exists $\varepsilon>0$ such that, for all $|\theta-\thesp|< \varepsilon$, the operator $\A_{\theta}$  has a bounded $H^{\infty}$-calculus of angle $\angH(\A_{\theta})\leq \psi$. 
%$$
%\A_{\theta}\text{ has a bounded $H^{\infty}$-calculus of angle $\angH(\A_{\theta})\leq \psi$. }
%$$ 
\end{corollary}

\begin{proof}%[Proof of Corollary \ref{cor:extrapolation}]
By Theorem \ref{t:extrapolation_R}, $\A_{\theta}$ is $\Rsec$-sectorial for all $|\thesp-\theta|<\varepsilon$ for some $\varepsilon>0$. Set $\gamma_{\pm}\stackrel{{\rm def}}{=}\gamma\pm \frac{\varepsilon}{2}$. The claim follows by applying Theorem \ref{t:extrapolation_H} to the families $(\A_{\theta})_{\theta\in [\gamma,\gamma_{+})}$ and $(\A_{-\theta})_{\theta\in [-\gamma,-\gamma_{-})}$.
\end{proof}

\subsection{Proof of Theorems \ref{t:extrapolation_R} and \ref{t:extrapolation_H}}
\label{ss:proofs_extrapolation}
%%%%%

An important ingredient for our proofs here is Sneiberg's lemma, cf. \cite{S74} and  \cite[Theorem 2.3 and Theorem 3.6]{TV88}. It has been used already in the context of $L^p$-theory to extrapolate $\Rsec$-sectoriality, see \cite{Kunstmann2005}. 
\begin{proof}[Proof of Theorem \ref{t:extrapolation_R}]%\todo{Antonio: 
%This proof  Theorem \ref{t:extrapolation_R} is similar to the one in the Appendix. We should think if we can shorten one of these proofs.}
Let us begin by noticing that, up to replacing $(\A_{\theta})_{\theta\in I}$ by $(\widehat{A}_{\eta})_{\eta\in [0,1]}$ with $\widehat{\A}_{\eta}= \A_{\alpha+\eta(\beta-\alpha)}$, we may assume $I=[0,1]$. Then, one has for $M_{j,\theta}$  defined in \eqref{eq:def_M_1_M_2_theta} that for all $\lambda\in \complement\overline{\Sigma_{\psi}}$, $\theta_1,\theta_2\in (0,1)$ and $j\in \{1,2\}$, 
\begin{equation}
\label{eq:M_j_equality_consistency}
M_{j,\theta_1}(\lambda) x = M_{j,\theta_2}(\lambda) x \quad \text{ for all }  \ x\in X_{j,\theta_1}\cap X_{j,\theta_2}.
\end{equation}
To see this, recall that, by Lemma \ref{l:consistency_A_theta}, \eqref{eq:M_j_equality_consistency} holds for all $\lambda<0$. Thus,  by the holomorphicity of the maps $\complement\overline{\Sigma_{\psi}}\ni \lambda\mapsto  M_{j,\theta_i}(\lambda) x\in X_{j,\theta_1}+X_{j,\theta_2}$ for all $x\in X_{j,\theta_1}\cap X_{j,\theta_2}$ and $i\in\{1,2\}$, \eqref{eq:M_j_equality_consistency} holds even for all  $\lambda\in \complement\overline{\Sigma_{\psi}}$.
 
To prove $\Rsec$-sectoriality by Theorem \ref{t:rsec_necessary_sufficient_condition_II} it is enough to show the existence of $\varepsilon\in (0,1)$
such that for $j\in \{1,2\}$
\begin{equation}
\label{eq:R_sectoriality_extrapolation}
\begin{aligned}
\text{for all $|\theta-\thesp|<\varepsilon$, }
\text{$M_{j,\theta}(\lambda)$ is invertible and }\Rsec\big(M_{j,\theta}(\lambda)^{-1}: \lambda\in \complement\overline{\Sigma_{\psi}}\big)<\infty.
\end{aligned}
\end{equation}

By assumption, the above statement holds for $\theta=\thesp$. For $\theta\neq \thesp$, we use Sneiberg's lemma, see \cite{S74}, and here we will employ its quantitative version given in \cite[Theorem 2.3 and Theorem 3.6]{TV88} (see also \cite[Subsection 1.3.5]{EgPhDthesis}). 
To this end, fix $N\geq 1$, and for a Banach space $Y$ and a Rademacher sequence $(\varepsilon_j)_{j\geq 1}$ we denote by $
\varepsilon_N(Y)$ the space $Y^N$ endowed with the norm
\begin{equation}
\label{eq:norm_varepsilon_N}
\|(x_j)_{j=1}^N\|_{\varepsilon_N(Y)}\stackrel{{\rm def}}{=}\E\Big\|\sum_{j=1}^N\varepsilon_j x_j\Big\|_{Y}.
\end{equation}
By \cite[Theorem 7.4.16]{Analysis2}, $\big(\varepsilon_N(X_{\theta})\big)_{\theta\in [0,1]}$ is a complex scale, i.e.
\begin{equation}
\label{eq:identification_complex_interpolation_varepsilon_spaces}
[\varepsilon_N(X_{\theta_1}),\varepsilon_N(X_{\theta_2})]_{\delta}
=\varepsilon_N(X_{\theta_1(1-\delta)+\delta\theta_2}) 
\text{ for all }\theta_1,\theta_2\in [0,1],  \delta\in (0,1).
\end{equation}
By $K$-convexity of $X_{0}$ and $ X_{1}$, and the fact that $X^{(\theta)}=[X_0,X_1]_{\theta}$,  
the constants in the above identification are independent of $N\geq 1$. 

Next fix $(\lambda_k)_{k=1}^N\subseteq \complement\overline{\Sigma_{\psi}}$ and $j\in\{1,2\}$. Let
%Since $M^{(\theta_1)}_j(\lambda)=M^{(\theta_2)}_j(\lambda)$ for all $\theta_1\neq \theta_2$ by Assumption \ref{ass:consistent_family}, there exists a solely determinate bounded linear operator 
$$
\mathcal{T}_j: \varepsilon_N(X_{j,0})+\varepsilon_N(X_{j,1})\to
\varepsilon_N(X_{j,0})+\varepsilon_N(X_{j,1})
$$
be given by \begin{equation}
\label{eq:M_j_mathcal_T_relation}
\mathcal{T}_j \x = \big( M_{j,0} (\lambda_k) x_{k,0}+ M_{j,1} (\lambda_k) x_{k,1} \big)_{k=1}^N 
\end{equation}
where $\x=(x_k)_{k=1}^N$, $x_k=x_{k,0}+x_{k,1}$ and $x_{k,0}\in X_{j,0}$, $x_{k,1}\in X_{j,1}$. By \eqref{eq:M_j_equality_consistency}, the right hand side in \eqref{eq:M_j_mathcal_T_relation} does not depend on the decomposition $x_k=x_{k,0}+x_{k,1}$ and therefore $\mathcal{T}_j$ is well defined.

%By \eqref{eq:M_j_mathcal_T_relation} and the fact that $X_{j,0}\cap X_{j,1}\embed [X_{j,0},X_{j,1}]_{\theta}$ is dense for all $\theta\in (0,1)$, one has 
Since $\mathcal{T}_j \x= \big( M_j^{(i)} (\lambda_k) x_k^{(i)} \big)_{k=1}^N$ for all $\x=(x_k)_{k=1}^N\in \varepsilon_N(X_j^{(i)})$ and $\diag$ is $\Rsec$-sectorial of angle $\leq \psi$,
\begin{equation}
\label{eq:bound_pm_extrapolation}
\|\mathcal{T}_j\|_{\calL(\varepsilon_N(X_{j,i}))} \leq 1+ c_A c_D \non_{\psi}^{\Rsec}(A_i)\non_{\psi}^{\Rsec}(D_i),
\quad \text{ for all }i\in \{0,1\},
\end{equation}
where we also used Assumption \ref{ass:consistent_family} with $L_{\theta}=0$.

By complex interpolation and \eqref{eq:identification_complex_interpolation_varepsilon_spaces}, there exists $C>0$ independent of $N\geq 1$, the choice of $(\lambda_k)_{k=1}^N\subseteq \complement\overline{\Sigma_{\psi}}$ and $\theta\in (0,1)$ such that
\begin{equation}
\label{eq:norm_independent_of_N_lambda}
\Big\|\mathcal{T}_j|_{\varepsilon_N(X_{j,\theta})}\Big\|_{ \calL(\varepsilon(X_{j,\theta})}\leq C.
\end{equation}
By \eqref{eq:M_j_equality_consistency} and the density of the embedding $X_{j,0}\cap X_{j,1}\embed [X_{j,0},X_{j,1}]_{\theta}$,
\begin{equation}
\label{eq:T_j_M_theta}
\mathcal{T}_j|_{\varepsilon(X_{j,\theta})} \x = \big( M_j^{(\theta)}(\lambda_k)x_k)_{k=1}^N \text{ for all }\x=(x_k)\in \varepsilon_N(X_{j,\theta})\text{ and }\theta\in (0,1).
\end{equation}
Combining \eqref{eq:T_j_M_theta}, Theorem \ref{t:rsec_necessary_sufficient_condition_II} and the $\Rsec$-sectoriality of $\A_{\thesp}$, one can check that $\mathcal{T}_j|_{\varepsilon(X_{j,\thesp})}$ is invertible and the norm of its inverse is independent of $N\geq 1$ and the choice of $(\lambda_k)_{k=1}^N\subseteq \complement\overline{\Sigma_{\psi}}$.  Thus, by Sneiberg's lemma and \eqref{eq:norm_independent_of_N_lambda}, there exist $\varepsilon \in (0,1)$, $C'>0$ independent  of $N\geq 1$ and the choice of $(\lambda_k)_{k=1}^N\subseteq \complement\overline{\Sigma_{\psi}}$ such that $\mathcal{T}_j|_{\varepsilon_N(X_{j,\theta})}$ is invertible for all $|\theta-\thesp|<\varepsilon$ and the norm of its inverse is $\leq C'$. 

By \eqref{eq:T_j_M_theta}, the arbitrariness of $N\geq 1$ and $(\lambda_k)_{k=1}^N$, one can check that $M_j(\lambda)$ is invertible for all $\lambda\in \complement\overline{\Sigma_{\psi}}$ and \eqref{eq:R_sectoriality_extrapolation} holds. 
\end{proof}

%It remains to prove  Theorem \ref{t:extrapolation_H}.

\begin{proof}[Proof of Theorem \ref{t:extrapolation_H}]
As in the proof of Theorem \ref{t:extrapolation_R}, we may assume $I=[0,1]$.

\eqref{it:extrapolation_H_A_theta}: Fix $\theta\in (0,1)$. It is enough to show that $\A_{\theta}$ has a bounded $H^{\infty}$-calculus. The bound on the angle follows from \cite[Corollary 3.9]{KKW} and the fact that $\angR(\A_{0})=\omega(\diag)$ since $X_0$ is a  Hilbert space.
The idea is to apply \cite[Theorem 1]{Stokes} to $\diag_{\theta}$ and $A_{\theta}$ extending the argument in \cite[Corollary 7]{Stokes}. Note that, by the consistency of $(\diag_{\theta})_{\theta\in I}$ and the fact that $\diag_k$ has a bounded $H^{\infty}$-calculus for $k\in \{0,1\}$, it follows that $\diag_{\theta}$ has a bounded $H^{\infty}$-calculus as well (cf.\ \cite[Proposition 4.9]{KKW}).

Let $\varphi(z)=\psi(z)=z(1+z)^{-2}$ and fix $\gamma\in (0,\frac{1}{2})$. 
By Propositions \ref{prop:fractional_powers_positive}, \ref{prop:fractions_Hilbert} and \cite[Proposition 3]{Stokes}, for all $\ell\in \Z$,
\begin{equation}
\label{eq:bound_1_R}
\begin{aligned}
\sup_{s,t\in [1,2]}\sup_{j\in \Z} 
\Big\|\varphi(s2^{j+\ell} \A_{0}) \psi(t 2^j \diag_0)\Big\|_{X_{0}} 
&\lesssim 2^{- \gamma |\ell| },\\
\sup_{s,t\in [1,2]}\sup_{j\in \Z} \Big\| \psi(t 2^j \diag_0) \varphi(s2^{j+\ell} \A_{0})
\Big\|_{X_{0}}
&\lesssim 2^{- \gamma |\ell| } ,
\end{aligned}
\end{equation}
where we also used that $X_0$ is a  Hilbert space. 

By \cite[Lemma 3.3]{KKW} and the $\Rsec$-sectoriality of $\diag_{1}$ and $\A_{1}$ we get
\begin{equation}
\begin{aligned}
\label{eq:bound_2_R}
\sup_{\ell\in \Z}\sup_{s,t\in [1,2]}
\Rsec\Big(\varphi(s2^{j+\ell} \A_{1}) \psi(t 2^j \diag_{1})\colon j\in \Z\Big)<\infty,\\
\sup_{\ell\in \Z}\sup_{s,t\in [1,2]}
\Rsec\Big( \psi(t 2^j \diag_1) \varphi(s2^{j+\ell} \A_{1})\colon j\in \Z\Big)<\infty.
\end{aligned}
\end{equation}
Reasoning as in the proof of Theorem \ref{t:extrapolation_R}, by consistency of $(\A_{\theta})_{\theta\in I}$ and the Dunford representation of $\phi(\xi A),\varphi(\eta B)$ for $\eta,\xi>0$ (see \eqref{eq:dunford_H_calculus_definition}), one can check that the operators $\varphi(s2^{j+\ell} \A_{k}) \psi(t 2^j \diag_{k})$ and $\varphi(s2^{j+\ell} \A_{k}) \psi(t 2^j \diag_{k})$ for $k\in \{0,1\}$ coincide on $X_{0}\cap X_1$, and therefore we can interpolate the bounds \eqref{eq:bound_1_R}-\eqref{eq:bound_2_R} obtaining
\begin{align}
\label{eq:bound_3_R_1}
\sup_{s,t\in [1,2]}\Rsec\Big(\varphi(s2^{j+\ell} \A_{\theta}) \psi(t 2^j \diag_{\theta})\colon j\in \Z\Big)
&\lesssim 2^{-\gamma(1-\theta)|\ell|},\\
\label{eq:bound_3_R_2}
\sup_{s,t\in [1,2]}\Rsec\Big(\psi(t 2^j \diag_{\theta})\varphi(s2^{j+\ell} \A_{\theta})\colon j\in \Z\Big)
&\lesssim  2^{-\gamma(1-\theta)|\ell|},
\end{align}
where we have used \cite[Proposition 8.4.4]{Analysis2} and the fact that $X_0,X_1$ have non-trivial type and therefore are $K$-convex due to Pisier's theorem (see e.g.\ \cite[Theorem 7.4.23]{Analysis2}). 
By $K$-convexity of $X_k$ for $k\in \{0,1\}$, \eqref{eq:bound_3_R_2} and \cite[Proposition 8.4.1]{Analysis2} we also get
\begin{equation}
\label{eq:bound_3_R_3}
\Rsec\Big(\big(\varphi(s2^{j+\ell} \A_{\theta})\big)^* \big(\psi(t 2^j \diag_{\theta})\big)^*\colon j\in \Z\Big)
\lesssim  2^{-\gamma(1-\theta)|\ell|}.
\end{equation}
Recall that $\diag_{\theta}$ has a bounded $H^{\infty}$-calculus. By \eqref{eq:bound_3_R_1} and \eqref{eq:bound_3_R_3}, $\A_{\theta}$ has a bounded $H^{\infty}$-calculus  due to \cite[Theorem 1]{Stokes}. 

\eqref{it:description_negative_fractional_powers}: \cite[Theorem 1]{Stokes} and \eqref{eq:bound_3_R_1}, \eqref{eq:bound_3_R_3} also yield, for all $\delta\in (0,\gamma(1-\theta))$,
$$
\Ran(\A_{\theta}^{\delta})=
\Ran(\diag_{\theta}^{\delta})\quad \text{ and }\quad 
\|\A_{\theta}^{-\delta}x\|_{X_{\theta}}\eqsim  \|\diag_{\theta}^{-\delta}x \|_{X_{\theta}}\quad \text{ for all }x\in \Ran(\diag_{\theta}^{\delta}).
$$
The conclusion follows by letting $\gamma\uparrow \frac{1}{2}$ and recalling that $\alpha=0,\beta=1$.
\end{proof}

\section{Applications}\label{sec:applications}  
The analysis of quasi- or semi-linear problems in maximal $L^p_t$-regularity spaces 
typically comes in two steps:
First a linearization is considered for which one has to prove maximal $L^p_t$-regularity. Then, as a second step, Lipschitz estimates on the non-linearities are needed in certain interpolation spaces to apply Banach's fixed point theorem, see e.g. \cite{ AV19_QSEE_1, AV19_QSEE_2, PruSim04, PrussWeight1, PrussWeight2, pruss2016moving}. In this section we focus on properties of the linearized problems which are relevant to characterize the relevant interpolation spaces, and which are also helpful for the stability analysis for the non-linear problem.  

We begin by deriving some consequences of our results for triangular matrices. These results are used in Subsections~\ref{subsec:simplified_LCD}-\ref{subsec:2ndOrder}.

\subsection{The block triangular case}
\label{s:special_cases_dominant}
In this subsection we consider a block triangular diagonally dominant operator matrix
 $\A$, this means that $C=0$,
then the statements in Sections \ref{sec:sec_Rsec} and \ref{sec:Hinfty} simplify considerably. Using classical results for bounded perturbations, one can include the case   
 $C\in \calL(X_1,X_2)$ as well.
\begin{corollary}\label{cor:tridiagonal}
	Let 	
	\begin{align*}
	\A
	=\begin{bmatrix}
	A & B  \\
	0 & D 
	\end{bmatrix} ,	 \ \  \hbox{with} \ \  B\colon \Do(D)\subseteq X_2\rightarrow X_1,\  \hbox{ and } \
	\norm{By}_{X_1}\lesssim\norm{D y}_{X_2} \ \  y\in \Do(D). 	
	\end{align*}
	\begin{enumerate}[{\rm(1)}]
		\item\label{it:tridiagonal_sectorial} If $A$ and $D$ are sectorial with angles $\omega(A)$ and  $\omega(D)$, respectively,  
		then $\A$ is sectorial with angle $\omega(\A)\leq \omega(A)\vee \omega(D)$.
		\item\label{it:tridiagonal_R_sectorial} If $A$ and $D$ are $\Rsec$-sectorial with angles $\angR(A)$ and $\angR(D)$, respectively,  
		then $\A$ is $\Rsec$-sectorial with $\Rsec$-sectoriality angle $\angR(\A)\leq \angR(A)\vee \angR(D)$.
		\item\label{it:tridiagonal_H_calculus} %If in addition to the assumptions in \eqref{it:tridiagonal_R_sectorial}, 
		If $A$ and $D$ have a bounded $H^{\infty}$-calculus of angle $\angH(A)$ and $\angH(D)$, respectively, 
		and there exist $\delta>0$, such that 
		\begin{align*}
	\Ran(B)&\subseteq \Ran(A^{\delta})  \text{ and }	\| A^{-\delta}B y\|_{X_1}\lesssim \|D^{1-\delta} y\|_{X_2} \quad \hbox{for all }x\in \Do(D), \quad \hbox{or} \\
		B(\Do(D^{1+\delta}))&\subseteq \Do(A^{\delta}) \text{ and }
	\|A^{\delta}B y\|_{X_1}\lesssim \| D^{1+\delta} y\|_{X_2} \quad \hbox{for all } y\in \Do(D^{1+\delta}).
		\end{align*}
		Then $\A$ has a bounded $H^{\infty}$-calculus on $X$ of angle $\leq \angH(A)\vee\angH(D)$.
	\end{enumerate}
\end{corollary}	
\begin{proof}
	In the situation of Corollary~\ref{cor:tridiagonal} Assumption~\ref{ass:standing} holds with $L=0$. 
	Statements \eqref{it:tridiagonal_sectorial} and \eqref{it:tridiagonal_R_sectorial} follow from Proposition~\ref{prop:pert_block_operators}, and statement \eqref{it:tridiagonal_H_calculus} follows from Theorem~\ref{t:calculus_pert_block_operators_notype}, where one has in each case $\varepsilon=0$.
\end{proof}

%\subsection{The case with one bounded diagonal block}\label{prop:A=0}

The diagonally dominant case with $A=0$ has also a particular structure, that is, 
\begin{align*}
\A
=\begin{bmatrix}
0 & B  \\
C & D 
\end{bmatrix},	
\end{align*}
 where by a bounded perturbation argument one can include also $A\in \calL(X_1)$. Applications for this case are the artificial Stokes system in Subsection~\ref{subsec:art_Stokes} and second order Cauchy problems with strong damping in Subsection~\ref{subsec:2ndOrder}.
 
\begin{corollary}\label{prop:Azero}
Let $A$ be bounded,  $D$ be a sectorial operator, and $\A$
	satisfy Assumption~\ref{ass:standing}. Then the following hold.
	\begin{enumerate}[{\rm(1)}]
		\item\label{it:Azero_sectorial} For all $\psi\in (\om(D),\pi)$, there exists $\nu\geq 0$ such that $\nu+\A$ is sectorial with angle $\om(\nu+\A)\leq \psi$.
		\item\label{it:Azero_Rsectorial_H_infty} If $D$ is $\Rsec$-sectorial (resp.\ has a bounded $H^{\infty}$-calculus), then \eqref{it:Azero_sectorial} holds with $\om(\nu+\A)$ replaced by $\angR(\nu+\A)$ (resp.\ $\angH(\nu+\A)$).
	\end{enumerate}
\end{corollary}

\begin{remark}\label{rem:Azero}
	The condition $A\in \calL(X_1)$ together with $\A$ diagonally dominant implies already that $C\colon \Do(A)=X_1\rightarrow X_2$ is bounded, while  $B$ can be unbounded. Note that, comparing Corollary~\ref{prop:Azero}\eqref{it:Azero_Rsectorial_H_infty} with $\om$ replaced by $\angH$ and the results in Section \ref{sec:Hinfty},  then one observes that for $\A'$ below  Assumptions~\ref{ass:fractions_pm}$(+)$ trivially holds while $\A$ can violate both Assumptions~\ref{ass:fractions_pm}$(\pm)$, compare Remark~\ref{rem:McIntosh_Yagi}.
\end{remark}

\begin{proof}
	By assumption $A\in \calL(X_1)$, and by Remark~\ref{rem:Azero} we infer $C\in \calL(X_1,X_2)$.
	Considering
		\begin{align*}
		\A'\stackrel{{\rm def}}{=}\begin{bmatrix}
				0 & B  \\
				0 & D 
			\end{bmatrix},	
		\end{align*}
by Remark~\ref{rem:Azero} one can apply Corollary~\ref{cor:tridiagonal}. The statement follows since $\A$ is a bounded perturbation of $\A'$. 
\end{proof}

\subsection{Simplified Ericksen--Leslie model for nematic liquid crystals}\label{subsec:simplified_LCD}
\label{ss:LCD}
The continuum theory of liquid crystals was developed by Ericksen~\cite{Ericksen} and Leslie~\cite{Leslie}. A simplified model has been introduced by Lin and Liu~\cite{Lin_Liu}, %see also \cite{Li, Li_Wang},
%The Ericksen--Leslie model is a continuum theory of liquid crystals. 
and here we consider the following simplified model normalizing all constants to one
%\todo{Antonio: Can we take $\nu,\gamma,\lambda=1$? Amru: Yes}
\begin{align}\label{eq:LCD}
\left\{
\begin{aligned}
\partial_{t} u + (u\cdot \nabla) u - \Delta u + \nabla \pi &= -  \text{div}([\nabla d]^{\top} \nabla d) \quad &\text{in} \ \R_+\times \Dom,  \\
\partial_{t} d+(u \cdot \nabla) d &= \Delta d+ \vert \nabla d \vert^{2} d \quad &\text{in} \ \R_+\times \Dom, \\
\div\, u &= 0 \quad &\text{in} \ \R_+\times \Dom,\\
%\abs{d}&=1, \quad &\text{ in } \ \R_+\times \Omega, \\
u&=0 \quad\hbox{and}\quad \partial_n d =0   &\text{in} \  \R_+\times \partial\Dom,
\end{aligned}\right.
\end{align}
with initial data $u(0)=u_0$ and $d(0)=d_0$. 
Here $u: \R_+ \times \Dom \rightarrow \mathbb{R}^{3} $ denotes the velocity field of the fluid, $\pi: \R_+ \times \Dom \rightarrow \mathbb{R} $ the pressure, and
$d: \R_+\times \Dom \rightarrow \mathbb{R}^{3} $ denotes the molecular orientation of the liquid crystal at the macroscopic level referred to as the director field. This physical interpretation of $d$ imposes the condition $\vert d \vert=1$ in $\R_+\times \Dom$. Recent developments and the literature on this subject are discussed by Hieber and Pr\"uss in the survey~\cite{HieberPruess2016}.
%\begin{equation}
%\vert d \vert=1 \quad \text{in} \quad \R_+\times \Dom.
%\label{abs}
%\end{equation} 

The simplified Ericksen-Leslie model \eqref{eq:LCD} has been investigated by Hieber, Nesensohn, Pr\"uss, and Schade in \cite{Hieberetal_2016}  as a quasi-linear evolution equation in maximal $L^q_t$-$L^p_x$-regularity spaces 
for a smooth bounded domain $\Dom\subseteq \R^3$, see also e.g. \cite{Choudhury2018} for a semilinear approach  where the term $\text{div}([\nabla d]^{\top} \nabla d)$ is estimated as non-linear right hand side in a negative Sobolev space. To define the relevant operator introduced in \cite{Hieberetal_2016} to linearize \eqref{eq:LCD} we need  some preparations and we assume that $\Dom\subseteq \R^3$ is a bounded $C^2$-domain. Let $p\in (1,\infty)$, and set 
$$L_{\sigma}^p(\Dom)\stackrel{{\rm def}}{=}
\{u\in L^p(\Dom)^3\colon\div\, u=0\text{ in }\mathcal{D}'(\Dom) \hbox{ and } n\cdot u\vert_{\partial\Dom}=0\},$$
where $n$ denotes the exterior normal vector field on $\partial\Dom$. Recall also that $n\cdot u\in \mathcal{D}'(\partial\Dom)$ as $\div\,u=0$, see e.g.\ \cite[Theorem III.2.2]{Galdi11}.
Then we denote by $\P_p : L^p(\Dom)^3 \to L_{\sigma}^p(\Dom)$ the Helmholtz projection,   
and by the $\P_p\Delta_p$ the Stokes operator, i.e.
\begin{align*}
\Do(\P_p\Delta_p)
&\stackrel{{\rm def}}{=}\{u\in H^{2,p}(\Dom)^3\cap L^p_{\sigma}(\Dom)^3\colon u=0\text{ on }\partial\Dom\},\\
\P_p\Delta_p
&\colon  \Do(\P_p\Delta_p) \subseteq 
L_{\sigma}^p(\Dom) \to L_{\sigma}^p(\Dom),\quad  u\mapsto \P_p (\Delta u),
\end{align*}
see e.g. \cite[Section 7.3]{pruss2016moving}. Also, we define the Neumann Laplacian $\Delta_{N,p}$ on $L^p(\Dom)^3$ as the operator $u\mapsto \Delta u$ with domain
\begin{align*}
\Do(\Delta_{N,p})
&\stackrel{{\rm def}}{=}\{u\in H^{2,p}(\Dom)^3\colon \partial_n u |_{\partial\Dom}=0\},
\end{align*} 
compare e.g. \cite[Section 7.4]{pruss2016moving}.	
%	\todo{Antonio: Add reference on Stokes operator and Neumann Laplacian. Amru: e.g. Prüss Simonett?}
Noticing that $ [\text{div}([\nabla d]^{\top} \nabla d) ]_i =  (\partial_i d_{\ell} )\Delta d_{\ell} +(\partial_k d_{\ell})( \partial_{i,k}^2 d_{\ell})$ (here, the summation over repeated indexes is employed), the linearization of \eqref{eq:LCD} for fixed $d$ %$d\in C^{1,\alpha}(\Dom;\R^3)$ 
is given by
\begin{equation}
\label{eq:LCD_operator}
\A_p^{\EL}(d)
=\begin{bmatrix}
-\P_p\Delta_p & \P_p\mathscr{B}(d) \\ 
 0 &-  \Delta_{N,p} 
\end{bmatrix}, \quad [\mathscr{B}(d) u ]_i= (\partial_i d_{\ell} )\Delta u_{\ell} +(\partial_k d_{\ell})( \partial_{i,k}^2 u_{\ell}),
\end{equation}
for $i\in \{1,2,3\}$ on
\begin{align*}
X=L^p_{\sigma}(\Dom)\times L^p(\Dom)^3 \quad \hbox{with domain} \quad \Do(\A_p^{\EL}(d))=\Do(\P_p\Delta_p) \times \Do(\Delta_{N,p}).
\end{align*}
%Since $\Delta_{N,p}$ is not injective, we employ the results from Subsection \ref{ss:not_injective} to formulate the main result of this subsection.
 %$X=L^p_{\sigma}(\Dom)\times L^p(\Dom;\R^3)$, with domain $\Do(\A_p(d))=\Do(\P_p\Delta_p) \times \Do(\Delta_{N,p})$.
%and 
%$$
%[\mathscr{B}(d) u ]_i= (\partial_i d_{\ell} )\Delta u_{\ell} +(\partial_k d_{\ell})( \partial_{i,k}^2 u_{\ell}) \ \ \text{ for all }i\in \{1,2,3\}.
%$$
%The main result of this section reads as follows.
%{eq:A_injective}
The main result of this subsection reads as follows.

\begin{proposition}
\label{prop:LCD}
Let $\Dom\subseteq\R^3$ be a bounded $C^{2}$-domain. 
Let $p\in (1,\infty)$ and $\A_p^{\EL}(d)$  be as in \eqref{eq:LCD_operator}.
Then for all $\nu>0$ the following hold.
\begin{enumerate}[{\rm(1)}]
\item\label{it:LCD_R_sectorial} If $d\in W^{1,\infty}(\overline{\Dom})^3$, then $\nu+ \A_p^{\EL}(d)$ is $\Rsec$-sectorial of angle $0$.
\item\label{it:LCD_H_calculus} If $d\in C^{1,\alpha}(\Dom)^3$ for some $\alpha>0$, then $\nu+ \A_p^{\EL}(d)$  has a bounded $H^{\infty}$-calculus of angle $0$.
\end{enumerate}
The constants in \eqref{it:LCD_R_sectorial} (resp.\ \eqref{it:LCD_H_calculus}) depend on $d$ only through $\|d\|_{W^{1,\infty}}$ (resp.\
$\|d\|_{C^{1,\alpha}}$).
\end{proposition}
\begin{remark}
For the deterministic setting in \cite{Hieberetal_2016} it has been sufficient  to prove maximal $L^q_t$-$L^p_x$-regularity for $\A_p^{\EL}(d)$
to solve the non-linear problem. %, and it has been proven in \cite{Hieberetal_2016} by writing $-\P_p\mathscr{B}(d)$ as a right hand side.
Proposition \ref{prop:LCD}\eqref{it:LCD_H_calculus} also implies stochastic maximal regularity, and therefore the quasilinear approach to \eqref{eq:LCD} developed in \cite{Hieberetal_2016} -- with the regularity assumptions as in \cite[Remark 4.2]{Hieberetal_2016} -- can be extended to the stochastic setting using the results by Veraar and the first author in \cite{AV19_QSEE_1,AV19_QSEE_2} to solve non-linear SPDEs.
% For the sake of brevity, we do not include any detail here. 
\end{remark}

%As the proof below shows, Proposition \ref{prop:LCD}\eqref{it:LCD_R_sectorial} holds for $d\in C^{1}(\overline{\Dom};\R^3)$.

\begin{proof}[Proof of Proposition \ref{prop:LCD}]
To prove Proposition \ref{prop:LCD}\eqref{it:LCD_R_sectorial}  we apply Corollary \ref{cor:tridiagonal}\eqref{it:tridiagonal_R_sectorial}. Note that, for all $u\in\Do(\Delta_{N,p})$,
\begin{align*}
\|\P_p\mathscr{B}(d) u\|_{L^p_{\sigma}(\Dom)}
\lesssim \|\mathscr{B}(d) u\|_{L^p(\Dom)^3}
&\lesssim \|d\|_{W^{1,\infty}(\Dom)^3} \Big(\sup_{i,j\in \{1,2,3\}} \| \partial_{i,j}^2 u\|_{L^p(\Dom)^3}\Big)\\
&\lesssim\|d\|_{W^{1,\infty} (\Dom)^3}  \|(\nu- \Delta_{N,p}) u\|_{L^p(\Dom)^3},
\end{align*}
where in the last inequality we have used that $\nu-\Delta_{N,p}$ is invertible for all $\nu>0$. Thus, $\nu+\A^{\EL}(d)$ is $\Rsec$-sectorial of angle $0$ by Corollary \ref{cor:tridiagonal}\eqref{it:tridiagonal_R_sectorial}.

To prove Proposition \ref{prop:LCD}\eqref{it:LCD_H_calculus}, we apply Corollary \ref{cor:tridiagonal}\eqref{it:tridiagonal_H_calculus} in the $(-)$-case. %As above, by Proposition \ref{prop:from_injective_to_non_injective}, it is enough to look at $\A_p^I(d)$ (see \eqref{eq:A_I_LCD}).
	Recall that $\P_p\Delta_p$ is invertible and therefore $\Ran((-\P_p\Delta_p)^{\gamma})=L^p_{\sigma}(\Dom)$ for all $\gamma\in (0,1)$. Moreover, since $\Dom$ is a $C^2$-domain, $(\P_p \Delta_p)^*=\P_{p'}\Delta_{p'}$, where $\frac{1}{p}+\frac{1}{p'}=1$, and
	\begin{equation}
	\begin{aligned}
	\label{eq:stokes_negative_domains}
	\Dd((-\P_p\Delta_p)^{-\gamma})
	&=\big(\Dd((-\P_{p'}\Delta_{p'})^{\gamma}) \big)^*\\
	&=
	\big(H^{2\gamma,p'}(\Dom)^3\cap L^{p'}_{\sigma}(\Dom)\big)^* \quad \text{ for all }\gamma\in  \Big(0,\frac{1}{2p}\Big).
	\end{aligned}
	\end{equation}
	By \cite[Proposition 9.14]{KKW} and interpolation, $\P_p$ extends uniquely to a map
	\begin{equation}
	\label{eq:Helmholtz_projection_negative_spaces}
	\P_p : H^{-2\gamma,p}(\Dom)^3\to \Dd((-\P_p\Delta_p)^{-\gamma})\ \
	\text{ for all }\gamma\in \Big(0,\frac{1}{2p}\Big)
	\end{equation}
	where we used that $\Dd((-\Delta_{D,p})^{-\gamma})=H^{-2\gamma,p}(\Dom)^3$ for all $\gamma\in (0,\frac{1}{2p})$ by \cite{Se}. Here, as in \cite[Proposition 9.14]{KKW}, $\Delta_{D,p}$ denotes the Dirichlet Laplacian.
	
	Since $\Do(\nu-\Delta_{N,p})\embed H^{2,p}(\Dom)^3$ for all $\nu>0$, 
	\begin{equation}
	\label{eq:LCD_domain_D_N}
	\Do((\nu-\Delta_{N,p})^{\gamma})\embed H^{2\gamma,p}(\Dom)^3\ \  \text{ for all } \gamma\in (0,1).
	\end{equation}
	As above, without loss of generality we assume $\alpha\in (0,\frac{1}{p})$. With this preparation, for all $\nu>0$ and $u \in \Do((\nu-\Delta_{N,p})^{1-\beta})$ with $\beta<\frac{\alpha}{2}$, we can estimate %as follows
	\begin{equation}
	\label{eq:LCD_estimate_H_calculus_negative}
	\begin{aligned}
	&\|(\nu-\P_{p}\Delta_p)^{-\beta} [\P_p \mathscr{B}(d) u] \|_{L^p_{\sigma}(\Dom)}\\
	&\qquad\qquad\ \eqsim \|(-\P_{p}\Delta_p)^{-\beta} [\P_p \mathscr{B}(d) u] \|_{L^p_{\sigma}(\Dom)}\\
	& \qquad\qquad \ \eqsim \|\P_p \mathscr{B}(d) u\|_{\Dd((-\P_p\Delta_p)^{-\beta})}\\
	&\qquad\qquad\stackrel{\eqref{eq:Helmholtz_projection_negative_spaces}}{\lesssim} 
	\|\mathscr{B}(d) u\|_{H^{-2\beta,p}(\Dom)^3}\\
	& \qquad\qquad \stackrel{\eqref{eq:LCD_domain_D_N}}{\lesssim } 
	\|d\|_{C^{1,\alpha}(\Dom)^3} \|(\nu-\Delta_{N,p})^{1-\beta} u \|_{L^p(\Dom)^3},
	\end{aligned}
	\end{equation}
	where we also used that $C^{\alpha}$ maps are pointwise multipliers on $H^{-2\beta,p}$.
	Now, 
	Corollary~\ref{cor:tridiagonal}\eqref{it:tridiagonal_H_calculus} in the $(-)$-case ensures that $\nu+\A_p^{\EL}(d)$ has a bounded $H^{\infty}$-calculus for all $\nu>0$ with a corresponding estimate in terms of $\|d\|_{C^{1,\alpha}(\Dom)^3}$.
\end{proof}

\begin{remark}\label{r:LCD_x_dependent}
To prove Proposition \ref{prop:LCD}\eqref{it:LCD_H_calculus}
one can also employ Corollary \ref{cor:tridiagonal} in the $(+)$-case. However, 
one has then to assume that $\Dom$ is a $C^{2,\alpha}$-domain for some $\alpha>0$.
Then, by elliptic regularity, one can check that 
\begin{equation}
\label{eq:LCD_helmholtz_projection_positive}
\P_p:H^{s,p}(\Dom)^3 \to H^{s,p}(\Dom)^3 \cap L^p_{\sigma}(\Dom) \ \ \text{ for all }s\in \big[0,\tfrac{1}{p}\big).
\end{equation}
Without loss of generality, we assume $\alpha\in (0,\frac{1}{p})$.  %Let $\Delta_{N,p}^{I}\stackrel{{\rm def}}{=}\Delta_{N,p}|_{\overline{\Ran(\Delta_{N,p})}}$. 
By Proposition \ref{prop:fractional_powers_positive}, 
for all $\beta\in (0,\alpha/2)$ and $\nu>0$,
\begin{align}
\label{eq:LCD_description_positive_domain}
\Do((\nu-\P_p \Delta_p)^{\beta})=\Do((-\P_p \Delta_p)^{\beta})
&=H^{2\beta,p}(\Dom)^3 \cap L^p_{\sigma}(\Dom),\\
\label{eq:LCD_description_positive_domain_2}
\Do((\nu-\Delta_{N,p})^{\beta})
&= H^{2\beta,p}(\Dom)^3.
\end{align} 
The former, the fact that $\Dom\in C^{2,\alpha}$ and elliptic regularity yield 
for all $\beta\in (0,\frac{\alpha}{2})$,
\begin{equation}
\label{eq:LCD_description_positive_domain_1_beta}
\Do((\nu-\Delta_{N,p})^{1+\beta})= \Big\{u\in H^{2+2\beta,p}(\Dom)^3\colon\partial_{n}u|_{\partial\Dom}=0\Big\}.
\end{equation}
Thus, for all $\beta\in (0,\frac{\alpha}{2})$ and $u\in \Do((\nu-\Delta_{N,p})^{1+\beta})$,
\begin{equation}
\begin{aligned}
\label{eq:LCD_estimate_H_calculus_positive}
\|(\nu-\P_p\Delta_p)^{\beta} [\P_p \mathscr{B}(d) u]\|_{L^p_{\sigma}(\Dom)}
&\stackrel{\eqref{eq:LCD_description_positive_domain}}{\eqsim} \| \P_p \mathscr{B}(d) u\|_{H^{2\beta,p}(\Dom)^3\cap L^p_{\sigma}(\Dom)}\\
&\stackrel{\eqref{eq:LCD_helmholtz_projection_positive}}{\lesssim} \| \mathscr{B}(d) u\|_{H^{2\beta,p}(\Dom)^3}	\\
& \ \stackrel{(i)}{\lesssim}\ \  \|d\|_{C^{1,\alpha}(\Dom)^3} \|  u\|_{H^{2+2\beta,p}(\Dom)^3}\\
&\stackrel{\eqref{eq:LCD_description_positive_domain_1_beta}}{\eqsim} 
\|d\|_{C^{1,\alpha}(\Dom)^3} \|(\nu-\Delta_{N,p})^{1+\beta} u\|_{L^p(\Dom)^3}
\end{aligned}
\end{equation}
where in $(i)$ we use as before that $C^{\alpha}$ functions are pointwise multipliers on $H^{2\beta,p}$ ones.
The estimate
\eqref{eq:LCD_estimate_H_calculus_positive} ensures that Corollary \ref{cor:tridiagonal} in the $(+)$-case can be applied.
		
The argument in the above proof can be extended to prove the boundedness of the $H^{\infty}$-calculus for $\mu+\A_p^{\EL}(d)$  for some $\mu\geq 0$ in case $-\P_p\Delta_p$ is replaced by $-\P_p\mathscr{A}_p$ where $\mathscr{A}_p(x)u=\sum_{i,j=1}^3 a^{i,j}(x) \partial_{i,j}^2 u$, $a^{i,j}\in C^{\alpha}(\Dom)^{3\times 3}$ and $a^{i,j}$ are uniformly elliptic. To see this, note that in the above proof the Stokes operator plays a role only through %the first identity in 
\eqref{eq:LCD_description_positive_domain}. The latter also holds if $-\P_p\Delta_p$ is replaced by $\mu-\P_p\mathscr{A}_p$ provided the latter operator has a bounded $H^{\infty}$-calculus for $\mu$ large enough. To prove the latter, recall that $\mu-\P_p\mathscr{A}_p$ is $\Rsec$-sectorial on $L^p_{\sigma}(\Dom)$ of angle $<\pi/2$ for $\mu$ large enough by \cite[Chapter 7]{pruss2016moving}. Thus it also has a bounded  $H^{\infty}$-calculus by Theorem \ref{t:transference_appendix} applied with $A=-\P_p\Delta_p$ and $B=\mu-\P_p\mathscr{A}_p$ (cf.\ \eqref{eq:LCD_description_positive_domain}, \eqref{eq:LCD_description_positive_domain_2} and \eqref{eq:stokes_negative_domains}).
\end{remark}

\subsection{The weak Keller-Segel operator}\label{subsec:Keller_Segel}
Keller-Segel equations arise in the mathematical modeling of chemotaxis, see e.g. \cite{Winkler2020, HorstmannI, HorstmannII, Hillen2009} for surveys and further literature.    %\todo{Antonio: Put some references here}
Here, we consider the classical Keller-Segel system given by
\begin{align}\label{eq:Keller_Segel}
\left\{
\begin{aligned}
\partial_t u -\Delta u +\nabla \cdot(u \nabla v)&=0,\quad &\text{in} \ \R_+\times \Dom,\\
\partial_t v-\Delta v +v -u&=0,\quad &\text{in} \ \R_+\times \Dom,\\
u(0)=u_0,\ \ \ \   v(0)&=v_0,\quad &\text{in} \ \R_+\times \Dom.
\end{aligned}\right.
\end{align}
where $u\colon \R_+\times  \Dom \rightarrow \R$ represents the density of a cell population and $v\colon \R_+\times \Dom \rightarrow \R$ the concentration of a chemoattractant.
We complement the above system with non-linear boundary conditions%\todo{Antonio: comments on different boundary conditions. This is a nonlinear boundary condition. Amru: REF..}
\begin{equation}
\label{eq:boundary_condition_KS}
\partial_{n} u - u\, \partial_{n } v=0,\ \  \text{ and }\ \  v=0,\quad \text{in}\ \R_+\times\partial\Dom, 
\end{equation}
where $\partial_{n}$ denotes the outer normal derivative at the boundary. 

In applications to (stochastic) partial differential equations, the weak setting has two advantages. Firstly, it (typically) requires less regularity assumption on $\partial\Dom$ compared to the strong setting, and secondly, in the stochastic framework it also requires minimal compatibility conditions for the noise.
To obtain the weak formulation of \eqref{eq:Keller_Segel}-\eqref{eq:boundary_condition_KS}, we multiple the first equation in \eqref{eq:Keller_Segel} by $\varphi\in C^{\infty}(\overline{\Dom})$. Using that
$$
\int_{\Dom} [-\Delta u +\nabla \cdot(u\cdot \nabla v) ]\varphi \, \dd x \stackrel{\eqref{eq:boundary_condition_KS}}{=}
\int_{\Dom} [\nabla u\cdot\nabla \varphi - u \nabla v\cdot \nabla \varphi] \, \dd x, 
$$
one can linearize  \eqref{eq:Keller_Segel}-\eqref{eq:boundary_condition_KS} in the weak setting by writing for $U=(u,v)^T$  
\begin{align*}
U'+\A^{\KS}_p(u) U=0, \quad U(0)=U_0, \ \  \hbox{ where }\ \   \A_p^{\KS}(z)
\stackrel{{\rm def}}{=}
\begin{bmatrix}
-\Delta_{N,p}^{\weak} & \mathscr{B}_p(z) \\
-\id & \id-\Delta_{D,p}^{\weak} 
\end{bmatrix}
\end{align*}
on $X_0=  (W^{1,p'}(\Dom))^* \times (W^{1,p'}_0(\Dom))^*$ with $\Do(\A_p^{\KS}(z))
\stackrel{{\rm def}}{=}W^{1,p}(\Dom)\times W^{1,p}_0(\Dom)$.
Here, for all $(u,v),(u',v')\in W^{1,p}(\Dom)\times W^{1,p}_0(\Dom)$,
\begin{align*}
\l u', \Delta^{\weak}_{N,p} u\r 
&\stackrel{{\rm def}}{=} -\int_{\Dom}\nabla u \cdot \nabla u' \,\dd x,\quad \quad
\l v', \Delta^{\weak}_{D,p} v\r 
\stackrel{{\rm def}}{=} -\int_{\Dom}\nabla v \cdot \nabla v'\,\dd x,\\
\l u', \mathscr{B}_p(z) v\r 
&\stackrel{{\rm def}}{=} - \int_{\Dom}  z\, \nabla v\cdot \nabla u' \,\dd x \ \  \hbox{where } z\in L^{\infty}(\Dom).
\end{align*}

\begin{proposition}
For all $p\in (1,\infty)$ the following hold.
\begin{enumerate}[{\rm(1)}]
\item\label{it:A_KS_R_sectoriality}
If  $\Dom\subseteq \R^d$ is a bounded $C^1$-domain and $z\in L^{\infty}(\Dom)$, then there exists $\nu\geq 0$ such that  
 $\nu+\A_p^{\KS}(z)$ is $\Rsec$-sectorial of angle $0$.
\item\label{it:A_KS_H_calculus}
If $\Dom$ is a $C^{1,\alpha}$-domain and $z\in C^{\alpha}(\Dom)$ for some $\alpha>0$, then there exists $\nu\geq 0$ such that
 $\nu+\A_p^{\KS}(z)$ has a bounded $H^\infty$-calculus of angle $0$. % provided $z\in C^{\alpha}(\Dom)$ and $\Dom$ is a $C^{1,\alpha}$-domain for some $\alpha>0$.
\end{enumerate}
The constants in \eqref{it:A_KS_R_sectoriality} (resp.\ \eqref{it:A_KS_H_calculus}) depend on $z$ only through $\|z\|_{L^{\infty}(\Dom)}$ (resp.\ $\|z\|_{C^{\alpha}(\Dom)}$).
\end{proposition}

\begin{proof}
By a standard  result for bounded perturbations (see e.g.\ \cite[Corollary 3.3.15 and Proposition 4.4.3]{pruss2016moving}), it suffices to show \eqref{it:A_KS_R_sectoriality}-\eqref{it:A_KS_H_calculus} for  $\A_p^{\KS}(z)$ replaced by %$\wh{\A}_p(z)$ where
$$
\wh{\A}_p(z)
\stackrel{{\rm def}}{=}
\begin{bmatrix}
\id -\Delta_{N,p}^{\weak} & \mathscr{B}_p(z) \\
0 & 2-\Delta_{D,p}^{\weak}
\end{bmatrix}
\ \ \text{ with } \ \  
\Do(\wh{\A}_p(z))=\Do(\A_p^{\KS}(z)).
$$
Let us recall that by \cite[Theorem 11.5]{DivH} $\id-\Delta_{N,p}^{\weak}$ and $2-\Delta_{D,p}^{\weak}$ have  a bounded $H^{\infty}$-calculus of angle $0$. In particular, they are also $\Rsec$-sectorial of angle $0$.
So, \eqref{it:A_KS_R_sectoriality} follows immediately from Corollary \ref{cor:tridiagonal}\eqref{it:tridiagonal_R_sectorial}.

The claim \eqref{it:A_KS_H_calculus} follows from Theorem \ref{t:extrapolation_H} for $p\neq 2$ if  Corollary \ref{cor:tridiagonal}\eqref{it:tridiagonal_H_calculus} applies to the case $p=2$. Theorem \ref{t:extrapolation_H} is then applicable. Its assumptions are indeed satisfied: First, as $\wh{\A}_p(z)$ for $p\in (1,\infty)$ is a consistent family of operators, \eqref{eq:estimate_A_A_0_theta_alpha_beta} holds since $\wh{\diag}_p=\hbox{diag}\{\id -\Delta_{N,p}^{\weak}, 2-\Delta_{D,p}^{\weak}\}$ and $\wh{\A}_p(z)$ are boundedly invertible. Second, by the previous argument $\wh{\A}_p(z)$ is $\Rsec$-sectorial for $p\in (1,\infty)$. Third, Corollary \ref{cor:tridiagonal}\eqref{it:tridiagonal_H_calculus} and $\omega(\wh{\A}_2(z))=0$ imply by  e.g.\ \cite[Corollary 10.4.10 and Theorem 10.4.22]{Analysis2} that  $\wh{\A}_2(z)$ generates a contraction semigroup w.r.t.\ an equivalent Hilbertian norm. 

To apply Corollary \ref{cor:tridiagonal}\eqref{it:tridiagonal_H_calculus} for $p=2$, we check the condition for the $(+)$-case of \eqref{it:tridiagonal_H_calculus}, that is, 
\begin{equation}
\label{eq:KS_map_B_domains}
\mathscr{B}_2(z)\colon\Do((2-\Delta^{\weak}_{D,2})^{1+\beta})\to \Do((\id-\Delta^{\weak}_{N,2})^{\beta}) \text{ is bounded for some }\beta>0.
\end{equation}

Since $\id-\Delta_{N,2}^{\weak}$ has a bounded $H^{\infty}$-calculus and by \cite[Theorem 4.6.1 and Corollary 4.5.2]{BeLo}, we have for all $\gamma\in (0,1/2)$ that
\begin{align*}
\Do((1-\Delta^{\weak}_{N,2})^{\gamma})
&=[(H^{1}(\Dom))^*,H^{1}(\Dom)]_{\gamma}\\
&=[(H^{1}(\Dom))^*,L^2(\Dom)]_{2\gamma}\\
&=\big([H^{1}(\Dom),L^2(\Dom)]_{2\gamma}\big)^*=\big(H^{1-2\gamma}(\Dom)\big)^*.
\end{align*}
Since $2-\Delta_{D,2}^{\weak}$ has a bounded $H^{\infty}$-calculus as well,
$
\Do((2-\Delta^{\weak}_{D,2})^{\gamma})=H^{-1+2\gamma}(\Dom)$ for all $\gamma\in (0,1/2)$. By elliptic regularity and the fact that $\Dom$ is a $C^{1,\alpha}$-domain  we have, for all $2\beta\in \big(0,\alpha\wedge\tfrac{1}{2}\big)$,
\begin{align*}
\Do((2-\Delta^{\weak}_{D,2})^{1+\beta})
=\big\{v\in H^{1}_0(\Dom)\,:\,\Delta^{\weak}_{D} v\in H^{-1+2\beta}(\Dom) \big\}=  H_0^{1+2\beta}(\Dom).
\end{align*}
With this at hand, we prove \eqref{eq:KS_map_B_domains}. Fix $0<\beta<\eta<\frac{\alpha}{2}\wedge\frac{1}{4} $. %By \eqref{eq:def_H_gamma}, 
For any $u\in H^{1-2\beta}(\Dom)$ there exists an extension $U\in H^{1-2\beta}(\R^d)$ such that $U|_{\Dom}=u$ and $\|U\|_{H^{1-2\beta}(\R^d)}\lesssim \|u\|_{H^{1-2\beta}(\Dom)}$ (with implicit constant independent of $u$). Similarly, choose $Z\in C^{\eta}(\R^d)$ such that $Z|_{\Dom}=z$ and $\|Z\|_{C^{\eta}(\R^d)}\lesssim \|z\|_{C^{\eta}(\Dom)}$. Let ${\rm E}_0$ be the extension by 0 outside $\Dom$. Then, for all $v \in C^{\infty}_0(\Dom)$ and $v$ as above, 
\begin{align*}
\l u, \mathscr{B}_2(z) v\r 
=\Big|\int_{\Dom} z \,\nabla u \cdot \nabla v \,\dd x \Big|
&=\Big|\int_{\R^d} Z\, \nabla U \cdot \nabla ({\rm E}_0 v) \,\dd x \Big|\\
&\lesssim \|Z\, \nabla ({\rm E}_0 v)\|_{H^{2\beta}(\R^d)} \|U\|_{H^{1-2\beta}(\R^d)}\\
&\lesssim \|z\|_{C^{\eta}(\Dom)}\| v\|_{H^{1+2\beta}(\Dom)} \| u\|_{H^{1-2\beta}(\Dom)}.
\end{align*}
By density of $C^{\infty}_0(\Dom)$ in $H^{1+\beta}_0(\Dom)$ and taking the supremum over all $\|u\|_{H^{1-2\beta}(\Dom)}\leq 1$, \eqref{eq:KS_map_B_domains} follows.%\todo{Antonio: This argument might be more involved if $u$ does not respect a Dirichlet boundary condition. Should we comment on different boundary conditions? Amru: See suggestion}
\end{proof}

\begin{remark}
\label{r:KS_different_boundary_conditions}
%(to be completed).
The boundary conditions considered here have been discussed recently in \cite{Winkler2021}.
A variety of zero-flux boundary conditions boundary is discussed in \cite[Section 2]{Hillen2009}, and there are also Keller-Segel models with 
pure Dirichlet or Neumann boundary conditions and combinations of the various cases. It seems that the above proof can be adapted to these situations by using different extension operators.
\end{remark}

\subsection{Artificial compressible Stokes system}\label{subsec:art_Stokes}
The artificial compressible Stokes system has been 
introduced in the context of steady state solutions to the Navier-Stokes equations, see \cite{Cho1967, Cho1968, Tem1969_1,Tem1969_2, Tem_book}.  It is formally given by
\begin{align}\label{eq:artifical_Stokes}
\A^{\AS}= \begin{bmatrix}
0 & \tfrac{1}{\epsilon^2} \div \\ \nabla & - \Delta + v_s\cdot \nabla + (\nabla v_s)^T
\end{bmatrix}, \quad \varepsilon>0,
\end{align}
with a given real valued vector field $v_s$. The spectral properties of this operator have been investigated recently in detail for the Hilbert space case in \cite{KaTe2018, Ter2018}. 
However, the $L^p$-theory for this operator has not been studied so far. Therefore, we consider  here $\A_p^{\AS}=\A^{\AS}$ in the space $X=X_1\times X_2$ with%\todo{Make it consistent with function spaces.}
\begin{align*}
X_1=H^{1,p}(\Dom) \quad \hbox{and} \quad X_2=L^p(\Dom)^3
\end{align*}
for $p\in (1,\infty)$ with domain
\begin{align*}
\Do(\A^{\AS}_p)=H^{1,p}(\Dom)\times \big( H^{2,p}(\Dom)^3\cap H_{0}^{1,p}(\Dom)^3\big).
\end{align*}
It turns out that  Corollary~\ref{prop:Azero} is applicable here, and it guarantees some basic operator theoretical properties by purely perturbative methods and properties of the Laplacian in $L^p$-spaces. Using the particular structure of the off-diagonal perturbation more detailed properties can be derived as in \cite{KaTe2018, Ter2018} for $p=2$.

\begin{proposition}\label{prop:artifical_Stokes}
	Let $\Dom$ be a bounded $C^2$ domain in $\R^3$. If $v_s\in H^{1,q}(\Dom)^3$ for $p,q\in (1,\infty)$ with $q>3/2$ and $q\geq p$, then  for each $\psi>0$ there exists $\mu \geq 0$ such that the shifted artificial Stokes system $\mu+\A_p^{\AS}$ has a bounded $H^\infty$-calculus of 
	angle $\leq \psi$.
\end{proposition}

\begin{remark}
	For the case $v_s\equiv 0$ one can consider the operator
	\begin{align*}
	M_2(\lambda)= \id - \tfrac{\varepsilon^2}{\lambda} \nabla\div (\lambda-\Delta)^{-1}, \quad \lambda\in \C \setminus (\sigma(-\Delta) \cup\{0\}). 
	\end{align*}
	This is related to the second Schur complement of $\A_p^{\AS}$
	\begin{align*}
	S_2(\lambda)= (\lambda-\Delta) - \tfrac{\varepsilon^2}{\lambda} \nabla\div, \quad \lambda\in \C \setminus (\sigma(-\Delta) \cup\{0\}), \quad \Do(S_2(\lambda))=\Do(D) 
	\end{align*}
	which is in fact -- up to a shift -- a Lam\'{e} operator studied for instance in \cite{Mit_2010, Dan_2010, Kun_2015}. Following the proof of \cite[Theorem 4.1]{Mit_2010}, $S_2(\lambda)$ is boundedly invertible on $L^p(\Dom)$ for $\lambda>0$. Hence $M_2(\lambda)= S_2(\lambda)(\lambda-D)^{-1}$ is a closed bijective operator and hence boundedly invertible. In particular, one can choose in this situation any $\mu>0$. The case with $v_s\equiv 0$ is comparable to the situation analyzed in \cite{Kost_2019}.
\end{remark}

\begin{proof}[Proof of Proposition~\ref{prop:artifical_Stokes}]
For $v_s\in H^{1,q}(\Dom)^3$ with $q>3/2$ and $q\geq p$ one has
\begin{align*}
\norm{v_s\cdot \nabla v}_{L^p} \leq \norm{v_s}_{L^{sp}} \norm{v}_{H^{1, rp}} 
\lesssim \norm{v_s}_{H^{1,q}} \norm{v}_{H^{2-\delta, p}}
\end{align*}
using Hölder's inequality and Sobolev embeddings with
\begin{align*}
s= \tfrac{3q}{p(3-q)}, \quad\tfrac{1}{s} + \tfrac{1}{r}=1, \quad \delta = \tfrac{3}{q},
\end{align*}
and
\begin{align*}
\norm{(\nabla v_s)^T v}_{L^p} \leq \norm{v_s}_{H^{1,sp}} \norm{v}_{L^{rp}} 
\lesssim \norm{v_s}_{H^{1,q}} \norm{v}_{H^{2-\delta, p}},
\end{align*}
here using Hölder's inequality and Sobolev embeddings with
\begin{align*}
s= \tfrac{p}{q}, \quad\tfrac{1}{s} + \tfrac{1}{r}=1, \quad \delta = 2- \tfrac{3}{q}.
\end{align*}
Since  the operator $D_0=- \Delta$ with $\Do(D_0)=H^{2,p}(\Dom)^3\cap H_{0}^{1,p}(\Dom)^3$ has a bounded $H^\infty$-calculus of angle zero, and as shown by the estimate above $v_s\cdot \nabla + (\nabla v_s)^T$ is a lower order perturbation, there exists a $\mu_0\geq 0$ such that $\mu_0+\diag$ has a bounded $H^\infty$-calculus of angle zero. 
Therefore, the operator $\A_p^{\AS}$ 
is diagonally dominant as
\begin{align*}
\norm{\nabla p}_{L^p(\Dom)^3} \leq \norm{p}_{H^{1,p}(\Dom)^3} \quad \hbox{and} \quad \norm{\tfrac{1}{\epsilon^2} \div\, v}_{H^{1,p}(\Dom)} \leq \tfrac{1}{\epsilon^2} \norm{v}_{H^{2,p}(\Dom)^3},
\end{align*}
and by 
Corollary~\ref{prop:Azero}
it follows that there is a $\mu \geq \mu_0$ such that $\mu+\A$ has a bounded $H^\infty$-calculus of angle zero. 
\end{proof}

\subsection{Second order Cauchy problems with strong damping}\label{subsec:2ndOrder}
Second order Cauchy problems of the form
\begin{align}\label{eq:second_order_CP_TS}
\partial_t^2 u + T \partial_t u + Su =f, \quad u(0)=u_0, \quad \partial_t u(0)=v_0
\end{align} 
for operators $S$ and $T$ in $Y$ can be re-written  as a first order Cauchy problem by formally setting $v=\partial_t u$ to obtain
\begin{align}\label{eq:second_order_CP}
\A^{\SD}= \begin{bmatrix}
0 & -\id \\ S & T
\end{bmatrix}.
\end{align}
Classical examples are d'Alembert's wave equation, where $T=0$ and $S=-\Delta$, and the beam equation where $T=0$ and $S=\Delta^2$.
These operators are not diagonally dominant and therefore cannot be treated by the methods presented here, and in fact these equations are not parabolic. %Moreover, it is known, that the wave equation is not well-posed in $L^p$-spaces
%except for the case where $\Omega\subset \R$, REF?. 
In \cite{Schnau_2010} for $S$ with $\Do(S)\subseteq Y$ the case where $T= S^{1/2}$, and $X_1=\Do(S^{1/2})$ while $X_2=Y$ is considered. This is parabolic, but not diagonally dominant, and 
most of the damped wave and plate equations are rather lower dominant, that is, $\Do(\A)=\Do(C)\times \Do(D)$, compare \cite[Definition 2.2.1]{Tretter},
than diagonally dominant. 
 %However, adding a damping term $T$ such that $\Do(T)\subset \Do(S)$, the operator $\A$ in \eqref{eq:second_order_CP} becomes diagonally dominant and by Proposition~\ref{prop:Azero}, the mapping properties of $\A$ reduce to the respective properties of $D=-T$.
 Only when one adds a relatively strong damping such as 
a \emph{Kelvin-Voigt-type damping}, then one has that $T$ and $S$ are of the same order. The following consequence of Corollary~\ref{prop:Azero} captures such situations. In the literature there are many works on the regularity properties  of solutions to plate or wave equations, but there seems to be little known about the $H^\infty$-calculus of such operators.
\begin{corollary} \label{cor:second_Order_CP_damping}
	Let $T$ and $S$ be operators in a Banach space $Y$. Assume that $T$ is closed and densely defined and let 
$S$ be $T$-bounded,  i.e.\ there exist $c_0,L_0>0$ such that
$$
\|S y\|_{Y} \leq c_0 \|T y\|_{Y}+ L_0\|y\|_{Y} \quad \text{ for all }y\in \Do(T).
$$
Then the operator $\A^{\SD}$ with strong damping as in \eqref{eq:second_order_CP} on the space
	\begin{align*}
	X= \Do(T) \times Y \quad\hbox{with} \quad \Do(\A^{\SD})=\Do(T) \times \Do(T)
	\end{align*}
	is diagonally dominant. Moreover, if $T$ is sectorial, 
	then for all $\psi\in (\om(T),\pi)$ there exists $\mu\geq 0$ such that $\A^{\SD}+\mu$
	 is sectorial of angle $\leq \psi$.
	 The respective statement holds for $\Rsec$-sectoriality and for the boundedness of the $H^\infty$-calculus.
\end{corollary}

\begin{example}
\label{ex:KelvinVoigt}
	An example is the Kelvin-Voigt plate-like equation  
	\begin{align*}
	\partial_t^2 u  + \varepsilon\Delta^2 \partial_t  u + \Delta^2 u=f, \quad
	u(0)=u_0, \quad\partial_tu(0)=v_0, \quad \varepsilon>0,
	\end{align*}
	discussed in \cite[Section 5.5]{LaTri_book_I}. Setting $\partial_t  u=v$, it translates to
	\begin{align*}
	\A^{\SD}= \begin{bmatrix}
	0 & -\id \\ \Delta^2 & \varepsilon\Delta^2
	\end{bmatrix}
	\end{align*}
	which following Corollary~\ref{cor:second_Order_CP_damping} can be considered in
	\begin{align*}
	H^{4,p}(\R^d)\times L^p(\R^d) \quad \hbox{with} \quad \Do(\A^{\SD})=H^{4,p}(\R^d)\times H^{4,p}(\R^d).
	\end{align*}
 Moreover, by Corollary~\ref{cor:second_Order_CP_damping}, for each $\psi>0$ there exists $\mu\geq 0$ such that $\mu+\A^{\SD}$ has a bounded $H^{\infty}$-calculus of angle $\leq \psi$.
	%\todo{Is it angle 0 or arbitrary small?}
	
	Considering a smooth bounded domain, one can still apply Corollary~\ref{cor:second_Order_CP_damping} as long as the boundary conditions assure that the bi-Laplacian $\Delta^2$ is sectorial, $\Rsec$-sectorial, or has a bounded $H^\infty$-calculus, respectively.
\end{example}

\begin{example}	
	The strongly damped wave equation given in \cite[Section 3.8]{LaTri_book_I} is
	\begin{align*}
	\partial_t^2 u - \varepsilon\Delta \partial_t  u - \Delta u =f, \quad
	u(0)=u_0, \quad\partial_tu(0)=v_0, \quad \varepsilon>0.
	\end{align*}
	This yields to
	\begin{align*}
	\A^{\SD}= \begin{bmatrix}
	0 & -\id \\ -\Delta & -\varepsilon\Delta
	\end{bmatrix}
	\end{align*}
		which by Corollary~\ref{cor:second_Order_CP_damping} can be considered in
	\begin{align*}
	H^{2,p}(\R^d)\times L^p(\R^d) \quad \hbox{with} \quad \Do(\A^{\SD})=H^{2,p}(\R^d)\times H^{2,p}(\R^d).
	\end{align*}
    Corollary~\ref{cor:second_Order_CP_damping} also ensures that -- up to a shift -- $\A^{\SD}$ has a bounded $H^{\infty}$--calculus of arbitrary small angle.%\todo{Angle $0$ or arbitrary small?}
	
	As in Example \ref{ex:KelvinVoigt}, we may also consider smooth domains with suitable boundary conditions as long as the corresponding Laplacian $-\Delta$ is sectorial, $\Rsec$-sectorial or has bounded $H^{\infty}$--calculus with such boundary conditions, respectively.
\end{example}

%Damping via temperature:
\begin{example}
	Consider a damped thermoelastic plate equation of the type 
	\begin{align*}
	\begin{split}
	\partial_t^2u+\varepsilon \Delta^2 \partial_tu+\Delta^2 u+\Delta \theta & = f, \\
	\partial_t\theta-\Delta \theta-\Delta \partial_t u & =g, \\
	u(0)=u_0, \ \ \  \partial_tu(0)&=v_0,\\ 
	\theta(0) & =\theta_0,
	\end{split}
	\end{align*}
	where one couples to the Kelvin-Voigt plate-like  equation above the temperature similar to \cite[3.11.1]{LaTri_book_I}.
Then one obtains the first order system with
		\begin{align*}
	\A^{\SD} =
	\begin{bmatrix}
	\begin{array}{cc|c}
0 & -\id  & 0 \\
	\Delta^2 & \varepsilon  \Delta^2  & \Delta \\ 	\hline
	0 & -\Delta & -\Delta
	\end{array}
	\end{bmatrix}  = \begin{bmatrix}
	A & B \\ C & D
	\end{bmatrix}. %\quad \hbox{with } \Do(\A)=H^2(\R^d)\times H^2(\R^d)^2,
	\end{align*}
	Here, one considers  the operators  $A$ and $D$ on the  diagonal
	in $X_1=H^{4,p}(\R^d)\times L^p(\R^d)$ and $X_2=L^p(\R^d)$ with $\Do(A)=H^{4,p}(\R^d)\times H^{4,p}(\R^d)$ and $\Do(D)=H^{2,p}(\R^d)$, respectively.
	Therefore, we consider the above operator $\A^{\SD}$ on the space $X=X_1\times X_2$ and $\Do(\A^{\SD})=\Do(A)\times \Do(D)$. Note that $A$ was studied in Example \ref{ex:KelvinVoigt}.
	
	The block $B$ is $D$-bounded (more precisely, we have $\|B v\|_{X_1}\lesssim \|v\|_{\Do(D)}$ for all $v\in \Do(D)$), and the block $C$ given by 
	\begin{align*}
	\Big\|\begin{bmatrix} 0 & -\Delta\end{bmatrix}\begin{bmatrix} u \\ v\end{bmatrix}  \Big\|_{L^p(\R^d)} = \norm{\Delta v}_{L^p(\R^d)}
	\end{align*}
	is a lower order term compared to $A$, and therefore by Corollary~\ref{cor:calculus_pert_block_operators_lower_order} -- up to a shift -- the operator $\A$ has a bounded $H^\infty$-calculus of arbitrarily small angle. Here again these operators on domains can also be considered along the same lines provided that the boundary conditions guarantee that the operators $\varepsilon\Delta^2$ and $-\Delta$ have a bounded $H^\infty$-calculus.
\end{example}

\subsection{Beris-Edwards $Q$-tensor model for liquid crystals}
\label{ss:Beris_Edwards}
In the Beris-Edwards model of
nematic liquid crystal, the molecular orientation is described by
a so-called Q-tensor,  a function 
\begin{align*}
Q\colon \R_+\times\Dom\rightarrow S_{0,\C}^d \stackrel{{\rm def}}{=}
\{Q\in \C^{d\times d}\,\colon \,Q=Q^T \hbox{ and }  \tr Q=0 \}, 
\end{align*}and the fluid properties by the velocity field $u\colon\R_+\times \Dom \rightarrow \R^3$.
The Beris-Edwards model has been investigated recently by Wrona in \cite{Wrona_PhD} in maximal $L^p_t$-$L_x^2$-spaces.
 %with values in the symmetric traceless matrices.
There, a linearization of the full quasi-linear model in the strong setting  for fixed $Q_0\colon \Dom\rightarrow S_{0,\C}^d$
in the space 
\begin{align*}
X&= L^q_{\sigma}(\Dom)\times H^{1,q}(\Dom;S_{0,\C}^d) 
\end{align*}
 is  given by the block operator matrix
\begin{align*}
\A_q^{\EB}(Q_0) = 
\begin{bmatrix}
-\mathbb{P}_q \Delta & -\mathbb{P}\div S_{\xi}(Q_0) (\id-\Delta_{N,q}) \\ -\tilde{S}_\xi(Q_0)\nabla & \id-\Delta_{N,q} 
\end{bmatrix}=
\begin{bmatrix}
A & B \\ C & D
\end{bmatrix}
\end{align*}
with domain
\begin{align*}
\Do(\A_q^{\EB}(Q_0))&= \big( H_{0}^{1,q}(\Dom)^3\cap H_{\sigma}^{2,q}(\Dom)^3\big) \times  \{Q\in H^{3,q}(\Dom;S_{0,\C}^d)\colon \partial_n Q=0\}.
\end{align*}
Here  
$\partial_n Q$ denotes the outer normal derivative of $Q$, $\mathbb{P}_q$ denotes the Helmholtz projection in $L^q(\Dom;\R^3)$, and
\begin{align*}
S_{\xi}(Q)P&= [Q,P]-\tfrac{2\xi}{d}P-\xi\{Q,P\} + 2\xi(Q+\id/d)\tr(QP), \\
\tilde{S}_{\xi}(Q)P&= [\overline{Q},P^W]-\tfrac{2\xi}{d}P^D-\xi\{\overline{Q},P^D\} + 2\xi(\overline{Q}+\id/d)\tr(\overline{Q} P^D),
\end{align*}
with $[Q,P]=QP-PQ$ and $\{Q,P\}=QP+PQ$ denoting the commutator and the anti-commutator, respectively, $\overline{Q}$ the complex conjugate of $Q$, and
\begin{align*}
P^D=\tfrac{1}{2}(P+P^T) \quad \hbox{and} \quad P^W=\tfrac{1}{2}(P-P^T) 
\end{align*}
the symmetric and anti-symmetric part of $P\in S_{0,\C}^d$ respectively. 
The parameter $\xi\in \R$ describes the ration of tumbling and aligning effects.
First, we verify that these operators fit into the theory presented here.

\begin{lemma}\label{lemma:Edward-Berris} 
	Let $q\in(1,\infty)$, $\Dom\subseteq\R^3$ be a bounded domain with $\partial\Dom\in C^3$ and $Q_0\in W^{1,\infty}(\Dom;S_{0,\C}^d)$, then $\A_q^{\EB}(Q_0)$ satisfies Assumption~\ref{ass:standing} with $L=0$, and it satisfies the estimates in Assumption~\ref{ass:fractions_pm}~\nameref{ass:fractions_-}.  

\end{lemma}
\begin{proof}
Diagonal dominance follows from the estimates
		\begin{align*}
	\norm{BQ}_{L^{q}}&\lesssim \norm{(\div S_{\xi}(Q_0)) (-\Delta + \id)Q}_{L^{q}}
	+\norm{(S_{\xi}(Q_0)) \div (-\Delta + \id)Q}_{L^{q}} \\
	 &\lesssim  \norm{S_{\xi}(Q_0)}_{W^{1,\infty}}\norm{Q}_{H^{3,q}} \lesssim \norm{Q_0}_{W^{1,\infty}}\norm{DQ}_{H^{1,q}}, \\
	\norm{Cu}_{H^{1,q}}&=\norm{\tilde{S}_\xi(Q_0)\nabla u}_{H^{1,q}} \leq \norm{\tilde{S}_\xi(Q_0)}_{W^{1,\infty}}\norm{u}_{H^{2,q}}\lesssim \norm{Q_0}_{W^{1,\infty}}\norm{Au}_{L^{q}}
	\end{align*}
	for $u\in \Do (A)$ and $Q\in \Do(D)$, where one uses that $A=-\P_q \Delta$ and $D=\id-\Delta_{N,q}$ are boundedly invertible.
\end{proof}

In \cite[Corollary 3.2.7 and Theorem 3.2.16]{Wrona_PhD} it is shown that in the Hilbert space case $q=2$ the operator $\A_2^s(Q_0)$ is $\mathcal{J}$-symmetric,  sectorial, 
and it generates a contraction semigroup. We summarize this result without proof here.
\begin{proposition}[The case $q=2$]\label{prop:wrona}
		Let $\Dom\subseteq\R^d$ be a bounded domain with $\partial\Dom\in C^3$ and $Q_0\in W^{1,\infty}(\Dom;S_{0,\C}^d)$, then $\A_2^{\EB}(Q_0)$ is an invertible $\mathcal{J}$-symmetric and sectorial of angle $<\pi/2$. 
\end{proposition}

Corollary~\ref{cor:extrapolation} yields the boundedness of the $H^\infty$-calculus for  $q$ near $2$. 
\begin{proposition}[$\Rsec$-sectoriality and bounded $H^\infty$-calculus near $q=2$]
	Let   $\Dom\subseteq\R^d$ be a bounded domain with $\partial\Dom\in C^3$ and $Q_0\in W^{1,\infty}(\Dom;S_{0,\C}^d)$, then there exists $\delta>0$ such that for all $q\in (2-\delta, 2+ \delta)$ the operators 
$\A_q^{\EB}(Q_0)$ are invertible and have a bounded $H^\infty$-calculus of angle less than $\pi/2$. In particular $\A_q^{\EB}(Q_0)$  are $\Rsec$-sectorial of angle less than $\pi/2$.
\end{proposition}
\begin{proof}
The operators
	$\A_q^{\EB}(Q_0)$ are a consistent family of operators for $q\in (1,\infty)$ by Lemma \ref{l:consistency_A_theta}\eqref{it:consistency_for_nested_family} (note that \eqref{eq:equivalence_norm_A_0_A_theta} holds for $q$ close to $2$ by Remark \ref{r:extrapolation_equivalence_norm_A_0_A}). For $q=2$ the statement follows by Corollary~\ref{cor:Jsymmetry} and Proposition~\ref{prop:wrona}.
	Hence, using Lemma~\ref{lemma:Edward-Berris}, one can now apply Corollary~\ref{cor:extrapolation}, and  the statement follows. 
\end{proof}	
	\begin{remark}
	One can extend the above argument to prove the boundedness of the $H^{\infty}$-calculus for the Beris-Edwards operator on 
		\begin{align*}
	H^{s,q}_{\sigma}(\Dom)\times H^{1+s,q}(\Dom;S_{0,\C}^d) \ 
	\text{ for all }q\in (2-\delta, 2+ \delta)\text{ and }s\in (-\delta, \delta),
		\end{align*}
		for some $\delta>0$ assuming $Q_0\in C^{1,\alpha}(\Dom)$ with $\alpha>0$.
	\end{remark}

\subsection{A differential operator of Beris-Edwards type}\label{subsec:BerisEdwardsType}
In this subsection we study the following differential operator on $X= H^{-1,p}(\R^d)\times L^{p}(\R^d)^d$%\todo{Be careful with complex numbers.}
\begin{equation}
\label{eq:A_delta_full_space}
\A_{p}^{\Delta}
\stackrel{{\rm def}}{=}
\begin{bmatrix}
\id -\Delta & \div (\id-\Delta) \\
 \nabla & \id-\Delta
\end{bmatrix} \ \
\text{ with } \ \
\Do(\A_{p}^{\Delta})= H^{1,p}(\R^d)\times H^{2,p}(\R^d)^d
\end{equation}
where $p\in (1,\infty)$. This differential operator has a structure similar to the one studied in Subsection~\ref{ss:Beris_Edwards} which arises in the study of Beris-Edwards model for liquid crystals. In contrast to the latter, $\A^{\Delta}_{p}$ allows us to give more direct computations and therefore is more suited for illustrative purposes.

Let us denote by $\Delta_{p}^{\strong}$ and $\Delta_{p}^{\weak}$ the realization of the Laplace operator on $L^{p}(\R^d)$ and $H^{-1,p}(\R^d)$, respectively. Thus  $\A_{p}^{\Delta}$ is a block operator matrix with the choice
$$
A=\id -\Delta_{p}^{\weak},  \qquad B = \div(\id- \Delta_{p}^{\strong} ), \qquad C= \nabla , \qquad D=\id  -\Delta_{p}^{\strong},
$$
where $\Delta_{p}^{\strong}$ in the definition of $B$ and $D$ acts component-wise on $H^{2,p}(\R^d)^d$.

\begin{proposition}
Let $p\in (1,\infty)$. Then $\A_{p}^{\Delta}$ has a bounded $H^{\infty}$-calculus of angle $0$. Moreover, for all $\beta\in \big(-\big((1-\frac{1}{p})\vee \frac{1}{p}\big),1\big]$,
\begin{equation}
\label{eq:description_A_Delta_homogeneous_fractional_power}
\Dd((\A_{p}^{\Delta})^{\beta})=H^{-1+2\beta,p}(\R^d) \times  H^{ 2\beta,p}(\R^d)^d .
\end{equation}
%where $\dot{H}^{s,p}$ denotes the \emph{homogeneous} Sobolev spaces.
\end{proposition}
%Here $\dot{H}^{s,p}$ denotes the \emph{homogeneous} Sobolev spaces. %, see e.g.\ \cite[Subsection 6.2]{BeLo}.

\begin{proof}
It is easy to check that $\A_{p}^{\Delta}$ is diagonally dominant, and it does not fit into the special cases analyzed in Subsection \ref{s:special_cases_dominant}.
For the reader's convenience we split the proof into several steps.

\emph{Step 1: $\A_{p}^{\Delta}$ is $\Rsec$-sectorial of angle $0$}. Fix $\psi\in (0,\pi)$. 
By Theorem \ref{t:rsec_necessary_sufficient_condition_II} and the fact that $(1-\Delta)^{s/2}: H^{\beta+s,p}(\R^d)\to H^{\beta,p}(\R^d)$ is an isomorphism for all $s,\beta\in \R$, it is enough to show that  
\begin{equation*}
\begin{aligned}
M_{1}(\lambda)
&\stackrel{{\rm def}}{=}
\id - \div(\id- \Delta) ((\lambda-1)+\Delta)^{-1} \nabla ((\lambda-1) +\Delta)^{-1}\\
&\ =\id-\Delta (\id-\Delta) ((\lambda-1)+\Delta)^{-2}
\in \calL(L^{p}(\R^d)),\\
\end{aligned}
\end{equation*}
is invertible for all $\lambda\in \complement\overline{\Sigma_{\psi}}$ and 
\begin{equation}
\label{eq:R_sectoriality_M_delta}
\Rsec(M_{1}(\lambda)^{-1}\colon\lambda \in \complement\overline{\Sigma_{\psi}})<\infty.
\end{equation}
To check \eqref{eq:R_sectoriality_M_delta}, it is enough to check that the symbol of $(M_{1}(\lambda))^{-1}$ satisfies the Lizorkin condition (see e.g.\ \cite[Theorem 4.3.9]{pruss2016moving}). 
 Note that $(M_{1}(\lambda))^{-1}$ has symbol
\begin{align*}
\widetilde{m}_{\lambda}(\xi)&\stackrel{{\rm def}}{=}
\Big(1 +\frac{|\xi|^2(1+|\xi|^2)}{(\lambda-1-|\xi|^2)^2}\Big)^{-1}\\
&\ =
 \frac{(\lambda-1-|\xi|^2)^2}{ (\lambda-1-|\xi|^2)^2+ |\xi|^2(1+|\xi|^2)} \qquad \text{ for all } \xi\in \R^d.
\end{align*}
%%%
By standard computations, one can check that there exists $C>0$ independent of $\lambda,\xi$ such that 
$$
\sup\big\{
|\xi^{\alpha} \partial_{\xi}^{\alpha} 
\widetilde{m}_{\lambda}(\xi)|\,:\,\alpha=(\alpha_1,\dots,\alpha_d)\text{ such that }\alpha_k\in \{0,1\}
\big\}
\leq C.
$$ 
Thus \eqref{eq:R_sectoriality_M_delta} follows from \cite[Theorem 4.3.9]{pruss2016moving}. 

\emph{Step 2: Boundedness of the $H^{\infty}$-calculus}. 
The claim of this step can be proven also by employing Theorem \ref{t:extrapolation_H} and Proposition \ref{prop:hilbert_space_positivity}. Indeed, by Theorem \ref{t:extrapolation_H} and Step 1, it is enough to show that $\A^{\Delta}_{2}$ has a bounded $H^{\infty}$-calculus. Note that the consistency follows from Lemma \ref{l:consistency_A_theta} and Step 1.
 Since $\Delta_{2}^{\weak}$ and $\Delta_{2}^{\strong}$ generate $C_0$--semigroup of contractions, by Theorem \ref{t:extrapolation_H} and Proposition \ref{prop:hilbert_space_positivity}, it is enough to check %$\mathcal{J}$-symmetry %type condition 
 skew-symmetry as
 	in Remark \ref{eq:cancellation_H_B_C} with $\gamma=(4\pi^2)^{-1}$, that is for all $f\in H^{2}(\R^d)^d$ and $g\in H^1(\R^d)$,
\begin{equation}
\label{eq:J_simmetry_cond_Beris_Edward_type_op}
\gamma \Re(\div(\id-\Delta) f |  g)_{H^{-1}(\R^d)}+ \Re(\nabla g | f)_{L^2(\R^d;\R^d)}=0 .  
\end{equation}
%Here $H^{s}(\R^d)\stackrel{{\rm def}}{=}H^{s,2}(\R^d)$ for all $s\in \R$ and analogously for $H^s(\R^d)^d$. 

To prove \eqref{eq:J_simmetry_cond_Beris_Edward_type_op}, we equip the space $H^{-1}(\R^d)$ with the inner product
\begin{equation}
(f|g)_{H^{-1}(\R^d)}= \int_{\R^d}(1+|\xi|^2)^{-1} \wh{f}(\xi)\cdot \overline{\wh{g}(\xi)}\, \dd\xi,
\end{equation}
 where $\wh{h}(\xi)\stackrel{{\rm def}}{=} \int_{\R^d} h(x)e^{-2\pi i x\cdot \xi} \, \dd x$  denotes the Fourier transform of $h$ at $\xi\in \R^d$.
Since the Fourier transform of $\div (\id-\Delta) f$ is given by $ -4 \pi^2 (1+|\xi|^2) ( 2\pi i  \xi\cdot \wh{f})$, we have 
\begin{equation*}
\begin{aligned}
\big(\div (\id-\Delta) f| g\big)_{H^{-1}(\R^d)}
&=-4 \pi^2\int_{\R^d}  (2 \pi i\xi\cdot \wh{f}(\xi))
 \overline{\wh{g}(\xi)}\, \dd\xi\\
 &=- 4 \pi^2 \int_{\R^d}  \wh{f}(\xi)\cdot
 \overline{\wh{\nabla g}(\xi)} \, \dd\xi \\
 &=-4 \pi^2(f | \nabla g)_{L^2(\R^d)^d}=-4 \pi^2 \overline{(\nabla g|f)_{L^2(\R^d)^d}}.
\end{aligned}
\end{equation*}
This yields \eqref{eq:J_simmetry_cond_Beris_Edward_type_op} and therefore $\A_2^{\Delta}$ has a bounded $H^{\infty}$-calculus by Proposition \ref{prop:hilbert_space_positivity} and Remark \ref{eq:cancellation_H_B_C}.

\emph{Step 3: Proof of \eqref{eq:description_A_Delta_homogeneous_fractional_power}}. 
The case $\beta\in [0,1]$ follows from Proposition \ref{prop:fractional_powers_positive}  recalling that, for all $\beta\in \R$, 
\begin{equation}
\label{eq:fractional_powers_laplacian_Hsp}
\Dd((\id-\Delta_{p}^{\strong})^{\beta})= H^{2\beta,p}(\R^d)\quad \text{ and }\quad
\Dd((\id -\Delta_{p}^{\weak})^{\beta})= \Dd((-\Delta_{p}^{\strong})^{\beta-1}).
\end{equation}
The case $p=2$ follows from Step 2, Proposition \ref{prop:fractions_Hilbert} and \eqref{eq:fractional_powers_laplacian_Hsp}.
It remains to study the case $\beta<0$ and $p\neq 2$. To this end,  we apply Theorem \ref{t:extrapolation_H}\eqref{it:description_negative_fractional_powers}. 
 
\emph{Case $p\in (2,\infty)$}: Let $r\in (p,\infty)$ be arbitrary. For all $\theta\in [0,1]$ we set  $\A_{\theta}=\A^{\Delta}_q$ where $\frac{1}{q}=\frac{1-\theta}{2}+\frac{\theta}{r}$. By Step 1 and Lemma \ref{l:consistency_A_theta}, $(\A_{\theta})_{\theta\in [0,1]}$ is a consistent family of sectorial operators. By Step 2, \eqref{eq:fractional_powers_laplacian_Hsp} and Theorem \ref{t:extrapolation_H}\eqref{it:description_negative_fractional_powers},
\begin{equation}
\label{eq:fractional_case_1_negative}
\Dd((\A_p^{\Delta})^{-\beta})= \Dd((\Delta_{p}^{\weak})^{-\beta})\times \Dd((\Delta_{p}^{\strong})^{-\beta})= \dot{H}^{-1+2\beta,p}(\R^d) \times  \dot{H}^{ 2\beta,p}(\R^d)^d 
\end{equation}
for all $\beta\in (0,\frac{1}{2}(1-\varphi))$ where $\varphi\in (0,1)$ satisfy $\frac{1}{q}=\frac{1-\varphi}{2}+\frac{\varphi}{r}$. Since $1-\varphi= \frac{2}{q}\frac{r-q}{r-2}$, \eqref{eq:fractional_case_1_negative} holds for all $\beta\in (0,\frac{1}{q}\frac{r-q}{r-2})$. Thus \eqref{eq:description_A_Delta_homogeneous_fractional_power} in this case follows by letting $r\to \infty$.

\emph{Case $p\in (1,2)$}: The argument is similar to the previous case considering $\A_{\theta}=\A^{\Delta}_q$ where $\frac{1}{q}=\frac{1-\theta}{r}+\frac{\theta}{2}$ where $r\in (1,p)$ is arbitrary and letting $r\to 1$.
\end{proof}

\appendix
\section{A transference result for the $H^{\infty}$-calculus}\label{sec:appendix}
Let $X$ be a Banach space and $p\in [1,2]$, then the space $X$ has \emph{type $p$} if there exists a constant $C_p\geq 0$ such that for all finite sequences $x_1, \ldots, x_N$ in $X$ and $\varepsilon_1, \ldots, \varepsilon_N$ being Rademacher sequences on a probability space $(\O,\mathscr{A},\P)$ one has
\begin{align*}
\E \Big\|\sum_{n=1}^N\varepsilon_n x_n\Big\|_X^p  \leq C_p\sum_{n=1}^N \|x\|_X^p,
\end{align*}
and $X$ has \emph{non-trivial type} if it has type $p$ for some $p\in (1,2]$, compare e.g. \cite[Definition 7.1.1 f.]{Analysis2}.
%A=T B=S
\begin{theorem}\label{t:transference_appendix}
Assume that $X$ reflexive Banach space with non-trivial type. 
Let $T$ be a linear operator on $X$ with a bounded $H^\infty$-calculus 
and let $S$ be an $\Rsec$-sectorial operator on $X$ such that
\begin{equation}
\label{eq:ass_ST_equal_app}
\Do(T)=\Do(S)\ \ \text{ and } \ \ \|T x\|_{X}\eqsim \|S x\|_{X} \ \text{ for all }x\in \Do(T).
\end{equation}
Suppose that for some $\delta\in (0,1)$ one has
%\begin{equation}
%\label{eq:mapping_property_AB}
%		\begin{aligned}
%		S(\Do(T^{1+\delta}))&\subseteq \Do(T^{\delta}) \ \ \text{ and } \ \ &
%		\|T^{\delta}S x\|_{X}\lesssim \| T^{1+\delta} x\|_{X}& \quad \hbox{for all } x\in \Do(T^{1+\delta}),
%		\\
%		\Ran(S)&\subseteq \Ran(T^{\delta}) \ \ \text{ and } \ \ &
%		\|T^{-\delta}S x\|_{X}\lesssim \| T^{1-\delta} x\|_{X} & \quad \hbox{for all } x\in \Do(T).
%		\end{aligned}
%\end{equation}
\begin{equation}
\label{eq:mapping_property_AB}
\begin{aligned}
S(\Do(T^{1+\delta}))&\subseteq \Do(T^{\delta})\text{ and} &
\|T^{\delta}S x\|_{X}\lesssim \| T^{1+\delta} x\|_{X}& \hbox{ for all } x\in \Do(T^{1+\delta}),
\\
\Ran(S)&\subseteq \Ran(T^{\delta})\text{ and} &
\|T^{-\delta}S x\|_{X}\lesssim \| T^{1-\delta} x\|_{X} & \hbox{ for all } x\in \Do(T).
\end{aligned}
\end{equation}
Then $S$ has a bounded $H^{\infty}$-calculus and
$
\om_{\Rsec}(S)=\angH(S).
$
\end{theorem}

The above result is folklore known to experts. For the reader's convenience we provide the proof extending the arguments in \cite[Section 9]{KKW}.
It seems that the geometric assumptions on $X$ can be removed by using in the proof below the results in \cite[Subsection 5.3]{KLW19}, instead of \cite[Corollary 7.8]{KKW}.

\begin{proof}
Let us begin by collecting some useful facts.  By \cite[Corollary 10.4.10]{Analysis2}, it is enough to show that $S$ has a bounded $H^{\infty}$-calculus.
 Recall that 
 $(\Dd(T^\beta))_{\beta\in \R}$ is a complex interpolation scale w.r.t.\ $\beta$, cf. \cite[Proposition 2.2]{KKW}, that is
\begin{equation}
\label{eq:XA_complex_scale} 
[\Dd(T^{\beta_1}),\Dd(T^{\beta_2})]_{\theta}=
\Dd(T^{\beta_1(1-\theta)+\beta_2\theta}), \quad  \text{ for all }\theta\in (0,1),
\end{equation}
 since $T$ has a bounded $H^{\infty}$-calculus and hence bounded imaginary powers. 
Recall that, by the definition of  $ \Dd(T^{\beta})$ in Subsection \ref{ss:fractional_powers},
\begin{equation}
\label{eq:density_X_beta_delta}
\Do(T^{\alpha})\hookrightarrow \Dd(T^{\beta})\ \text{ is dense for all } 0\leq \beta\leq \alpha<\infty.
\end{equation}
Moreover, the operator $T$ induces operators
\begin{equation}
\label{eq:fractional_scale_A}
\dot{T}_{\theta}\colon \Dd(T^{1+\theta})\cap\Dd(T^{\theta})\subseteq \Dd(T^{\theta})\to \Dd(T^{\theta}) \quad \text{ for all } \theta\in \R,
\end{equation}
 on the $\Dd(T^{\theta})$-scale of spaces, 
compare \cite[Proposition 2.1]{KKW} or Subsection \ref{ss:fractional_powers}.
%\footnote{Amru: Should we set actually $\dot{T}_{\theta}$ instead of $T_{\theta}$ in accordance with (2.5)? If we take the domain as in (2.5) must we change the interpolations below accordingly?}

By \eqref{eq:mapping_property_AB}, \eqref{eq:density_X_beta_delta}, $\delta<1$ and the definition of the homogeneous scale  \eqref{eq:def_homogeneous_scale}, the operator $S$ extends uniquely to bounded linear operators $\widetilde{S}_{\pm \delta}\in \calL(\Dd(T^{1\pm \delta}),\Dd(T^{\pm \delta}))$ satisfying 
$\widetilde{S}_{\pm \delta}x= S x$ for all $x\in \Dd(T^{1\pm \delta})\cap \Do(T)$.  By restriction, we obtain the following unbounded linear operator
\begin{align*}
S_{\pm \delta}\colon \Dd(T^{1\pm \delta})\cap \Dd(T^{\pm \delta})\subseteq \Dd(T^{\pm \delta})\to \Dd(T^{\pm \delta}), &\\\text{ satisfying }\ \ 
S_{\pm \delta} x= S x \   \text{ for all }x\in \Dd(T^{1\pm \delta})\cap \Do(T).&
\end{align*}
Similarly, by restriction, interpolation and \eqref{eq:XA_complex_scale},
$S$ induced uniquely linear operators
\begin{equation}
\label{eq:extension_B}
S_{\theta}\colon \Dd(T^{\theta+1})\cap \Dd(T^{\theta})\subseteq \Dd(T^{\theta})\to \Dd(T^{\theta}) \quad \text{ for all } \theta\in [-\delta,\delta].
\end{equation}

Next we will need the following lemma which is proven below.

\begin{lemma}\label{lemma:B_minusalpha}
Under the assumptions of Theorem~\ref{t:transference_appendix} there exists a $\delta_0\in (0,1)$ such that $S_{-\delta_0}$ is  $\Rsec$-sectorial on $\Dd(T^{-\delta_0})$ with angle $\leq \angR(S)\vee \angH(T)$, and moreover $\Dd(S_{-\delta_0})= \Dd(T^{1-\delta_0})$.
\end{lemma}

To verify that \cite[Corollary 7.8]{KKW} is applicable to $S_{-\delta_0}$ one needs that by Lemma~\ref{lemma:B_minusalpha} $S_{-\delta_0}$ is $\Rsec$-sectorial on $\Dd(T^{-\delta_0})$  and that the spaces $\Dd(T^{\theta})$ have non-trivial type, which is equivalent to $B$-convexity, see \cite[Proposition 7.6.8]{Analysis2}. This follows since $\Dd(T^{\theta})$ are isomorphic to $X$, compare \cite[Section 2]{KKW}. 
Next, let $\langle \cdot, \cdot\rangle_{\theta}$ be the Rademacher interpolation functor introduced in \cite[Section 7]{KKW} (see also \cite{KLW19,LL21}). By \cite[Corollary 7.8]{KKW}, for all $\theta\in (0,1)$, the operator $
S_{\theta-\delta_0}|_{
\langle \Dd((S_{-\delta_0})^0),\Dd(S_{-\delta_0} )\rangle_{\theta} }$ induced by $S_{-\delta_0}$ on $
\langle \Dd((S_{-\delta_0})^0),\Dd(S_{-\delta_0} )\rangle_{\theta}$ has a bounded $H^{\infty}$-calculus (cf.\  \cite[Proposition 15.24]{KuWe} and \cite[Proposition 5.3.5]{LL21} for similar situations). Note that, by using the last assertion of Lemma \ref{lemma:B_minusalpha},
\begin{align*}
\langle \Dd((S_{-\delta_0})^0),\Dd(S_{-\delta_0} )\rangle_{\theta} 
&= \langle \Dd(T^{-\delta_0}),\Dd(T^{1-\delta_0}) \rangle_{\theta} \\
&\stackrel{(i)}{=} [ \Dd(T^{-\delta_0}),\Dd(T^{1-\delta_0})]_{\theta} = \Dd(T^{\theta-\delta_0}),
\end{align*}
where in $(i)$ we applied \cite[Theorem 7.4]{KKW} to $T$ which  guarantees that here the complex  and the Rademacher interpolation scale agree for $T$. 
Collecting the previous facts, we have
$$
S_{\theta-\delta_0}|_{\Dd(T^{\theta-\delta_0})} 
\text{ has a bounded $H^{\infty}$-calculus for all $\theta\in (0,1)$ on $\Dd(T^{\theta-\delta_0})$}.
$$
Choosing $\theta=\delta_0<1$ and using that $\Dd(T^{0})=X$, we obtain that $S_{0}=S$ has a bounded $H^{\infty}$-calculus on $X$. This completes the proof of Theorem \ref{t:transference_appendix}.
\end{proof}

\begin{proof}[Proof of Lemma~\ref{lemma:B_minusalpha}]
We split the proof into two steps.

\emph{Step 1}: There exists $\delta_1\in (0,\delta)$ such that  $\Dd(S_{\theta})= \Dd(T^{1+\theta})$ for all $\theta\in (-\delta_1,\delta_1)$. Firstly, let us note that the sectoriality of $S$ and \eqref{eq:ass_ST_equal_app} imply that $S$ extends to an isomorphism $\widetilde{S}$ between $\Dd(T)$ and $X$.
Secondly, recall that $S_{\theta}$ are defined as explained before \eqref{eq:extension_B} via complex interpolation and restriction. Hence, by \eqref{eq:XA_complex_scale} and the Sneiberg lemma (see e.g.\ \cite[Theorem 2.3 and 3.6]{TV88}) there exists a $\delta_1\in (0,\delta)$ such that $\widetilde{S}_{\theta}$ is an isomorphism for all $\theta\in [-\delta_1,\delta_1]$. In particular
$$
\|\widetilde{S}_{\theta}x \|_{\Dd(T^{\theta})}\eqsim \|x\|_{\Dd(T^{1+\theta})}\ \  \text{ for all }\theta\in [-\delta_1,\delta_1].
$$
Now the conclusion follows from the definition of $S_{\theta}$ as the restriction of $\widetilde{S}_{\theta}$ on $\Dd(T^{\theta})$ and a density argument (see e.g.\ \cite[Proposition 5.3.1]{KLW19}).

\emph{Step 2}: There exists $\delta_0\in (0,1)$ such that $S_{-\delta_0}$ is $\Rsec$-sectorial on $\Dd(T^{-\delta_0})$ with angle $\leq \angR(S)\vee \angH(T)$. Let $\delta_1\in (0,\delta)$ be as in Step 1. We begin by introducing suitable spaces of sequences. To this end, let $(\lambda_j)_{j\geq 1}\subseteq \complement\Sigma_{\phi}$ be a dense subset with $\phi\in (\om_{\Rsec}(S)\vee \angH(T),\pi)$, and let $(\varepsilon_j)_{j\geq 1}$ be a Rademacher sequence over a probability space $(\O,\mathscr{A},\P)$. 
Then, for $\theta\in [-\delta_1,\delta_1]$, we set
\begin{equation}
\begin{aligned}
\label{eq:definition_sequences_space}
\mathcal{X}_{\theta}
&\stackrel{{\rm def}}{=}
\Big\{(x_j)_{j\geq 1}\subseteq \Dd(T^{\theta})\colon \E\Big\| \sum_{j\geq 1}  \varepsilon_j x_j \Big\|_{\Dd(T^{\theta})}<\infty \Big\},\\
\mathcal{Y}_{\theta}
&\stackrel{{\rm def}}{=}
\Big\{(x_j)_{j\geq 1}\subseteq \Dd(T^{\theta+1})\cap \Dd(T^{\theta})\colon \E\Big\| \sum_{j\geq 1} (\lambda_j-\dot{T}_{\theta})\varepsilon_j x_j \Big\|_{\Dd(T^{\theta})} <\infty \Big\},
\end{aligned}
\end{equation}
endowed with the natural norms.

Since $X$ has non-trivial type and $\Dd(T^{\theta})$ is isomorphic to $X$, compare \cite[Section 2]{KKW}, $\Dd(T^{\theta})$ has non-trivial type as well and therefore, it is $K$-convex due to \cite[Theorem 7.4.23]{Analysis2}. In particular, $(\mathcal{X}_{\theta})_{\theta\in [-\delta_1,\delta_1}]$ is a complex interpolation scale by \cite[Theorem 7.4.16(1)]{Analysis2} and \eqref{eq:XA_complex_scale}. 
By $\Rsec$-sectoriality of $T$, the assignment $(x_j)_{j\geq 1}\mapsto \big((\lambda_j-\dot{T}_{\theta})^{-1}x_j\big)_{j\geq 1}$ induces an isomorphism between $\mathcal{X}_{\theta}$ and $\mathcal{Y}_{\theta}$, cf.\ \eqref{eq:definition_Tdot}.  By compatibility \eqref{eq:consistency_T_T_gamma} and \cite[Proposition 2.1]{KKW}, $(\mathcal{Y}_{\theta})_{\theta\in [-\delta_1,\delta_1]}$ is also a complex scale.

For any sequence $\x=(x_j)_{j\geq 1}\in \mathcal{Y}_{\theta}$ we set 
$$
\mathcal{T}_{\theta}(\x)=((\lambda_j -S_{\theta}) x_j)_{j\geq 1}.
$$
By \eqref{eq:fractional_scale_A}--\eqref{eq:extension_B} and $\phi>\angR(S)\vee \angH(T)$, we get that
$$
\mathcal{T}_\theta\colon \mathcal{Y}_{\theta}\to \mathcal{X}_{\theta}\text{ is bounded for }\theta\in [-\delta_1,\delta_1].
$$
Next we prove that $\mathcal{T}_0 \in \calL(\mathcal{Y}_0,\mathcal{X}_0)$ is an isomorphism with inverse
$$
\mathcal{S}_0(\x)= ((\lambda_j -S )^{-1} x_j)_{j \geq 1},
$$ 
where $S=S_0$.
Clearly, $\mathcal{T}_0\mathcal{S}_0$ and $\mathcal{S}_0\mathcal{T}_0$ are equal to the identity on the subset of finite sequences in $\mathcal{Y}_0$ and $\mathcal{X}_0$, respectively. By density it remains to prove that $\mathcal{S}_0\in \calL(\mathcal{X}_0,\mathcal{Y}_0)$. 
To this end, recall that $\Do(T)=\Do(S)$ and thus $\|T x\|_{X} \eqsim \|S x\|_{X}$ for all $x\in \Do(T)$. The latter implies
\begin{equation}
\begin{aligned}
\label{eq:estimate_R_sectoriality}
&\|\mathcal{S}_0(\x)\|_{\mathcal{Y}_0}\\
&\leq\E\Big\|\sum_{j\geq 1} \varepsilon_j \lambda_j (\lambda_j- S)^{-1}x_j \Big\|_X
+\E\Big\|T \sum_{j\geq 1}  \varepsilon_j   (\lambda_j- S)^{-1}x_j \Big\|_X\\
&\eqsim 
\E\Big\|\sum_{j\geq 1}  \varepsilon_j \lambda_j (\lambda_j- S)^{-1}x_j \Big\|_X
+\E\Big\|\sum_{j\geq 1}  \varepsilon_j S (\lambda_j- S)^{-1}x_j \Big\|_X\\
&\stackrel{(i)}{\leq } \sup_{J\geq 1}\Big(
\E\Big\|\sum_{ 1\leq j \leq J} \varepsilon_j  \lambda_j (\lambda_j- S)^{-1}x_j \Big\|_X
+\E\Big\|\sum_{1\leq j \leq J} \varepsilon_j S (\lambda_j- S)^{-1}x_j \Big\|_X
\Big)\\
&\stackrel{(ii)}{\lesssim}  \Rsec(S(\lambda - S)^{-1}:\lambda\in  \Sigma_{\phi})\, \E \Big\|\sum_{j\geq 1} \varepsilon_j x_j \Big\|_{X}\lesssim \|\x\|_{\mathcal{X}_0},
\end{aligned}
\end{equation}
where in $(i)$ we used  Fatou's lemma and in $(ii)$ the $\Rsec$-sectoriality of $S$, $(\lambda_j)_{j\geq 1}\subseteq\complement\overline{\Sigma_{\phi}}$ and $\phi>\angR(S)\vee \angH(T)$. Note that the constants 
\eqref{eq:estimate_R_sectoriality} are independent of the choice of the sequence $(\lambda_j)_{j\geq 1}\subseteq \complement\Sigma_{\phi}$.
 
Similar to Step 1, Sneiberg's lemma \cite[Theorem 2.3 and 3.6]{TV88} is applicable to $\mathcal{T}_\theta$. Thus  there exists $\delta_0\in (0,\delta_1)$, independent of $(\lambda_j)_{j\geq 1}\subseteq \complement\overline{\Sigma_{\phi}}$, such that 
\begin{equation}
\mathcal{T}_\theta:\mathcal{Y}_{\theta}\to \mathcal{X}_{\theta}  \text{ is boundedly invertible for all }\theta\in [-\delta_0,\delta_0]
\end{equation}
and 
$$
 \|\mathcal{T}_\theta^{-1}\|_{\calL(\mathcal{X}_{\theta},\mathcal{Y}_{\theta})}\leq C_\theta \Rsec(S(\lambda - S)^{-1}:\lambda\in  \Sigma_{\phi})
 \stackrel{{\rm def}}{=}
 C_{S,\theta},
$$
where $C_\theta>0$ is independent of the sequence $(\lambda_j)_{j\geq 1}$.

Since $\delta_0<\delta<1$, by \cite[Proposition 5.3.1]{KLW19} the embedding $\Do(T)\cap \Ran(T)\embed \Dd(T^{\theta})$ is dense for all $\theta\in [-\delta_0,\delta_0]$.
Hence, one can check that
$\mathcal{T}_\theta^{-1}:\mathcal{X}_{\theta}\to \mathcal{Y}_{\theta}$ for all $\theta\in [-\delta_0,\delta_0]$ is given by
$$
\mathcal{T}_\theta^{-1}(\x)\stackrel{{\rm def}}{=}((\lambda_j-S_{\theta})^{-1}x_j)_{j\geq 1}.
$$ 
Since $\mathcal{T}_{\delta_0}^{-1}\in \calL(\mathcal{X}_{-\delta_0},\mathcal{Y}_{-\delta_0})$, for all finite set $J\subseteq \Z$ and finite sequence $(x_j)_{j=1}^{\# J}\subseteq \Dd(T^{-\delta_0})$ we have
\begin{equation}
\label{eq:inequality_R_boundedness_B}
\E \Big\|\sum_{j\in J}\varepsilon_j  \lambda_j  (\lambda_j-S_{-\delta_0})^{-1}x_j \Big\|_{\Dd(T^{-\delta_0})}
\leq C_{S,\theta}\,\E\Big\|\sum_{j\in J} \varepsilon_j x_j\Big\|_{\Dd(T^{-\delta_0})}.
\end{equation}

Recall that $(\lambda_j)_{j\geq 1}\supseteq \complement\overline{\Sigma_{\phi}}$ is dense. Equation \eqref{eq:inequality_R_boundedness_B} and the continuity of the resolvent map $\lambda \mapsto (\lambda-S_{-\delta_0})^{-1}$ ensure that
$
\rho(S_{-\delta_0})\supseteq \complement\overline{\Sigma_{\phi}}
$
and \eqref{eq:inequality_R_boundedness_B} holds for all finite set $(\lambda_j)_{j\in J}\subseteq\C\setminus \Sigma_{\phi}$. In particular, $S_{-\delta_0}$ is  $\Rsec$-sectorial with angle $<\phi$.  The arbitrariness of $\phi$ yields $\angR(S_{-\delta_0})\leq \angR(S)\vee \angH(T)$.
%%%%
\end{proof}

\subsection*{Acknowledgement} We would like to thank Tim Binz, Emiel Lorist and Mark Veraar for valuable discussions. We also thank the anonymous referees for their helpful comments and suggestions, and for the very accurate reading of the manuscript.

\bibliographystyle{plain}
\bibliography{literature}

\def\polhk#1{\setbox0=\hbox{#1}{\ooalign{\hidewidth
  \lower1.5ex\hbox{`}\hidewidth\crcr\unhbox0}}} \def\cprime{$'$}
\begin{thebibliography}{100}

\bibitem{Adler_2018}
M.~Adler and K.-J. Engel.
\newblock Spectral theory for structured perturbations of linear operators.
\newblock {\em J. Spectr. Theory}, 8(4):1393--1442, 2018.

\bibitem{ADN_1964}
S.~Agmon, A.~Douglis, and L.~Nirenberg.
\newblock Estimates near the boundary for solutions of elliptic partial
  differential equations satisfying general boundary conditions. {II}.
\newblock {\em Comm. Pure Appl. Math.}, 17:35--92, 1964.

\bibitem{ALV23}
A.~Agresti, N.~Lindemulder, and M.C. Veraar.
\newblock On the trace embedding and its applications to evolution equations.
\newblock {\em Mathematische Nachrichten}, 296(4):1319--1350, 2023.

\bibitem{AV19}
A.~Agresti and M.C. Veraar.
\newblock Stability properties of stochastic maximal {$L^p$}-regularity.
\newblock {\em J. Math. Anal. Appl.}, 482(2):123553, 35, 2020.

\bibitem{AV19_QSEE_1}
A.~Agresti and M.C. Veraar.
\newblock Nonlinear parabolic stochastic evolution equations in critical spaces
  part {I}. {S}tochastic maximal regularity and local existence.
\newblock {\em Nonlinearity}, 35(8):4100--4210, 2022.

\bibitem{AV19_QSEE_2}
A.~Agresti and M.C. Veraar.
\newblock Nonlinear parabolic stochastic evolution equations in critical spaces
  part {II}.
\newblock {\em J. Evol. Equ.}, 22(2):Paper No. 56, 2022.

\bibitem{Am}
H.~Amann.
\newblock {\em Linear and quasilinear parabolic problems. {V}ol. {I},
  {A}bstract linear theory}, volume~89 of {\em Monographs in Mathematics}.
\newblock Birkh\"auser Boston Inc., Boston, MA, 1995.

\bibitem{AHS1994}
H.~Amann, M.~Hieber, and G.~Simonett.
\newblock Bounded {$H_\infty$}-calculus for elliptic operators.
\newblock {\em Differential Integral Equations}, 7(3-4):613--653, 1994.

\bibitem{Amansag_etal2022}
A.~Amansag, H.~Bounit, A.~Driouich, and S.~Hadd.
\newblock Staffans-{W}eiss perturbations for maximal {$L^p$}-regularity in
  {B}anach spaces.
\newblock {\em J. Evol. Equ.}, 22(1):Paper No. 15, 14, 2022.

\bibitem{Arl_2002}
Y.~Arlinski\u{\i}.
\newblock On sectorial block operator matrices.
\newblock {\em Mat. Fiz. Anal. Geom.}, 9(4):533--571, 2002.

\bibitem{DivH}
P.~Auscher, N.~Badr, R.~Haller-Dintelmann, and J.~Rehberg.
\newblock The square root problem for second-order, divergence form operators
  with mixed boundary conditions on {$L^p$}.
\newblock {\em J. Evol. Equ.}, 15(1):165--208, 2015.

\bibitem{Barta_2008}
T.~B\'{a}rta.
\newblock On {R}-sectorial derivatives on {B}ergman spaces.
\newblock {\em Bull. Aust. Math. Soc.}, 77(2):305--313, 2008.

\bibitem{Bat2005}
A.~B\'{a}tkai, P.~Binding, A.~Dijksma, R.~Hryniv, and H.~Langer.
\newblock Spectral problems for operator matrices.
\newblock {\em Math. Nachr.}, 278(12-13):1408--1429, 2005.

\bibitem{Batkai_2012}
A.~B\'{a}tkai, P.~Csom\'{o}s, K.-J. Engel, and B.~Farkas.
\newblock Stability and convergence of product formulas for operator matrices.
\newblock {\em Integral Equations Operator Theory}, 74(2):281--299, 2012.

\bibitem{Ben_2010}
A.~Ben~Amar, A.~Jeribi, and B.~Krichen.
\newblock Essential spectra of a {$3\times 3$} operator matrix and an
  application to three-group transport equations.
\newblock {\em Integral Equations Operator Theory}, 68(1):1--21, 2010.

\bibitem{BeLo}
J.~Bergh and J.~L{\"o}fstr{\"o}m.
\newblock {\em Interpolation spaces. {A}n introduction}.
\newblock Springer-Verlag, Berlin, 1976.
\newblock Grundlehren der Mathematischen Wissenschaften, No. 223.

\bibitem{Mug_2008}
S.~Cardanobile and D.~Mugnolo.
\newblock Qualitative properties of coupled parabolic systems of evolution
  equations.
\newblock {\em Ann. Sc. Norm. Super. Pisa Cl. Sci. (5)}, 7(2):287--312, 2008.

\bibitem{Char_2014}
S.~Charfi, Aref J., and R.~Moalla.
\newblock Essential spectra of operator matrices and applications.
\newblock {\em Math. Methods Appl. Sci.}, 37(4):597--608, 2014.

\bibitem{Cho1967}
A.J. Chorin.
\newblock The numerical solution of the {N}avier-{S}tokes equations for an
  incompressible fluid.
\newblock {\em Bull. Amer. Math. Soc.}, 73:928--931, 1967.

\bibitem{Cho1968}
A.J. Chorin.
\newblock Numerical solution of the {N}avier-{S}tokes equations.
\newblock {\em Math. Comp.}, 22:745--762, 1968.

\bibitem{Choudhury2018}
A.P. Choudhury, A.~Hussein, and P.~Tolksdorf.
\newblock Nematic liquid crystals in {L}ipschitz domains.
\newblock {\em SIAM J. Math. Anal.}, 50(4):4282--4310, 2018.

\bibitem{Coulhon1986}
Thierry Coulhon and Damien Lamberton.
\newblock R\'{e}gularit\'{e} {$L^p$} pour les \'{e}quations d'\'{e}volution.
\newblock In {\em S\'{e}minaire d'{A}nalyse {F}onctionelle 1984/1985},
  volume~26 of {\em Publ. Math. Univ. Paris VII}, pages 155--165. Univ. Paris
  VII, Paris, 1986.

\bibitem{Dan_2010}
R.~Danchin.
\newblock On the solvability of the compressible {N}avier-{S}tokes system in
  bounded domains.
\newblock {\em Nonlinearity}, 23(2):383--407, 2010.

\bibitem{DDHPV}
R.~Denk, G.~Dore, M.~Hieber, J.~Pr{\"u}ss, and A.~Venni.
\newblock New thoughts on old results of {R}. {T}.\ {S}eeley.
\newblock {\em Math. Ann.}, 328(4):545--583, 2004.

\bibitem{DHP}
R.~Denk, M.~Hieber, and J.~Pr{\"u}ss.
\newblock {$R$}-boundedness, {F}ourier multipliers and problems of elliptic and
  parabolic type.
\newblock {\em Mem. Amer. Math. Soc.}, 166(788), 2003.

\bibitem{Hum_2019}
R.~Denk and F.~Hummel.
\newblock Dispersive mixed-order systems in {$L^p$}-{S}obolev spaces and
  application to the thermoelastic plate equation.
\newblock {\em Adv. Differential Equations}, 24(7-8):377--406, 2019.

\bibitem{DenkKaip}
R.~Denk and M.~Kaip.
\newblock {\em General parabolic mixed order systems in {${L_p}$} and
  applications}, volume 239 of {\em Operator Theory: Advances and
  Applications}.
\newblock Birkh\"auser/Springer, Cham, 2013.

\bibitem{Racke_2009}
R.~Denk, R.~Racke, and Y.~Shibata.
\newblock {$L_p$} theory for the linear thermoelastic plate equations in
  bounded and exterior domains.
\newblock {\em Adv. Differential Equations}, 14(7-8):685--715, 2009.

\bibitem{Shib_2017}
R.~Denk and Y.~Shibata.
\newblock Maximal regularity for the thermoelastic plate equations with free
  boundary conditions.
\newblock {\em J. Evol. Equ.}, 17(1):215--261, 2017.

\bibitem{EgPhDthesis}
M.~Egert.
\newblock {\em On Kato’s conjecture and mixed boundary conditions, Sierke,
  2015}.
\newblock PhD thesis, PhD. Thesis, TU Darmstadt.

\bibitem{Engel_1995}
K.-J. Engel.
\newblock Operator matrices and systems of evolution equations.
\newblock Habilitationsschrift, Universit\"at T\"ubingen, 1995.

\bibitem{Eng_1986}
K.-J. Engel and R.~Nagel.
\newblock On the spectrum of certain systems of linear evolution equations.
\newblock In {\em Differential equations in {B}anach spaces ({B}ologna, 1985)},
  volume 1223 of {\em Lecture Notes in Math.}, pages 102--109. Springer,
  Berlin, 1986.

\bibitem{Ericksen}
J.L. Ericksen.
\newblock Hydrostatic theory of liquid crystals.
\newblock {\em Arch. Rational Mech. Anal.}, 9:371--378, 1962.

\bibitem{F14_extrapolation_MR}
S.~Fackler.
\newblock The {K}alton-{L}ancien theorem revisited: maximal regularity does not
  extrapolate.
\newblock {\em J. Funct. Anal.}, 266(1):121--138, 2014.

\bibitem{Nau_2014}
S.~Fackler and T.~Nau.
\newblock Local strong solutions for the non-linear thermoelastic plate
  equation on rectangular domains in {$L^p$}-spaces.
\newblock {\em NoDEA Nonlinear Differential Equations Appl.}, 21(6):775--794,
  2014.

\bibitem{Fackler_Counterexamples}
Stephan Fackler.
\newblock Regularity properties of sectorial operators: counterexamples and
  open problems.
\newblock In {\em Operator semigroups meet complex analysis, harmonic analysis
  and mathematical physics}, volume 250 of {\em Oper. Theory Adv. Appl.}, pages
  171--197. Birkh\"{a}user/Springer, Cham, 2015.

\bibitem{Winkler2021}
J.~Fuhrmann, J.~Lankeit, and M.~Winkler.
\newblock A double critical mass phenomenon in ano-flux-{D}irichlet
  {K}eller-{S}egel system.
\newblock {\em arXiv preprint arXiv:2101.06748}, 2021.

\bibitem{Galdi11}
G.~Galdi.
\newblock {\em An introduction to the mathematical theory of the Navier-Stokes
  equations: Steady-state problems}.
\newblock Springer Science \& Business Media, 2011.

\bibitem{Kost_2019}
L.~Grubi\v{s}i\'{c}, V.~Kostrykin, K.A. Makarov, S.~Schmitz, and
  K.~Veseli\'{c}.
\newblock The {T}an {$2\Theta$} theorem in fluid dynamics.
\newblock {\em J. Spectr. Theory}, 9(4):1431--1457, 2019.

\bibitem{Kost2013_1}
L.~Grubi\v{s}i\'{c}, V.~Kostrykin, K.A. Makarov, and K.~Veseli\'{c}.
\newblock Representation theorems for indefinite quadratic forms revisited.
\newblock {\em Mathematika}, 59(1):169--189, 2013.

\bibitem{Kost2013_2}
L.~Grubi\v{s}i\'{c}, V.~Kostrykin, K.A. Makarov, and K.~Veseli\'{c}.
\newblock The {T}an {$2\Theta$} theorem for indefinite quadratic forms.
\newblock {\em J. Spectr. Theory}, 3(1):83--100, 2013.

\bibitem{HaKu2006}
B.H. Haak, M.~Haase, and P.C. Kunstmann.
\newblock Perturbation, interpolation, and maximal regularity.
\newblock {\em Adv. Differential Equations}, 11(2):201--240, 2006.

\bibitem{Haase:2}
M.~Haase.
\newblock {\em The functional calculus for sectorial operators}, volume 169 of
  {\em Operator Theory: Advances and Applications}.
\newblock Birkh\"auser Verlag, Basel, 2006.

\bibitem{Hieberetal_2016}
M.~Hieber, M.~Nesensohn, J.~Pr\"{u}ss, and K.~Schade.
\newblock Dynamics of nematic liquid crystal flows: the quasilinear approach.
\newblock {\em Ann. Inst. H. Poincar\'{e} Anal. Non Lin\'{e}aire},
  33(2):397--408, 2016.

\bibitem{HieberPruess2016}
M.~Hieber and J.~Pr\"{u}ss.
\newblock Modeling and analysis of the {E}ricksen-{L}eslie equations for
  nematic liquid crystal flows.
\newblock In {\em Handbook of mathematical analysis in mechanics of viscous
  fluids}, pages 1075--1134. Springer, Cham, 2018.

\bibitem{Hillen2009}
T.~Hillen and K.~J. Painter.
\newblock A user's guide to {PDE} models for chemotaxis.
\newblock {\em J. Math. Biol.}, 58(1-2):183--217, 2009.

\bibitem{HorstmannI}
D.~Horstmann.
\newblock From 1970 until present: the {K}eller-{S}egel model in chemotaxis and
  its consequences. {I}.
\newblock {\em Jahresber. Deutsch. Math.-Verein.}, 105(3):103--165, 2003.

\bibitem{HorstmannII}
D.~Horstmann.
\newblock From 1970 until present: the {K}eller-{S}egel model in chemotaxis and
  its consequences. {II}.
\newblock {\em Jahresber. Deutsch. Math.-Verein.}, 106(2):51--69, 2004.

\bibitem{Hua_2019}
J.~Huang, J.~Sun, A.~Chen, and C.~Trunk.
\newblock Invertibity of {$2\times 2$} operator matrices.
\newblock {\em Math. Nachr.}, 292(11):2411--2426, 2019.

\bibitem{Analysis2}
T.P. Hyt\"onen, J.M.A.M.~van Neerven, M.C. Veraar, and L.~Weis.
\newblock {\em Analysis in {B}anach spaces. {V}ol. {II}. {P}robabilistic
  {M}ethods and {O}perator {T}heory.}, volume~67 of {\em Ergebnisse der
  Mathematik und ihrer Grenzgebiete. 3. Folge.}
\newblock Springer, 2017.

\bibitem{Ibro_2017_2}
O.O. Ibrogimov.
\newblock Essential spectrum of non-self-adjoint singular matrix differential
  operators.
\newblock {\em J. Math. Anal. Appl.}, 451(1):473--496, 2017.

\bibitem{Ibro_2013}
O.O. Ibrogimov, H.~Langer, M.~Langer, and C.~Tretter.
\newblock Essential spectrum of systems of singular differential equations.
\newblock {\em Acta Sci. Math. (Szeged)}, 79(3-4):423--465, 2013.

\bibitem{Ibro_2016}
O.O. Ibrogimov, P.~Siegl, and C.~Tretter.
\newblock Analysis of the essential spectrum of singular matrix differential
  operators.
\newblock {\em J. Differential Equations}, 260(4):3881--3926, 2016.

\bibitem{Ibro_2017}
O.O. Ibrogimov and C.~Tretter.
\newblock Essential spectrum of elliptic systems of pseudo-differential
  operators on {$L^2(\Bbb{R}^N)\oplus L^2(\Bbb{R}^N)$}.
\newblock {\em J. Pseudo-Differ. Oper. Appl.}, 8(2):147--166, 2017.

\bibitem{Jeribi}
A.~Jeribi.
\newblock {\em Spectral theory and applications of linear operators and block
  operator matrices}.
\newblock Springer, Cham, 2015.

\bibitem{KaTe2018}
Y.~Kagei, T.~Nishida, and Y.~Teramoto.
\newblock On the spectrum for the artificial compressible system.
\newblock {\em J. Differential Equations}, 264(2):897--928, 2018.

\bibitem{KKW}
N.J. Kalton, P.C. Kunstmann, and L.~Weis.
\newblock Perturbation and interpolation theorems for the {$H\sp
  \infty$}-calculus with applications to differential operators.
\newblock {\em Math. Ann.}, 336(4):747--801, 2006.

\bibitem{KLW19}
N.J. Kalton, E.~Lorist, and L.~Weis.
\newblock Euclidean structures and operator theory in {B}anach spaces, 2019.

\bibitem{KWcalc}
N.J. Kalton and L.~Weis.
\newblock The {$H^\infty$}-calculus and sums of closed operators.
\newblock {\em Math. Ann.}, 321(2):319--345, 2001.

\bibitem{Kato}
T.~Kato.
\newblock Fractional powers of dissipative operators.
\newblock {\em J. Math. Soc. Japan}, 13:246--274, 1961.

\bibitem{PrussWeight1}
M.~K{\"o}hne, J.~Pr{\"u}ss, and M.~Wilke.
\newblock On quasilinear parabolic evolution equations in weighted
  {$L_p$}-spaces.
\newblock {\em J. Evol. Equ.}, 10(2):443--463, 2010.

\bibitem{Kom_I}
H.~Komatsu.
\newblock Fractional powers of operators.
\newblock {\em Pacific J. Math.}, 19:285--346, 1966.

\bibitem{Kom_II}
H.~Komatsu.
\newblock Fractional powers of operators. {II}. {I}nterpolation spaces.
\newblock {\em Pacific J. Math.}, 21:89--111, 1967.

\bibitem{Kom_III}
H.~Komatsu.
\newblock Fractional powers of operators. {III}. {N}egative powers.
\newblock {\em J. Math. Soc. Japan}, 21:205--220, 1969.

\bibitem{Kom_IV}
H.~Komatsu.
\newblock Fractional powers of operators. {IV}. {P}otential operators.
\newblock {\em J. Math. Soc. Japan}, 21:221--228, 1969.

\bibitem{Kom_V}
H.~Komatsu.
\newblock Fractional powers of operators. {V}. {D}ual operators.
\newblock {\em J. Fac. Sci. Univ. Tokyo Sect. I}, 17:373--396, 1970.

\bibitem{Kom_VI}
H.~Komatsu.
\newblock Fractional powers of operators. {VI}. {I}nterpolation of non-negative
  operators and imbedding theorems.
\newblock {\em J. Fac. Sci. Univ. Tokyo Sect. IA Math.}, 19:1--63, 1972.

\bibitem{Kost_2007}
V.~Kostrykin, K.A. Makarov, and A.~K. Motovilov.
\newblock Perturbation of spectra and spectral subspaces.
\newblock {\em Trans. Amer. Math. Soc.}, 359(1):77--89, 2007.

\bibitem{Kost_2004}
V.~Kostrykin, K.A. Makarov, and A.K. Motovilov.
\newblock A generalization of the {$\tan 2\Theta$} theorem.
\newblock In {\em Current trends in operator theory and its applications},
  volume 149 of {\em Oper. Theory Adv. Appl.}, pages 349--372. Birkh\"{a}user,
  Basel, 2004.

\bibitem{Kost_2005}
V.~Kostrykin, K.A. Makarov, and A.K. Motovilov.
\newblock On the existence of solutions to the operator {R}iccati equation and
  the {$\tan\Theta$} theorem.
\newblock {\em Integral Equations Operator Theory}, 51(1):121--140, 2005.

\bibitem{Kun_2015}
P.~C. Kunstmann and M.~Uhl.
\newblock {$L^p$}-spectral multipliers for some elliptic systems.
\newblock {\em Proc. Edinb. Math. Soc. (2)}, 58(1):231--253, 2015.

\bibitem{Kunstmann2005}
P.C. Kunstmann.
\newblock On elliptic non-divergence operators with measurable coefficients.
\newblock In {\em Nonlinear elliptic and parabolic problems}, volume~64 of {\em
  Progr. Nonlinear Differential Equations Appl.}, pages 265--272.
  Birkh\"{a}user, Basel, 2005.

\bibitem{KuWePert}
P.C. Kunstmann and L.~Weis.
\newblock Perturbation theorems for maximal {${L}_p$}-regularity.
\newblock {\em Ann. Scuola Norm. Sup. Pisa Cl. Sci. (4)}, 30(2):415--435, 2001.

\bibitem{KuWe}
P.C. Kunstmann and L.~Weis.
\newblock Maximal {$L\sb p$}-regularity for parabolic equations, {F}ourier
  multiplier theorems and {$H\sp \infty$}-functional calculus.
\newblock In {\em Functional analytic methods for evolution equations}, volume
  1855 of {\em Lecture Notes in Math.}, pages 65--311. Springer, Berlin, 2004.

\bibitem{KW13_errata}
P.C. Kunstmann and L.~Weis.
\newblock Erratum to: {P}erturbation and interpolation theorems for the
  {$H^\infty$}-calculus with applications to differential operators.
\newblock {\em Math. Ann.}, 357(2), 2013.

\bibitem{Stokes}
P.C. Kunstmann and L.~Weis.
\newblock New criteria for the {$H^\infty$}-calculus and the {S}tokes operator
  on bounded {L}ipschitz domains.
\newblock {\em J. Evol. Equ.}, 17(1):387--409, 2017.

\bibitem{Lan_2001}
H.~Langer, A.~Markus, V.~Matsaev, and C.~Tretter.
\newblock A new concept for block operator matrices: the quadratic numerical
  range.
\newblock {\em Linear Algebra Appl.}, 330(1-3):89--112, 2001.

\bibitem{Winkler2020}
J.~Lankeit and M.~Winkler.
\newblock Facing low regularity in chemotaxis systems.
\newblock {\em Jahresber. Dtsch. Math.-Ver.}, 122(1):35--64, 2020.

\bibitem{LaTri_book_I}
I.~Lasiecka and R.~Triggiani.
\newblock {\em Control theory for partial differential equations: continuous
  and approximation theories. {I}}, volume~74 of {\em Encyclopedia of
  Mathematics and its Applications}.
\newblock Cambridge University Press, Cambridge, 2000.
\newblock Abstract parabolic systems.

\bibitem{PrussWeight2}
J.~LeCrone, J.~Pr\"{u}ss, and M.~Wilke.
\newblock On quasilinear parabolic evolution equations in weighted
  {$L_p$}-spaces {II}.
\newblock {\em J. Evol. Equ.}, 14(3):509--533, 2014.

\bibitem{Leslie}
F.~M. Leslie.
\newblock Some constitutive equations for liquid crystals.
\newblock {\em Arch. Rational Mech. Anal.}, 28(4):265--283, 1968.

\bibitem{Lin_Liu}
F.-H. Lin and C.~Liu.
\newblock Nonparabolic dissipative systems modeling the flow of liquid
  crystals.
\newblock {\em Comm. Pure Appl. Math.}, 48(5):501--537, 1995.

\bibitem{LL21}
N.~Lindemulder and E.~Lorist.
\newblock A discrete framework for the interpolation of {B}anach spaces.
\newblock {\em arXiv preprint arXiv:2105.08373}, 2021.

\bibitem{Liu_2020}
J.~Liu, J.~Huang, and A.~Chen.
\newblock Semigroup generations of unbounded block operator matrices based on
  the space decomposition.
\newblock {\em Oper. Matrices}, 14(2):295--304, 2020.

\bibitem{Seel_2016}
K.A. Makarov, S.~Schmitz, and A.~Seelmann.
\newblock On invariant graph subspaces.
\newblock {\em Integral Equations Operator Theory}, 85(3):399--425, 2016.

\bibitem{MY_1990}
A.~McIntosh and A.~Yagi.
\newblock Operators of type {$\omega$} without a bounded {$H_\infty$}
  functional calculus.
\newblock In {\em Miniconference on {O}perators in {A}nalysis ({S}ydney,
  1989)}, volume~24 of {\em Proc. Centre Math. Anal. Austral. Nat. Univ.},
  pages 159--172. Austral. Nat. Univ., Canberra, 1990.

\bibitem{Mit_2010}
M.~Mitrea and S.~Monniaux.
\newblock Maximal regularity for the {L}am\'{e} system in certain classes of
  non-smooth domains.
\newblock {\em J. Evol. Equ.}, 10(4):811--833, 2010.

\bibitem{Moe_2008}
M.~M\"{o}ller and F.~H. Szafraniec.
\newblock Adjoints and formal adjoints of matrices of unbounded operators.
\newblock {\em Proc. Amer. Math. Soc.}, 136(6):2165--2176, 2008.

\bibitem{Mug_2006}
D.~Mugnolo.
\newblock Matrix methods for wave equations.
\newblock {\em Math. Z.}, 253(4):667--680, 2006.

\bibitem{Nag89}
R.~Nagel.
\newblock Towards a ``matrix theory'' for unbounded operator matrices.
\newblock {\em Math. Z.}, 201(1):57--68, 1989.

\bibitem{Nag97}
R.~Nagel.
\newblock Characteristic equations for the spectrum of generators.
\newblock {\em Ann. Scuola Norm. Sup. Pisa Cl. Sci. (4)}, 24(4):703--717
  (1998), 1997.

\bibitem{Nag_1985}
Rainer Nagel.
\newblock Well-posedness and positivity for systems of linear evolution
  equations.
\newblock {\em Confer. Sem. Mat. Univ. Bari}, (203):29, 1985.

\bibitem{NVW11eq}
J.M.A.M.~van Neerven, M.C. Veraar, and L.~Weis.
\newblock Maximal {$L^p$}-regularity for stochastic evolution equations.
\newblock {\em SIAM J. Math. Anal.}, 44(3):1372--1414, 2012.

\bibitem{MaximalLpregularity}
J.M.A.M.~van Neerven, M.C. Veraar, and L.~Weis.
\newblock Stochastic maximal {$L^p$}-regularity.
\newblock {\em Ann. Probab.}, 40(2):788--812, 2012.

\bibitem{PruSim04}
J.~Pr{\"u}ss and G.~Simonett.
\newblock {Maximal regularity for evolution equations in weighted
  $L_p$-spaces}.
\newblock {\em Archiv der Mathematik}, 82(5):415--431, 2004.

\bibitem{pruss2016moving}
J.~Pr\"{u}ss and G.~Simonett.
\newblock {\em Moving interfaces and quasilinear parabolic evolution
  equations}, volume 105 of {\em Monographs in Mathematics}.
\newblock Birkh\"{a}user/Springer, 2016.

\bibitem{CriticalQuasilinear}
J.~Pr\"{u}ss, G.~Simonett, and M.~Wilke.
\newblock Critical spaces for quasilinear parabolic evolution equations and
  applications.
\newblock {\em J. Differential Equations}, 264(3):2028--2074, 2018.

\bibitem{Schmitz_2015}
S.~Schmitz.
\newblock Representation theorems for indefinite quadratic forms without
  spectral gap.
\newblock {\em Integral Equations Operator Theory}, 83(1):73--94, 2015.

\bibitem{Schnau_2010}
R.~Schnaubelt and M.~Veraar.
\newblock Structurally damped plate and wave equations with random point force
  in arbitrary space dimensions.
\newblock {\em Differential Integral Equations}, 23(9-10):957--988, 2010.

\bibitem{Se}
R.~Seeley.
\newblock Interpolation in {$L^{p}$} with boundary conditions.
\newblock {\em Studia Math.}, 44:47--60, 1972.

\bibitem{Seel_2016_2}
A.~Seelmann.
\newblock Notes on the subspace perturbation problem for off-diagonal
  perturbations.
\newblock {\em Proc. Amer. Math. Soc.}, 144(9):3825--3832, 2016.

\bibitem{Shk_2016}
A.A. Shkalikov and C.~Trunk.
\newblock On stability of closedness and self-adjointness for {$2\times 2$}
  operator matrices.
\newblock {\em Mat. Zametki}, 100(6):932--938, 2016.

\bibitem{Sum_2018}
Suma'inna.
\newblock The existence of {$\mathcal{R}$}-bounded solution operators of the
  thermoelastic plate equation with {D}irichlet boundary conditions.
\newblock {\em Math. Methods Appl. Sci.}, 41(4):1578--1599, 2018.

\bibitem{Shib_2019}
Suma'inna, H.~Saito, and Y.~Shibata.
\newblock On some nonlinear problem for the thermoplate equations.
\newblock {\em Evol. Equ. Control Theory}, 8(4):755--784, 2019.

\bibitem{TV88}
A.~Tabacco~Vignati and M.~Vignati.
\newblock Spectral theory and complex interpolation.
\newblock {\em J. Funct. Anal.}, 80(2):383--397, 1988.

\bibitem{Tem1969_1}
R.~Temam.
\newblock Sur l'approximation de la solution des \'{e}quations de
  {N}avier-{S}tokes par la m\'{e}thode des pas fractionnaires. {I}.
\newblock {\em Arch. Rational Mech. Anal.}, 32:135--153, 1969.

\bibitem{Tem1969_2}
R.~Temam.
\newblock Sur l'approximation de la solution des \'{e}quations de
  {N}avier-{S}tokes par la m\'{e}thode des pas fractionnaires. {II}.
\newblock {\em Arch. Rational Mech. Anal.}, 33:377--385, 1969.

\bibitem{Tem_book}
R.~Temam.
\newblock {\em Navier-{S}tokes equations}.
\newblock AMS Chelsea Publishing, Providence, RI, 2001.
\newblock Theory and numerical analysis, Reprint of the 1984 edition.

\bibitem{Ter2018}
Y.~Teramoto.
\newblock Stability of bifurcating stationary solutions of the artificial
  compressible system.
\newblock {\em J. Math. Fluid Mech.}, 20(3):1213--1228, 2018.

\bibitem{Tretter}
C.~Tretter.
\newblock {\em Spectral theory of block operator matrices and applications}.
\newblock Imperial College Press, London, 2008.

\bibitem{Tr1}
H.~Triebel.
\newblock {\em Interpolation theory, function spaces, differential operators}.
\newblock Johann Ambrosius Barth, Heidelberg, second edition, 1995.

\bibitem{S74}
I.~Ja. \v{S}ne\u{\i}berg.
\newblock Spectral properties of linear operators in interpolation families of
  {B}anach spaces.
\newblock {\em Mat. Issled.}, 9(2(32)):214--229, 254--255, 1974.

\bibitem{We}
L.~Weis.
\newblock Operator-valued {F}ourier multiplier theorems and maximal {$L\sb
  p$}-regularity.
\newblock {\em Math. Ann.}, 319(4):735--758, 2001.

\bibitem{Wrona_PhD}
M.~Wrona.
\newblock {\em Liquid Crystals and the Primitive Equations: An Approach by
  Maximal Regularity}.
\newblock PhD thesis, Technische Universit{\"a}t, Darmstadt, 2020.

\bibitem{Zab_1978}
J.~Zabczyk.
\newblock On decomposition of generators.
\newblock {\em SIAM J. Control Optim.}, 16(4):523--534, 1978.

\bibitem{Zab_1980}
J.~Zabczyk.
\newblock Erratum: ``{O}n decomposition of generators'' [{SIAM} {J}. {C}ontrol
  {O}ptim. {\bf 16} (1978), no. 4, 523--534; {MR} {\bf 58} \#23757].
\newblock {\em SIAM J. Control Optim.}, 18(3):325, 1980.

\end{thebibliography}

\end{document}